\documentclass[12pt,a4paper, parskip=half]{article} 
\usepackage[T1]{fontenc}
\usepackage{mathptmx} 
\usepackage[scaled]{helvet}
\usepackage{courier}
\usepackage[active]{srcltx}
\usepackage[utf8]{inputenc}
\usepackage[english]{babel}
\usepackage{lmodern}
\usepackage{amsfonts}
\usepackage{amsmath}
\usepackage{graphicx,epsfig}
\usepackage{amssymb}
\usepackage{theorem,subfigure}
\usepackage[usenames,dvipsnames]{color} 
\usepackage{colortbl}
\definecolor{gray73}{RGB}{186,186,186}
\usepackage{graphicx}
\usepackage{array}
\usepackage{tikz}
\usetikzlibrary{matrix,arrows}
\usepackage{verbatim}
\usepackage{geometry}
\usepackage{multirow}
\usepackage{algorithmicx}
\usepackage{algpseudocodex}

\usepackage{todonotes}

\usepackage{capt-of}
\usepackage{caption}

\DeclareCaptionLabelFormat{mylabelformat}{#1 #2:}
\usepackage{hyperref} 
\usepackage[capitalise]{cleveref}
\newcommand\mycom[2]{\genfrac{}{}{0pt}{}{#1}{#2}}
\newtheorem{theorem}{Theorem}[section]
\newtheorem{lemma}{Lemma}[section]
\newtheorem{corollary}{Corollary}[section]
\newtheorem{remark}[theorem]{Remark}

\newtheorem{algorithm}[theorem]{Algorithm}

\newenvironment{proof}[1][Proof]{\noindent\textbf{#1:} }{\ \rule{0.5em}{0.5em}}

\numberwithin{equation}{section}

\newcommand{\xx}{\mathbf x}
\newcommand{\yy}{\mathbf y}

\newcommand{\elll}{\boldsymbol \ell}
\newcommand{\nuu}{\boldsymbol \nu}
\newcommand{\YY}{\mathbf Y}

\newcommand{\WW}{\mathbf W}
\newcommand{\mm}{\mathbf m}
\newcommand{\cc}{\mathbf c}
\newcommand{\nn}{\mathbf n}
\newcommand{\rr}{\mathbf r}

\newcommand{\zz}{\mathbf z}

\newcommand{\sign}{\mathrm{sign}}

\newcommand{\diag}{\mathrm{diag}}
\renewcommand{\vec}{\mathrm{vec}}

\newcommand{\argmin}{\mathop{\mathrm{argmin}}}
\renewcommand{\mod}{\scriptsize{{\mathrm{mod}}}\,}

 at12pt
 at9pt

\begin{document}
\title{\LARGE MOCCA: A Fast Algorithm for Parallel MRI Reconstruction Using Model Based Coil Calibration}
\author {Gerlind Plonka\footnote{University of G\"ottingen, Institute for Numerical and Applied Mathematics, Lotzestr.\ 16-18,  37083 G\"ottingen, Germany. Email: (plonka,y.riebe)@math.uni-goettingen.de} \qquad
Yannick Riebe$^*$
}

\maketitle

\abstract{
\textbf{Abstract.} We propose  a new fast algorithm  for simultaneous recovery of the coil sensitivities and of the magnetization image from incomplete Fourier measurements in parallel MRI. Our approach is based on a parameter model for  the coil sensitivities using bivariate trigonometric polynomials of small degree.  The derived MOCCA algorithm has low computational complexity of ${\mathcal O}(N_c N^2 \log N)$ for $N \times N$ images and $N_c$ coils and achieves very good performance for incomplete MRI data. We  present  a complete mathematical analysis of the proposed reconstruction method. Further, we show that MOCCA achieves similarly good reconstruction results as ESPIRiT with a considerably smaller numerical effort which is due to the employed parameter model. Our numerical examples show that MOCCA can outperform several other reconstruction methods.
}
\smallskip


\section{Introduction}
One of the biggest innovations in magnetic resonance imaging (MRI)  within the last years  is the concept of parallel MRI.
In this setting,  the use of multiple receiver  coils  allows  the reconstruction  of high-resolution  images from undersampled Fourier data such that the acquisition time  can be substantially reduced.
Assume we have $N_c$ receiver channels and (incomplete) discrete measurements, so-called $k$-space data, on a cartesian grid  
given in the form 
\begin{equation}\label{meas}
\textstyle y_{\nuu}^{(j)} := y^{(j)}(\frac{\nuu}{N}) =  \displaystyle \int\limits _{\Omega} s^{(j)}(\xx) \, m(\xx) \, {\mathrm e}^{-2 \pi {\mathrm i} \frac{\nuu}{N} \cdot \xx} d\xx + \textstyle n^{(j)}(\frac{\nuu}{N}), \quad j=0, \ldots , N_{c}-1,
\end{equation}
for $\nuu \in \Lambda_N:=\{ -\frac{N}{2}, \ldots , \frac{N}{2}-1\} \times \{ -\frac{N}{2}, \ldots , \frac{N}{2}-1\}$, with $N \in 2 {\mathbb N}$, in the $2D$-case,
and with $\xx=(x_{1}, x_{2})^{T}$.
Here, $m$ denotes the complex-valued unknown magnetization image and  $s^{(j)}$ are the  complex valued sensitivity profiles  of the $N_{c}$ individual coils. The  signal is disturbed by noise $n^{(j)}(\frac{\nuu}{N})$. Further, $\Omega \subset {\mathbb R}^{2}$ is the bounded area of interest. We assume for simplicity that $\Omega$ is a square centered  around zero. 
To achieve the wanted acceleration of the acquisition time, the goal is to 
 reconstruct the high-resolution magnetization image $m$  from a subsampled amount of data $y_{\nuu}^{(j)}$, $\nuu \in \Lambda_{\mathcal P}  \subset \Lambda_N$, thereby exploiting the information from the parallel receiver channels. 

\noindent
Unfortunately, in general, the coil sensitivity functions $s^{(j)}$  are also not a priori known and have to be  recovered from the measured data. Indeed,   the acquisition process produces unpredictable correlations between the magnetization image and the coil sensitivities on the one hand and correlations between different coil sensitivities on the other hand such that
 $s^{(j)}$  cannot be accurately estimated beforehand.
Therefore, the parallel MRI reconstruction problem can be seen as a multi-channel blind deconvolution problem.

\subsection*{Contributions} In this paper, we present a new \textbf{MO}del-based \textbf{C}oil {\textbf{CA}libration \!(MOCCA) algorithm to reconstruct the coil sensitivities $s^{(j)}$ and the magnetization image $m$ from the given (incomplete) measurements. 
Our new method employs the assumption that the coil sensitivities $s^{(j)}$ are smooth functions which can be represented using {\em bivariate trigonometric polynomials of small degree}  such that  all $s^{(j)}$ are already  determined by a small number of parameters. In other words, $s^{(j)}$ are assumed to have small support in $k$-space.
Therefore,  our MOCCA algorithm has  low computational complexity of  ${\mathcal O}(N_c N^2 \log N)$ to recover both, the sensitivities as well as the discrete $N \times N$ magnetization image $m$.
If the employed discrete models for $s^{(j)}$  and $m$ (see (\ref{discmodel})-(\ref{sjk}) or more generally (\ref{models1}) - (\ref{general})) are exactly satisfied, we can show that the MOCCA algorithm  reconstructs $s^{(j)}$ and $m$ exactly (up to one unavoidable ambiguity factor). \\
The parameter vectors determining the coil sensitivities $s^{(j)}$ can be simultaneously computed for all coils by finding the nullspace vector ${\mathbf c}$ of the so-called ''MOCCA-matrix'' (see (\ref{AM})), which is constructed from the given $k$-space data in the autocalibration signal (ACS) region. If the data satisfy the model, we show  that  the nullspace of our MOCCA-matrix is of dimension one almost surely, i.e., ${\mathbf c}$ is uniquely defined up to one normalization factor.\\
Our approach is conceptionally  different from ESPIRiT \cite{ESPIRIT} and from all other subspace methods, 
where the sensitivities are computed based on a low-rank assumption of structured $k$-space data.  
This low-rank assumption is not used in our method, instead, we assume that the sensitivities can be well represented by bivariate trigonometric polynomials of small degree.

Starting with the model (\ref{sjk}), we propose a completely different procedure to compute the coil sensitivities $s^{(j)}$. The low complexity of our MOCCA algorithm stems from the fact that all parameters needed to determine the sensitivities can be recovered by computing just one eigenvector of a moderately sized MOCCA-matrix.
By contrast, in  ESPIRiT \cite{ESPIRIT} and PISCO \cite{Lobos23},
the low-rank assumption for a structured matrix of $k$-space data leads to $N^2$ separated eigenvalue problems to compute the  coil sensitivity vectors $(s_{\nn}^{(j)})_{j=0}^{N_c-1}$ at each pixel ${\nn}$. In those methods,  even after applying the so-called sum-of-squares (sos) condition $\|(s_{\nn}^{(j)})_{j=0}^{N_c-1}\|_2=1$,  the additional problem arises that one needs to find suitable phase factors for each eigenvalue problem to ensure global smoothness of the phase. 
This problem is entirely avoided in the MOCCA algorithm.

We provide a complete  mathematical analysis of the proposed MOCCA reconstruction algorithm.
In the upcoming  paper \cite{KP24}, we will give deeper insights into the  relation between the model-based reconstruction of coil sensitivities based on trigonometric polynomials in MOCCA and the assumption that the $k$-space 
data satisfy a structured low-rank model as applied  e.g.\  in 
ESPIRiT \cite{ESPIRIT}, SAKE \cite{SAKE} or PISCO \cite{Lobos23}. 
This relation shows that  our model assumption is equally suitable.
At the same time, the new model allows a significantly faster computation of the sensitivities from the ACS region  by the MOCCA algorithm since it requires neither
the high effort of computing the coil sensitivities pixel-wise from many eigenvalue problems 
as in \cite{ESPIRIT,SAKE,Lobos23} nor has to deal with all additionally arising problems regarding 
the phase of sensitivity maps.\\
 While we will focus here on modeling the coil sensitivities using bivariate trigonometric polynomials, the approach can 
 be generalized to other expansions into smooth functions as e.g.\  linear combination of Gaussians, 
 algebraic polynomials or of splines. Moreover, the model can be directly incorporated into a variational 
 minimization approach such that other constraints on  $m$ and $s^{(j)}$ can be built in.

\subsection*{Related Work} Parallel magnetic resonance image reconstruction has a long history.
One of the most well-known methods is SENSE \cite{Pruess}. While in \cite{Pruess} the sensitivity functions are estimated beforehand, the SENSE-approach is flexible regarding the modeling of sensitivities.\\
A series of methods in parallel MRI  starts directly from a discretized setting as in  (\ref{discmodel}), where the given (incomplete) data are thought to be obtained by a discrete Fourier transform of the product of corresponding discrete function values of $m$ and $s^{(j)}$, see e.g.\ \cite{grappa,spirit,SM97,JGES98,HGHJ01,ESPIRIT,framelets}. 
Corresponding reconstruction algorithms are frequently 
based on local approximation of  unacquired $k$-space data,
see e.g.\ \cite{grappa, spirit}, or on 
subspace methods, see in particular \cite{Zhang,SAKE,ESPIRIT,she,LORAKS,PLORAKS,ALOHA,Hu23,Lobos23}.
 Often, a  prediction method for local approximation of 
 unacquired data $y^{(j)}_{\nuu}$ in $k$-space is used, which can be interpreted as a local interpolation scheme. This idea had been proposed  in GRAPPA \cite{grappa} and its forerunners SMASH \cite{SM97,JGES98}  and was developed further in SPIRiT \cite{spirit}. The main idea is to compute interpolation weights  from fully sampled  data in the  ACS region  of the $k$-space, which in turn are applied to recover the unacquired $k$-space samples. In SPIRiT \cite{spirit}, a self-consistency condition has been introduced to improve this concept.
Subspace methods rely on the assumption that a blockwise Hankel matrix constructed from $k$-space data has (approximately) low rank.
One of the best performing subspace methods is ESPIRiT \cite{ESPIRIT}, which is currently widely  used.
A closely related approach in \cite{Lobos23} proposes a nullspace algorithm together with several ideas to accelerate the computational effort.
 Low rank matrix completion approaches, 
 as \cite{SAKE}, however tend to be computationally expensive.
 
There have been already earlier attempts to apply model-based techniques  to recover the sensitivities for parallel MRI reconstruction, see e.g. \cite{Morrison, JSENSE, Sheng09}, where it is assumed that the sensitivities can be written as bivariate algebraic polynomials.
These ideas strongly differ from our approach regarding the sensitivity model as well as regarding the reconstruction method. 
 In \cite{JSENSE, Sheng09},
 the polynomial coefficients determining the coil sensitivities are computed  using a time-consuming optimization method based on alternating minimization.

Besides the above mentioned models, we shortly refer to two further important classes of reconstruction models which have been studied extensively.
The introduction of compressed sensing  algorithms led to MRI reconstructions from incomplete data via convex optimization, \cite{Lustig}.
Furthermore, the better understanding that the  coil sensitivities need to be computed or at least adjusted during the reconstruction process, led to essential improvements based on more sophisticated optimization models for MR imaging, see e.g. \cite{Block,uecker08,Ramani,Chen,Allison,Muck,uecker17,ENLIVE,Fessler,framelets}. 
These models employ suitable constraints on the sensitivity functions $s^{(j)}$ or the magnetization image $m$.
Furthermore, other sampling grids  as e.g. radial sparse MRI \cite{Feng}, or spirals \cite{Sheng09} can be  incorporated directly into some of these  optimization approaches.

Finally, within recent years, deep learning methods conquered the research area, see e.g.\ \cite{Hammernik,Jacob,Knoll,Wang,IMJENSE,SPICER}. In particular, in \cite{IMJENSE}, the neural network is based on modeling the sensitivities as algebraic polynomials.
While achieving very good reconstruction results, these methods,  however, still need to be better understood in order to avoid artificial structures which are not due to the data but to the training process, see e.g.\ \cite{Muckley20}, Figure 6.

\subsection*{Organization of the Paper}
In Section 2, we  
introduce the discrete model for the $k$-space data on a cartesian grid and the bivariate trigonometric polynomial model for the coil sensitivities $s^{(j)}$. 
We also present a generalized  model that allows to incorporate the sum-of-squares (sos) condition, which is usually taken to handle the occurring ambiguities.
In Section 3, we introduce the new MOCCA algorithm (see Algorithm \ref{alg2}) for parallel MRI reconstruction from incomplete $k$-space data.
The algorithm consists of two steps. In the first step, the model parameters for the coil sensitivities $s^{(j)}$ are recovered from the $k$-space data in the  ACS region by computing  the nullspace vector of the MOCCA matrix.
In the second step,  the magnetization image $\mm$ 
is computed by solving a least-squares problem. 
We propose an iterative algorithm (see Algorithm  \ref{alg3}) which provides a solution with minimal $2$-norm. 
If the incomplete acquired data follow special patterns, we show that the large equation system arising from the least-squares problem to recover $\mm$ falls apart into  many small systems, such that the computational effort can be strongly reduced by Algorithm \ref{alg4}.
The complete MOCCA algorithm has a complexity of ${\mathcal O}(N_c N^2 \log N)$.
To improve the MOCCA reconstruction result, we propose  a  nonlinear local smoothing scheme as a post-processing step.
In Section 4 we analyze  the new MOCCA reconstruction algorithms. 
We present uniqueness results for the case when the data exactly satisfy the model, show the convergence of the Algorithm \ref{alg3}, derive Algorithm \ref{alg4}, and  provide conditions that guarantee unique solvability of the least-squares problem in the case of subsampled  $k$-space data.
In Section 5, we present some numerical experiments for MRI data. 
In particular, we compare the reconstruction results of the MOCCA algorithm with several other reconstruction methods.
Conclusions are given in Section 6.

\section{Modeling of  Coil Sensitivities and the Discrete Problem in Parallel MRI}}
\label{sec1}

Throughout the paper, we restrict ourselves to the 2D-case, while all ideas can be directly transferred to 3D.
We assume for simplicity that 
the measurement data are given on a cartesian grid. More exactly, assume that  we have given the data as in (\ref{meas})
for $\nuu \in \Lambda_{N}$ (or for $\nuu$ from a subset $\Lambda_{\mathcal P}$ of $\Lambda_{N}$).
In other words, the \textbf{sampling grid} $\frac{1}{N}\Lambda_{N}$ is a cartesian grid of equidistant points in $[-\frac{1}{2}, \frac{1}{2})^2 $ with cardinality $N^2$.
 
We apply the notation 
${\yy}^{(j)} := (y^{(j)}_{\nuu})_{\nuu \in \Lambda_{N}}$ and  ${\mathbf s}^{(j)} := (s_{\nn}^{(j)})_{\nn \in \Lambda_{N}} = (s^{(j)}(\nn))_{\nn \in \Lambda_{N}}$ for $j=0, \ldots , N_c-1$, and  $\mm :=(m_{\nn})_{\nn \in \Lambda_{N}}= (m(\nn))_{\nn \in \Lambda_{N}}$.
Here and throughout the paper,  the index set $\Lambda_N$ corresponding to an $(N \times N)$-image is identified  with a one-dimensional index set of length $N^2$ for the vectorized image, if this is more appropriate, which is easily recognized from the context.
Discretization of the integral in (\ref{meas}) then yields
\begin{equation}\label{discmodel}
\textstyle {\yy}^{(j)} = {\mathcal F} \, (\mm \circ {\mathbf s}^{(j)}), \qquad j=0, \ldots , N_c-1.
\end{equation}
Here $\mm\, \circ \,{\mathbf s}^{(j)} := (m_{\nn} s_{\nn}^{(j)})_{\nn \in \Lambda_N}$ denotes the pointwise product  
and 
 ${\mathcal F}$ 
 denotes the 
 Fourier operator corresponding to the $2D$-Fourier  transform of $(N \times N)$-matrices. For vectorized images 
 $\mm\, \circ \,{\mathbf s}^{(j)} := \diag ({\mathbf s}^{(j)}) \, \mm = \diag(\mm) \,  {\mathbf s}^{(j)}$,
 ${\mathcal F}={\mathcal F}_{N^{2}} =(\omega_N^{\nuu \cdot \nn})_{\nuu,\nn \in \Lambda_N}$ is  of size $N^{2} \times N^{2}$  with $\omega_N:= {\mathrm e}^{-2\pi {\mathrm i}/N}$ and 
$\omega_N^{\nuu \cdot \nn}= \omega_N^{\nu_1n_1+\nu_2n_2}$ for $\nuu=(\nu_1,\nu_2)$ and $\nn=(n_1,n_2)$, and  can be presented as the Kronecker product of the Fourier matrices $(\omega_N^{k \ell})_{k,\ell=-\frac{N}{2}}^{\frac{N}{2}-1}$.
Model (\ref{discmodel}) is usually the starting point
for subspace methods to reconstruct $\mm$, see e.g. \cite{grappa,spirit,SAKE,ESPIRIT}.

\subsection{Model for Coil Sensitivities}
In practice, the coil sensitivities ${s}^{(j)}$ are also unknown for $j=0, \ldots, N_{c}-1$, 
but we can suppose that these functions are smooth. 
The idea is now to model the $s^{(j)}$ such that they can be reconstructed from a small number of parameters.
Therefore,  we propose to present the coil sensitivities  as bivariate trigonometric polynomials, i.e., $s^{(j)}$ have small support in $k$-space. Let $\Lambda_{L} := \{-n, \ldots , n\} \times \{-n, \ldots , n\}$, i.e., $|\Lambda_{L}| =L^{2}$ with $L=2n+1 \ll N$. Now let 
\begin{equation}\label{sjk}
\textstyle s^{(j)}_{\nn} := s^{(j)}(\nn) = \sum\limits_{\rr \in \Lambda_{L}} c_{\rr}^{(j)} \, \omega_{N}^{-\rr \cdot \nn}, \qquad j =0, \ldots , N_c-1, \; \nn \in \Lambda_{N}. 
 \end{equation}
 In our numerical experiments, it has been sufficient to use $L=5$ (or $L=7$), such that every $s^{(j)}$ is already determined by $25$ (or $49$) parameters $c_{\rr}^{(j)}$.
Let $\WW := ({\omega}_N^{-\rr\cdot\nn})_{\nn\in\Lambda_N,\rr\in\Lambda_L}$ denote  a  partial matrix of the inverse (scaled) Fourier matrix $N^2 {\mathcal F}^{-1}$ with only $L^2$ columns corresponding to indices in 
$\Lambda_L$.
Then,  with ${\mathbf c}^{(j)}:= (c_{\rr}^{(j)})_{\rr \in \Lambda_L}$,  we obtain \begin{equation}\label{sj}
 {\mathbf s}^{(j)} := (s^{(j)}_{\nn})_{\nn \in \Lambda_N} = \WW \, \cc^{(j)}, \qquad j=0, \ldots, N_c-1. 
 \end{equation}
 In other words, filling up $(c_{\rr}^{(j)})_{\rr \in \Lambda_L}$ by zeros to get $\tilde{\mathbf c}^{(j)} = (\tilde{c}_{\rr}^{(j)})_{\rr \in \Lambda_N}$ with $\tilde{c}_{\rr}^{(j)}= c_{\rr}^{(j)}$ for $\rr \in \Lambda_L$ and $\tilde{c}_{\rr}^{(j)}=0$ otherwise, we have ${\mathbf s}^{(j)} = N^2 {\mathcal F}^{-1} \tilde{\mathbf c}^{(j)}$.
 \smallskip

 \noindent
\subsection{Acquired $k$-space Measurements}

To accelerate the acquisition time, only a subset of all $k$-space measurements $(y_{\nuu}^{(j)})_{\nuu \in \Lambda_N}$, $j=0, \ldots , N_c-1$,  is acquired. As most of the reconstruction methods, we however require complete $k$-space measurements  within the neighborhood  of the $k$-space center, see Figure \ref{figure0}.
Let $\Lambda_M$  be a centered $M \times M$ index set, i.e.,
$ \textstyle \Lambda_M := \{ -\lfloor \frac{M}{2} \rfloor, \ldots , \lfloor \frac{M-1}{2} \rfloor \} \times \{ -\lfloor \frac{M}{2} \rfloor, \ldots , \lfloor \frac{M-1}{2} \rfloor \} \subset \Lambda_N
$
(with $\lfloor x\rfloor = \max\{n \in {\mathbb Z}: \, n \le x \}$ being the largest integer smaller than of equal to $x$).
We assume that 
 $N \gg M \ge  L$, and that the set of $k$-space measurements $\{y^{(j)}_{\nuu-\rr}: \, \nuu \in \Lambda_M, \,  \rr \in \Lambda_L\}$  is completely acquired
such that the entries  of 
$$\YY_{M,L}^{(j)} := (y_{\nuu-\rr}^{(j)})_{\nuu \in \Lambda_M, \rr \in \Lambda_L}$$ 
are well determined for each $j \in \{0, \ldots , N_c-1\}$.
In other words,  $y^{(j)}_{\nuu}$ are assumed to be given  for  all 
$\nuu \in \Lambda_{M+L-1} = \{  -\lfloor \frac{M}{2} \rfloor-n, \ldots , \lfloor \frac{M-1}{2} \rfloor+n \} \times \{ -\lfloor \frac{M}{2} \rfloor-n, \ldots , \lfloor \frac{M-1}{2}\rfloor +n\}$, 
i.e. $\Lambda_{M+L-1}$ has to be contained in
 the  \textbf{autocalibration signal (ACS) region}, see Figure \ref{figure0}.

\noindent
Beside the measurements from  $\Lambda_{M+L-1}$ we assume that  further measurements $y^{(j)}_{\nuu}$ for $\nuu \in \Lambda_N \setminus \Lambda_{M+L-1}$ are acquired, but these measurements may be incomplete.
Let ${\mathcal P}$ denote the projection operator that provides the acquired measurements for each  coil. Then
$${\mathcal P} \yy^{(j)} = {\mathcal P} \, {\mathcal F} \, (\mm \circ {\mathbf s}^{(j)}) $$
denote the given measurements for $j=0, \ldots , N_c-1$, where zeros are inserted in case of unacquired data.
The subset of $\Lambda_N$ of indices corresponding to acquired  measurements is called $\Lambda_{\mathcal P} \subset \Lambda_N$.

\begin{figure}
\centering
\includegraphics[scale=0.3]{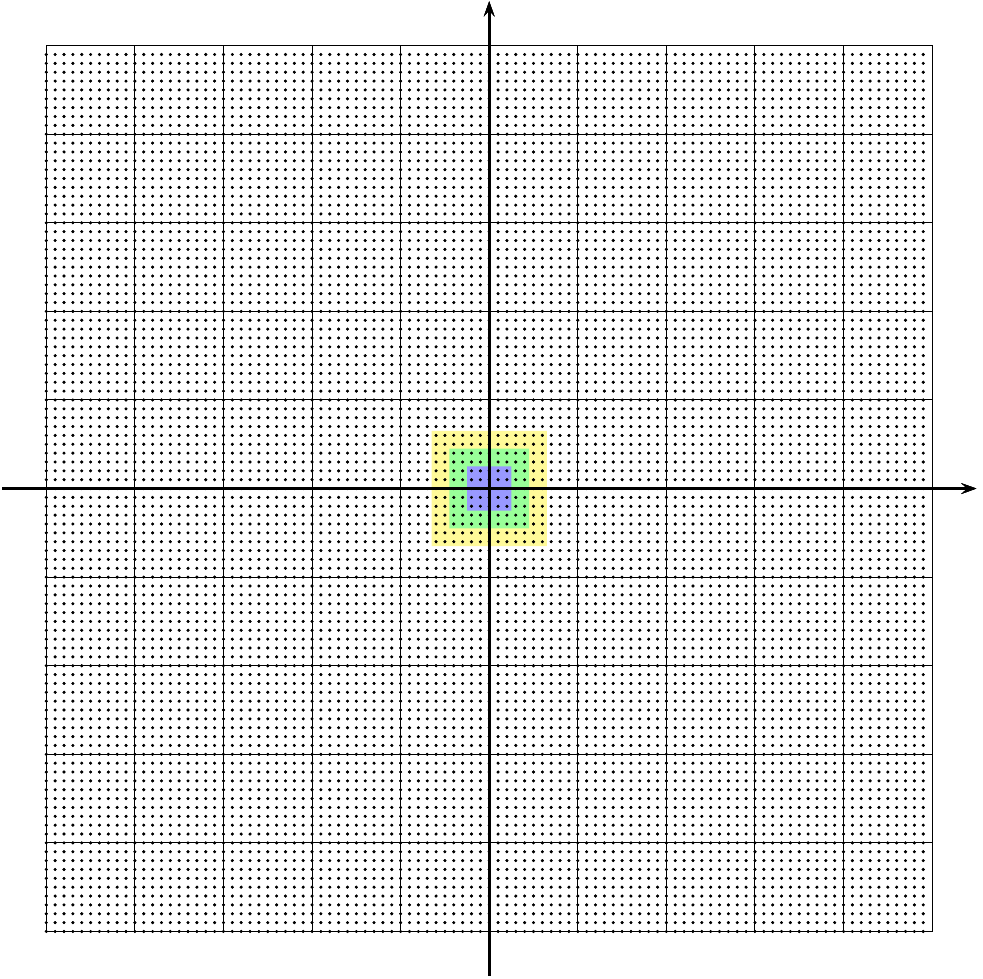}
\caption{Illustration of the grid $\Lambda_N$ (with $N=100$), and the subgrids $\Lambda_L  \subset \Lambda_M \subset \Lambda_{M+L-1}$, the blue subgrid $\Lambda_L$ (with $L=5$), the  green subgrid $\Lambda_M$ (with $M=9$), and the yellow  subgrid  $\Lambda_{M+L-1}$, which needs to be part of the ACS region.}
\label{figure0}
\end{figure}

\medskip

\subsection{Problem Statement} 
Using  the discrete model (\ref{discmodel}) together with (\ref{sjk}), we will solve the following reconstruction problem:\\
For a  given subset  ${\mathcal P} \yy^{(j)}$, $j=0, \ldots, N_c-1$ of $k$-space data,  find  $\mm=(m_{\nn})_{\nn \in \Lambda_N}$ and the $N_{c}$ parameter vectors ${\mathbf c}^{(j)} = (c_{\rr}^{(j)})_{\rr \in \Lambda_L} \in {\mathbb C}^{L^2}$ determining ${\mathbf s}^{(j)}$, $j=0, \ldots , N_c-1$, via (\ref{sjk}).
\smallskip

\subsection{Generalized Model and Ambiguities}
\label{ambi}
Model  (\ref{sjk}) for the coil sensitivities can simply be generalized to
\begin{equation}\label{models1}
 \tilde{s}^{(j)}_{\nn}:=  \gamma_{\nn} \, s^{(j)}_{\nn}  = \gamma_{\nn}  \sum\limits_{\rr \in \Lambda_{L}} c_{\rr}^{(j)} \, \omega_{N}^{-\rr \cdot \nn} \qquad j = 0, \ldots , N_c-1,\; \nn \in \Lambda_N,
 \end{equation}
where $s^{(j)}_{\nn}$ in (\ref{sjk}) is multiplied with the (nonvanishing) sample $\gamma_{\nn}= \gamma(\nn)$ of a smooth window function $\gamma$, and where $\gamma$ does not depend on $j$.
Then, (\ref{discmodel})  can be rewritten as
\begin{align} \label{general} \yy^{(j)} = 
{\mathcal F} 
({\mm}  \circ {\mathbf s}^{(j)}) 
= {\mathcal F} 
(\mm\circ {\boldsymbol\gamma}^{-1} \circ  {\boldsymbol\gamma} \circ {\mathbf s}^{(j)}) 
={\mathcal F} (\tilde{\mm} \circ  \tilde{\mathbf s}^{(j)}) \end{align}
with ${\boldsymbol\gamma}:=(\gamma_{\nn})_{\nn \in \Lambda_N}$, ${\boldsymbol\gamma}^{-1}:=(\gamma_{\nn}^{-1})_{\nn \in \Lambda_N}$,
 $\tilde{\mathbf s}^{(j)} := {\mathbf s}^{(j)} \circ  {\boldsymbol\gamma}$, and $\tilde{\mm} := \mm \circ  {\boldsymbol\gamma}^{-1}$. In other words, having found $\mm$ and ${\mathbf s}^{(j)}$, $j=0, \ldots , N_c-1$, satisfying (\ref{discmodel}) and (\ref{sjk}) for all acquired data, we obtain many further solutions $\tilde{\mm}$ and $\tilde{\mathbf s}^{(j)}$ of (\ref{discmodel}), if we apply the more general model (\ref{models1}). \\
 Analogously as in \cite{Lobos23}, Section II, the factors  $\gamma_{\nn}$ in (\ref{models1})
can be viewed as  ambiguity factors.
In our ''classical'' MOCCA  model (\ref{sjk}), this ambiguity problem is resolved by fixing $\gamma_{\nn}=1$ for $\nn \in \Lambda_N$, such that $s_{\nn}^{(j)}$ are indeed samples of bivariate trigonometric polynomials. Then we are  left with only one global unavoidable ambiguity factor $\mu \in {\mathbb C} \setminus \{0\}$, since for obtained $\mm$ and ${\mathbf s}^{(j)}$ we also find the solution $\frac{1}{\mu}\, \mm$ and $\mu \, {\mathbf s}^{(j)}$, $j=0, \ldots , N_c-1$. 
Applying the generalized model (\ref{models1}), the ambiguities arising from the discretization (\ref{discmodel}) will be  resolved by employing the sum-of-squares (sos) condition, as in most reconstruction algorithms.

\subsection{Sum-of-Squares Condition}
\label{sec:sos}
Many recovery algorithms in parallel MRI (see e.g.\ \cite{grappa,Larsson,ESPIRIT}) employ the following strategy to recover $\mm$ from the (incomplete) acquired data ${\mathcal P}{\mathbf y}^{(j)}$. In a first step, the unacquired $k$-space data $({y}_{\nuu}^{(j)})_{\nuu \in \Lambda_N \setminus \Lambda_{\mathcal P}}$ are approximated from the acquired data $({y}_{\nuu}^{(j)})_{\nuu \in \Lambda_{\mathcal P}}$. Having the complete data  $\yy^{(j)}$ in $k$-space, one applies the inverse Fourier transform $\check{\yy}^{(j)} = {\mathcal F}^{-1} \yy^{(j)}$, $j=0, \ldots , N_c-1$, and computes 
the components $m_{\nn}$ of ${\mm}$ using the sos condition,
\begin{align}\label{sos1} \textstyle m_{\nn} := \big(\sum\limits_{j=0}^{N_c-1} |\check{y}_{\nn}^{(j)}|^2\big)^{\frac{1}{2}}, \qquad \nn \in \Lambda_N. 
\end{align}
At the first glance, this procedure  completely disregards the sensitivity vectors ${\mathbf s}^{(j)}$. We shortly explain, how (\ref{sos1}) relates to our setting using the generalized model (\ref{models1}).
For $\mm=(m_{\nn})_{\nn\in \Lambda_N}$ and ${\mathbf s}^{(j)}= (s^{(j)}_{\nn})_{\nn \in \Lambda_N}$, 
satisfying (\ref{discmodel}) and (\ref{sjk})  for $j=0, \ldots , N_c-1$, we define 
\begin{align} \label{gamma}  \textstyle \gamma_{\nn} :=  \left\{ \begin{array}{ll} 
\alpha_{\nn} \big(\sum\limits_{j=0}^{N_c-1} |s^{(j)}_{\nn}|^2 \big)^{-1/2},  & \nn \in  \Lambda_N,  \, \sum\limits_{j=0}^{N_c-1} |s^{(j)}_{\nn}|^2 >0, \\
0 & \text{otherwise}, \end{array} \right.
\end{align}
with $\alpha_{\nn} := \sign(m_{\nn}) \in {\mathbb C}$ with $\sign(m_{\nn}) := \frac{m_{\nn}}{|m_{\nn}|}$ for $ |m_{\nn}| \neq 0$ and $\sign(m_{\nn}) :=0$ for $ |m_{\nn}| = 0$.
Then, we obtain for $\tilde{\mathbf s}^{(j)} ={\boldsymbol\gamma} \circ {\mathbf s}^{(j)}= (\gamma_{\nn}s_{\nn}^{(j)})_{\nn \in \Lambda_N}$ the relation 
$\sum_{j=0}^{N_c-1} |\tilde{s}_{\nn}^{(j)} |^2 = 1$ for $\gamma_{\nn} \neq 0$, 
and $\sum_{j=0}^{N_c-1} |\tilde{s}_{\nn}^{(j)} |^2 = 0$ for $\gamma_{\nn} =0$, where $\tilde{\mathbf s}^{(j)}$ satisfies the generalized model (\ref{models1}).
According to (\ref{general}), we  set $\tilde{m}_{\nn} = \frac{1}{\gamma_{\nn}} m_{\nn}$ for $\gamma_{\nn}\neq 0$ and $\tilde{m}_{\nn}=0$ for $\gamma_{\nn}= 0$, 
such that $\check{\yy}^{(j)} = {\mathcal F}^{-1} {\mathbf y}^{(j)} = \tilde{\mm} \circ \tilde{\mathbf s}^{(j)}$.  Then we conclude for all $\nn \in \Lambda_N$ with $\gamma_{\nn} \neq 0$
\begin{align*} \textstyle \tilde{m}_{\nn} &= \textstyle \frac{\sign(m_{\nn}) |m_{\nn}|}{\gamma_{\nn}} = \frac{\sign(m_{\nn})}{\alpha_{\nn}} \big(\sum\limits_{j=0}^{N_c-1} |{s}_{\nn}^{(j)} |^2\big)^{1/2} |m_{\nn}|
= |m_{\nn}|  \big(\sum\limits_{j=0}^{N_c-1} |{s}_{\nn}^{(j)} |^2\big)^{1/2} \\
& \textstyle =
 \big(\sum\limits_{j=0}^{N_c-1} |{s}^{(j)}_{\nn} {m}_{\nn}|^2 \big)^{\frac{1}{2}} = \big(\sum\limits_{j=0}^{N_c-1} |\check{y}_{\nn}^{(j)}|^2 \big)^{\frac{1}{2}}.
\end{align*} 
Hence, the sos solution (\ref{sos1}) can be  obtained from our generalized model (\ref{models1}) by choosing $\gamma_{\nn}$ as in (\ref{gamma}).
This observation shows that the application of the sos condition is another way to resolve the ambiguity problem in the discrete parallel MRI setting. Finally, we note that $m_{\nn}$ can only  be determined if $\sum_{j=0}^{N_c-1} |s_{\nn}^{(j)}|^2>0$.
\begin{remark}
In subspace algorithms as \cite{ESPIRIT} and \cite{Lobos23}, where the sensitivities are computed from the $k$-space data in the ACS region in a first step by solving separate eigenvalue problems at every pixel index $\nn$, the sos condition is inherently employed by computing  normalized eigenvectors $(s_{\nn}^{(j)})_{j=0}^{N_c-1}$. However, then still the phase factors for every $\nn \in \Lambda_N$ need to be suitably chosen. This problem is avoided in our approach by taking the model (\ref{sjk}) or (\ref{models1}).
\end{remark}

\section{MOCCA Algorithm for Reconstruction from Incomplete $k$-space Data}
\label{sec:incomplete}

In this section we will derive the MOCCA algorithm that  reconstructs $\mm = (m_{\nn})_{\nn \in \Lambda_N}$ and the sensitivities
${\mathbf s}^{(j)}$ (which are by (\ref{sjk}) already determined by 
${\mathbf c}^{(j)}= (c_{\rr}^{(j)})_{\rr \in \Lambda_L}$)  from the acquired subset ${\mathcal P}{\yy}^{(j)}$, $j=0, \ldots , N_c-1$. As seen before, (\ref{discmodel}) with (\ref{sjk}) requires  to determine only $N^2 + L^2 N_c$ parameters.
Therefore, we should be able to recover these parameters from incomplete $k$-space measurements.
\smallskip

\subsection{MOCCA Algorithm}
\label{sec:Mocca}
Our approach consists of two steps.
 First we recover the sensitivities ${\mathbf s}^{(j)}$ from the acquired $k$-space data in the  (ACS) region using only the discretization model (\ref{discmodel}) and the sensitivity model (\ref{sjk}). That is, we have to recover 
${\mathbf c}^{(j)}$, $j=0, \ldots , N_c-1$, determining  ${\mathbf s}^{(j)}$ via (\ref{sjk}).
In the second step of the algorithm  we will reconstruct $\mm$.
\smallskip

\noindent
\textbf{Step 1: Reconstruction of ${\mathbf c}^{(j)}$.}\\
 First, we  derive a relation between ${\mathbf c}^{(j)}$ and the $k$-space data $\yy^{(j)}$ that we can apply to recover ${\mathbf c}^{(j)}$ only from the ACS  region.

\begin{theorem}\label{theoc}
Assume that  the vectorized coil sensitivities ${\mathbf s}^{(j)} \in {\mathbb C}^{N^2}$ satisfy $(\ref{sjk})$, i.e., ${\mathbf s}^{(j)}= {\mathbf W} {\mathbf c}^{(j)}$ for $j=0, \ldots , N_c-1$, with $\cc^{(j)} \in {\mathbb C}^{L^2}$ as in $(\ref{sj})$.
If  the model $(\ref{discmodel})$ 
is satisfied then 
\begin{align} \label{gross}
\left(\!\!\! \begin{array}{cccc}  
- ( \sum\limits_{\mycom{\ell=0}{\ell\neq 0}}^{N_c-1} {\mathbf Y}_{N,L}^{(\ell)}) & {\mathbf Y}_{N,L}^{(0)} & \ldots & {\mathbf Y}_{N,L}^{(0)} \\
{\mathbf Y}_{N,L}^{(1)} & -( \sum\limits_{\mycom{\ell=0}{\ell\neq 1}}^{N_c-1} {\mathbf Y}_{N,L}^{(\ell)}) & \ldots & {\mathbf Y}_{N,L}^{(1)} \\
\vdots &  &   \ddots & \vdots \\
{\mathbf Y}_{N,L}^{(N_c-1)} &  \ldots &   
{\mathbf Y}_{N,L}^{(N_c-1)} & - (\sum\limits_{\mycom{\ell=0}{\ell\neq N_c-1}}^{N_c-1}{\mathbf Y}_{N,L}^{(\ell)})   \end{array} \!\!\!\right) \left( \!\!\!\begin{array}{c} {\mathbf c}^{(0)} \\ {\mathbf c}^{(1)} \\{\mathbf c}^{(2)} \\ \vdots  \\{\mathbf c}^{(N_c-1)} \end{array} \!\!\!\right) = {\mathbf 0},
\end{align}
where ${\mathbf Y}_{N,L}^{(j)} := ({y}^{(j)} _{(\nuu-\rr) \mod \Lambda_N})_{\nuu \in \Lambda_N, \rr\in \Lambda_L}  \in {\mathbb C}^{N^2 \times L^2}$.
The matrix in $(\ref{gross})$ is called the MOCCA-matrix ${\mathbf A}_N  \in {\mathbb C}^{N^2N_c \times L^2N_c}$, such that $(\ref{gross})$ reads ${\mathbf A}_N {\mathbf c} = {\mathbf 0}$.
\end{theorem}

\begin{proof}
Model (\ref{discmodel})  together with (\ref{sjk}) (resp.\ (\ref{sj})) leads to 
\begin{equation}\label{prod1} \check{\mathbf y}^{(j)} := {\mathcal F}^{-1} {\mathbf y}^{(j)} =  {\mathbf s}^{(j)} \circ \mm =    \mm \circ  ({\mathbf W} {\mathbf c}^{(j)} ) = ( {\mathbf W} {\mathbf c}^{(j)} ) \circ  \mm\end{equation}
for $j=0, \ldots , N_c-1$.  Therefore, we also find
$$
\textstyle \sum\limits_{\mycom{\ell=0}{\ell\neq j}}^{N_c-1}\check{\yy}^{(\ell)}=\sum\limits_{\mycom{\ell=0}{\ell\neq j}}^{N_c-1}  ( {\mathbf W} {\mathbf c}^{(\ell)} ) \circ  \mm =( {\mathbf W} \, \sum\limits_{\mycom{\ell=0}{\ell\neq j}}^{N_c-1} {\mathbf c}^{(\ell)}) \circ \mm$$
and it follows from  (\ref{prod1}) that 
\begin{align}\nonumber
{\check{\yy}}^{(j)} \circ ( {\mathbf W} \, \sum\limits_{\mycom{\ell=0}{\ell\neq j}}^{N_c-1} {\mathbf c}^{(\ell)})
 &=   \textstyle  ({\mathbf W} \, \sum\limits_{\mycom{\ell=0}{\ell\neq j}}^{N_c-1} {\mathbf c}^{(\ell)}) \circ  \check{\yy}^{(j)} 
= ({\mathbf W} \, \sum\limits_{\mycom{\ell=0}{\ell\neq j}}^{N_c-1} {\mathbf c}^{(\ell)}) \circ ({\mathbf W} \, {\cc}^{(j)} )\circ \mm \\
\label{reljj}
&= \textstyle ({\mathbf W} \, \sum\limits_{\mycom{\ell=0}{\ell\neq j}}^{N_c-1} {\mathbf c}^{(\ell)}) \circ \mm \circ ({\mathbf W} {\cc}^{(j)})
= (\sum\limits_{\mycom{\ell=0}{\ell\neq j}}^{N_c-1}\check{\yy}^{(\ell)}) \circ ({\mathbf W} \, {\cc}^{(j)}).
\end{align}
This is equivalent with
$$ - \diag \Big(\sum\limits_{\mycom{\ell=0}{\ell\neq j}}^{N_c-1}\check{\yy}^{(\ell)} \Big)\,  {\mathbf W}  {\cc}^{(j)} +
\diag ({\check{\yy}}^{(j)} ) \, {\mathbf W} \,  \Big(\sum\limits_{\mycom{\ell=0}{\ell\neq j}}^{N_c-1} {\mathbf c}^{(\ell)} \Big) = {\mathbf 0}, \qquad j=0, \ldots, N_c-1. $$
Application of  the $2D$-Fourier matrix  ${\mathcal F}$ to the vectorized images yields
\begin{equation}\label{eigc0} \textstyle 
 [ - {\mathcal F} \, \diag(\sum\limits_{\mycom{\ell=0}{\ell\neq j}}^{N_c-1}\check{\yy}^{(\ell)}) \, {\mathbf W}, \, {\mathcal F} \, \diag({\check{\yy}}^{(j)})\, {\mathbf W}] \, \left( \begin{array}{c} {\cc}^{(j)} \\ \sum\limits_{\mycom{\ell=0}{\ell\neq j}}^{N_c-1} {\mathbf c}^{(\ell)} \end{array} \right) = {\mathbf 0}, \qquad j=0, \ldots , N_c-1.
 \end{equation}
For the ${\rr}$-th column of the matrix ${\mathcal F} \, \diag( \check{\yy}^{(j)}) {\mathbf W} \in {\mathbb C}^{N^2 \times L^2} $ we observe that 
\begin{align} \nonumber \textstyle 
{\mathcal F}\,  \diag( \check{\yy}^{(j)}) (\omega_N^{-\nn\cdot \rr})_{\nn \in \Lambda_N} &=(\omega_N^{\nuu\cdot \nn})_{\nuu, \nn \in \Lambda_N} \,  (\omega_N^{-\nn\cdot \rr} \check{y}^{(j)}_{\nn})_{\nn \in \Lambda_N} \\
\label{projectR}
&= \textstyle \big(\sum\limits_{\nn \in \Lambda_N} \check{y}^{(j)}_{\nn} \, 
\omega_N^{\nn\cdot (\nuu-\rr)} \big)_{\nuu \in \Lambda_N}
= \big({y}^{(j)}_{(\nuu-\rr) \mod \Lambda_N}\big)_{\nuu \in \Lambda_N} .
\end{align}
Hence, the $\nuu$-th row of ${\mathcal F} \, \diag( \check{\yy}^{(j)}) {\mathbf W}$ contains the $L^2$ components $y^{(j)}_{(\nuu-\rr) \mod \Lambda_N}$ for $\rr \in \Lambda_L$, i.e., all components of $\yy^{(j)}$ in the $(L\times L)$-block around pixel $\nuu$ in the matrix representation of $\yy^{(j)}$ (using periodic boundary conditions).
Thus, (\ref{eigc0}) implies for $j=0, \ldots , N_c-1$,
\begin{align}\label{eigc1}
  \Big(  -\big(\sum\limits_{\mycom{\ell=0}{\ell\neq j}}^{N_c-1}y^{(\ell)}_{(\nuu-\rr) \mod \Lambda_N}\big)_{\nuu \in \Lambda_N, \rr\in \Lambda_L}, ({y}^{(j)} _{(\nuu-\rr) \mod \Lambda_N})_{\nuu \in \Lambda_N, \rr\in \Lambda_L}  \Big)
\, \left( \begin{array}{c} {\cc}^{(j)} \\  \sum\limits_{\mycom{\ell=0}{\ell\neq j}}^{N_c-1} {\mathbf c}^{(\ell)} \end{array} \right) = {\mathbf 0}.
\end{align}
Equivalently, with 
${\mathbf Y}_{N,L}^{(j)} := ({y}^{(j)} _{(\nuu-\rr) \mod \Lambda_N})_{\nuu \in \Lambda_N, \rr\in \Lambda_L}  \in {\mathbb C}^{N^2 \times L^2}$ we obtain (\ref{gross}).
\end{proof}

If the nullspace of ${\mathbf A}_N$ is one-dimensional, then $\cc$  is uniquely determined up to a global constant.
Obviously, it is sufficient  to use  a much smaller number of rows of the matrix ${\mathbf A}_N$ above to recover $\cc$. 
Therefore, we just need to  employ the given $k$-space data from the ACS region and consider only  the matrix blocks
${\mathbf Y}_{M,L}^{(j)} := ({y}^{(j)} _{(\nuu-\rr) \mod \Lambda_N})_{\nuu \in \Lambda_M, \rr\in \Lambda_L}  \in {\mathbb C}^{M^2 \times L^2}$.
We construct the MOCCA-matrix ${\mathbf A}_M \in {\mathbb C}^{M^2 N_c \times L^2 N_c}$ similarly as in (\ref{gross}), where the block matrices ${\mathbf Y}_{N,L}^{(j)}$ are replaced by their submatrices ${\mathbf Y}_{M,L}^{(j)}$ with only $M^2$ rows. 
 If the  nullspace of ${\mathbf A}_M$ is still one-dimensional, then  ${\mathbf A}_M \, \cc = {\mathbf 0}$  can be solved, and ${\mathbf c}$ is uniquely determined up to a  global (complex) factor.
We choose $\cc$ with $\|\cc\|_2=1$. 
Then, we find the coil sensitivity vectors ${\mathbf s}^{(j)} = \WW \, \cc^{(j)}$, $j=0, \ldots, {N_c-1}$.
\smallskip

\noindent
\textbf{Step 2: Reconstruction of $\mm$.}\\
Having determined ${\mathbf s}^{(j)}$, we compute $\mm$ using the SENSE approach \cite{Pruess}.
In a first step, we normalize ${\mathbf s}^{(j)}$. 
We apply now the generalized model (\ref{models1}) and compute $\tilde{\mathbf s}^{(j)} = {\boldsymbol\gamma} \circ {\mathbf s}^{(j)}$ with ${\boldsymbol\gamma} = (\gamma_{\nn})_{\nn \in \Lambda_N}$ as in (\ref{gamma}),
where $\alpha_{\nn}:= 1$. We obtain
\begin{align}\label{sjn} 
\tilde{s}_{\nn}^{(j)} :=  \left\{ \begin{array}{ll}\Big( \sum\limits_{\ell=0}^{N_c-1} |s_{\nn}^{(\ell)}|^2\Big)^{-\frac{1}{2}}  s_{\nn}^{(j)}&  
\text{if} \;  \sum\limits_{\ell=0}^{N_c-1} |s_{\nn}^{(\ell)}|^2 > \epsilon, \\
 0  & \text{otherwise,} \end{array}\right.
\end{align}
where theoretically $\epsilon =0$, and numerically $\epsilon$ is a small threshold larger than $0$.
Then, model (\ref{discmodel}) still holds in the form 
$ {\mathbf y}^{(j)} = {\mathcal F} (\tilde{\mathbf s}^{(j)} \circ \tilde{\mathbf m})$ for $j=0, \ldots , N_c-1$  with $\tilde{\mm} := {\boldsymbol \gamma}^{-1} \circ \mm$. 
To reconstruct $\tilde{\mm}$, we solve the  minimization problem
\begin{align}\label{minm} \textstyle  \tilde{\mm} := \argmin\limits_{{\mm} \in {\mathbb C}^{N^2}}  \left( \sum\limits_{j=0}^{N_c-1} \|  {\mathcal P} {\yy}^{(j)} - ({\mathcal P} \, {\mathcal F} \, \diag (\tilde{\mathbf s}^{(j)}))\,  {\mm} \|_2^2 \right),
\end{align}
which leads to the linear system 
\begin{align}\label{lins}  \textstyle  \Big(\sum\limits_{j=0}^{N_c-1} ({\mathbf B}^{(j)})^* \, {\mathbf B}^{(j)} \Big) \, \tilde{\mm} = \sum\limits_{j=0}^{N_c-1}
({\mathbf B}^{(j)})^* \,  {\mathcal P} {\mathbf y}^{(j)} 
\end{align}
with ${\mathbf B}^{(j)} := {\mathcal P} {\mathcal F} \diag( \tilde{\mathbf s}^{(j)})$.
Thus, $\tilde{\mm}$ is uniquely determined if the positive semidefinite coefficient matrix $\sum_{j=0}^{N_c-1} ({\mathbf B}^{(j)})^* \, {\mathbf B}^{(j)}$ is invertible.
The obtained solution set $(\tilde{\mm}, \tilde{\mathbf s}^{(j)})$ with $\tilde{\mathbf s}^{(j)}$ in (\ref{sjn}) and $\tilde{\mm}$ in (\ref{minm}) now refers to the generalized coil sensitivity model (\ref{models1}) (with $\boldsymbol\gamma$ as in (\ref{gamma})  and $\alpha_{\nn}=1$), while the solution set $(\mm, {\mathbf s}^{(j)})$ with $\mm =  {\boldsymbol\gamma} \circ \tilde{\mm}$
refers to (\ref{sjk}).
Replacing now  the  components of $\tilde{\mm}$ by its modulus,  $|\tilde{\mm}| := (|\tilde{m}_{\nn}|)_{\nn \in \Lambda_N}$, is equivalent with replacing the coefficients $\alpha_{\nn}=1$ by $\alpha_{\nn} = \sign(m_{\nn})$ as in (\ref{sec:sos}),
such that the resulting solution $|\tilde{\mm}|$ satisfies the sos condition (\ref{sos1}).
Finally, there is still one ambiguity, which can be used to normalize the image $\tilde{\mm}$.

The complete reconstruction algorithm  to compute $(\tilde{\mm}, \tilde{\mathbf s}^{(j)})$ is summarized in Algorithm \ref{alg2}.

\begin{algorithm}\label{alg2}
[MOCCA algorithm for reconstruction of $\mm$ and ${\mathbf s}^{(j)}$] \\
\textbf{Input:} Incomplete $k$-space data ${\mathcal P} {\yy}^{(j)}$, $j=0, \ldots , N_c-1$,  with 
$({\mathcal P} {\yy}^{(j)})_{\nuu} = y_{\nuu}^{(j)}$
for $\nuu \in \Lambda_{\mathcal P}$ and 
$({\mathcal P} {\yy}^{(j)})_{\nuu} = 0$ for $\nuu \in \Lambda_N \setminus \Lambda_{\mathcal P}$.\newline
$L=2n+1$ determining the support of the trig. polynomials $s^{(j)}$ in $(\ref{sjk})$ in $k$-space.\newline
$\Lambda_{M+L-1} \subset \Lambda_{\mathcal P}  \subset \Lambda_N$,  where $M \ge L$. 
Regularization parameter $\beta \ge 0$. \\
\begin{enumerate}
\item  Build  the matrices ${\mathbf Y}_{M,L}^{(j)} := ({y}^{(j)} _{(\nuu-\rr) \mod \Lambda_N})_{\nuu \in \Lambda_M, \rr\in \Lambda_L}$, $j=0, \ldots , N_c-1$, and 
\begin{align} \label{AM}
{\mathbf A}_M:=\left( \!\!\! \begin{array}{cccc}  
- (\sum\limits_{\mycom{\ell=0}{\ell \neq 0}}^{N_c-1} {\mathbf Y}_{M,L}^{(\ell)}) & {\mathbf Y}_{M,L}^{(0)}  \!\!\!& \ldots & {\mathbf Y}_{M,L}^{(0)} \\
{\mathbf Y}_{M,L}^{(1)} &  \!\!\!-( \sum\limits_{\mycom{\ell=0}{\ell \neq 1}}^{N_c-1} {\mathbf Y}_{M,L}^{(\ell)}) \!\!\! &  \ldots & {\mathbf Y}_{M,L}^{(1)} \\
\vdots &    & \ddots & \vdots \\
{\mathbf Y}_{M,L}^{(N_c-1)} &  \ldots & 
   {\mathbf Y}_{M,L}^{(N_c-1)} & \!\!\! - (\sum\limits_{\mycom{\ell=0}{\ell \neq N_c-1}}^{N_c-1} {\mathbf Y}_{M,L}^{(\ell)})    \end{array} \!\!\!\right).
\end{align}
\item Compute a singular vector ${\mathbf c} \in {\mathbb C}^{N_c L^2}$  of ${\mathbf A}_M$ of norm $1$ corresponding to the smallest singular value of ${\mathbf A}_M$. 
\item Extract the vectors ${\mathbf c}^{(j)}$, from  $\cc= ((\cc^{(0)})^T,  (\cc^{(1)})^T, \ldots , (\cc^{(N_c-1)})^T)^T$. \\
Compute ${\mathbf s}^{(j)} = {\mathbf W} \, \cc^{(j)}$, $j=0, \ldots , N_c-1$ by FFT, where ${\mathbf W} := (\omega_N^{-\nn \cdot \rr})_{\nn \in \Lambda_N, \rr \in \Lambda_L}$.
\item Compute ${\mathbf d} = (d_{\nn})_{\nn \in \Lambda_N}:= \sum_{j=0}^{N_c-1} \overline{{{\mathbf s}^{(j)}}}  \circ {{\mathbf s}^{(j)}} $, where $\circ$ is the pointwise product.
Set ${\mathbf d}^{+} = (d_{\nn}^+)_{\nn\in \Lambda_N}$ with $d_{\nn}^+ := \frac{1}{d_{\nn}}$ for $|d_{\nn}| > \epsilon$ and $d_{\nn}^+ =0$ otherwise.
\item  For $j=0:N_c-1$ compute
$\tilde{\mathbf s}^{(j)} := ({\mathbf d}^+)^{\frac{1}{2}} \circ  {\mathbf s}^{(j)} = 
(\sqrt{d}^+_{\nn} s_{\nn}^{(j)})_{\nn \in \Lambda_N}$.
\item
Solve the equation system 
\begin{align}\label{riesig}  \textstyle  \Big(  \beta \,  {\mathbf I} + \sum\limits_{j=0}^{N_c-1} ({\mathbf B}^{(j)})^* {\mathbf B}^{(j)} \Big) \, \tilde{\mm} = 
\sum\limits_{j=0}^{N_c-1} ({\mathbf B}^{(j)})^* \,  {\mathcal P} \,  {\yy}^{(j)}
\end{align}
with ${\mathbf B}^{(j)} := {\mathcal P} {\mathcal F} \, \diag (\tilde{\mathbf s}^{(j)})$, either iteratively or directly,
to compute $\tilde{\mm}$.
\item Set  $\tilde{\mathbf s}^{(j)} := (\sign(\tilde{m}_{\nn}) \tilde{s}^{(j)}_{\nn})_{\nn \in \Lambda_N}$ for $j=0, \ldots , N_c-1$ and 
 $\tilde{\mm}:=(|\tilde{m}_{\nn}|)_{\nn \in \Lambda_N}$.
 \end{enumerate}
\textbf{Output: } $(\tilde{\mm}$, $\tilde{\mathbf s}^{(j)}),$ 
$j=0, \ldots , N_c-1$ satisfying $(\ref{models1})$, $(\ref{general})$, and $(\ref{sos1})$.
\end{algorithm}

\noindent
\begin{remark}
1. To compute the singular vector to the smallest singular value $\sigma_{\min}$ of ${\mathbf A}_M$
one can employ an inverse power iteration to ${\mathbf A}_M^*{\mathbf A}_M$ which converges fast in our case.\\
2. Beside the MOCCA matrix ${\mathbf A}_M$ in $(\ref{AM})$, there exist many other  matrices that serve the same purpose to 
determine the vectors $\cc = ({\mathbf c}^{(j)})_{j=0}^{N_c-1}$. 
Taking for example
$$ {\mathbf A}_{M,0} := \left( \begin{array}{ccccc}  
-  {\mathbf Y}_{M,L}^{(1)} & {\mathbf Y}_{M,L}^{(0)} & {\mathbf 0}& \ldots &{\mathbf 0} \\
{\mathbf 0} & - {\mathbf Y}_{M,L}^{(2)} & {\mathbf Y}_{M,L}^{(1)}  & \ldots & {\mathbf 0} \\
\vdots &  &  & \ddots & \vdots \\
{\mathbf Y}_{M,L}^{(N_c-1)} &  \ldots & {\mathbf 0}  &  {\mathbf 0} & -  {\mathbf Y}_{M,L}^{(0)}   \end{array} \right),$$
we conclude similarly as in  $(\ref{reljj})$  that ${\mathbf A}_{M,0} {\mathbf c} = {\mathbf 0}$. \\
3. Relations of the type $\check{\mathbf y}^{(j)} \circ {\mathbf s}^{(\ell)} - \check{\mathbf y}^{(\ell)} \circ {\mathbf s}^{{(j)}} = {\mathbf 0}$ for $j, \, \ell \in \{0, \ldots, N_c-1\}$ have been used also in other papers, see e.g.\  $\text{\cite{Harikumar}}$. But without using a suitable model for ${\mathbf s}^{(j)}$, these relations are not sufficient to derive a reconstruction method for ${\mathbf s}^{(j)}$.\\
4. Note that for the special case $\Lambda_{\mathcal P} = \Lambda_N$ (completely acquired $k$-space data), the  coefficient matrix of the system  (\ref{riesig}) simplifies to 
$$ \textstyle \sum\limits_{j=0}^{N_c-1} ({\mathbf B}^{(j)})^* {\mathbf B}^{(j)} \!= \!
\sum\limits_{j=0}^{N_c-1}  \diag(\tilde{\mathbf s}^{(j)})^* {\mathcal F}^*{\mathcal F} \diag(\tilde{\mathbf s}^{(j)}) \!= \!N^2 \sum\limits_{j=0}^{N_c-1}  \diag(\tilde{\mathbf s}^{(j)})^* \diag(\tilde{\mathbf s}^{(j)}) \!= \!N^2 {\mathbf I}_{N^2}, $$
 such that 
$\mm = \frac{1}{N^2}\sum_{j=0}^{N_c-1}  \diag(\tilde{\mathbf s}^{(j)})^* {\mathcal F}^* {\mathbf y}^{(j)} = 
\sum_{j=0}^{N_c-1}  \diag(\tilde{\mathbf s}^{(j)})^* \check{\mathbf y}^{(j)}$. \\
5. To improve the condition number of  the coefficient matrix in (\ref{lins}), 
it can 
be replaced  by the positive definite matrix $\beta \, {\mathbf I} +\sum_{j=0}^{N_c-1} ({\mathbf B}^{(j)})^* \, {\mathbf B}^{(j)} $ with small $\beta >0$ and the  identity matrix ${\mathbf I}$ of size $N^2 \times N^2$. This regularization  is often used  for least square methods, see e.g.\ $\cite{spirit}$. \\
6. In step 2 of the MOCCA algorithm to  reconstruct the magnetization image $\mm$ the 
least squares problem (\ref{minm}) can be simply extended by incorporating a 
 further regularization term for $\mm$ and considering instead  the minimization problem 
\begin{align}\label{mingen}
 \tilde{\mm} = \textstyle \argmin\limits_{{\mm} \in {\mathbb C}^{N^2}}  \left( \sum\limits_{j=0}^{N_c-1} \frac{1}{2} \|  {\mathcal P} {\yy}^{(j)} - ({\mathcal P} \, {\mathcal F} \, \diag (\tilde{\mathbf s}^{(j)})) \, {\mm} \|_2^2  + \lambda \Phi({\mm}) \right).
\end{align}  
Here,  $\Phi(\mm)$ denotes a suitable constraint on the image $\mm$, and $\lambda>0$ is the regularization parameter. As usual in image processing applications, we can for example force that $\mm$ has a sparse representation in a suitably chosen transformed domain, i.e., $\Phi(\mm) := \|{\mathbf T} \mm\|_1$, where ${\mathbf T}$ denotes the transform matrix  and $\| \cdot \|_1$ is the usual 1-norm of a vector. Note that  in  \cite{Keeling}, actually two regularization functionals for $\mm$ have been proposed, including the total variation regularization. In \cite{LiChan16}, a Haar-framelet based transform matrix has been constructed. Both approaches  are computationally rather expensive compared to solving (\ref{minm}).\\
7. Note again that the solution ambiguities presented in Section \ref{ambi} are here resolved by employing the sos condition, i.e., by normalizing the sensitivities in step 5 of Algorithm \ref{alg2}. Furthermore the remaining global ambiguity and the phase ambiguities (see $\alpha_{\nn}$ in (\ref{gamma})) are fixed by choosing $\mm$ with $\|\mm\|_2=1$ and with nonnegative real entries. This choice is taken, since the result is later compared with the ''ground truth" image obtained from (\ref{sos1}).
\end{remark}

\subsection{Complexity of the MOCCA Algorithm}
Algorithm \ref{alg2} requires 
${\mathcal O}(N_c\,  N^2 \log N)$ operations. In step 2, the singular vector corresponding to the smallest singular value of the MOCCA matrix ${\mathbf A}_M$ can be obtained by an SVD of ${\mathbf A}_M$ with at most ${\mathcal O}((N_cM^2)^2 N_cL^2)$ operations (disregarding the block Hankel structure of ${\mathbf A}_M$). We can assume here that  $N_cM^2 L < {\mathcal O}(N \log N)$ such that this effort is less than ${\mathcal O}(N_c\,  N^2 \log N)$. As we will see in the numerical experiments, small numbers $L$ and $M$ can be chosen,  e.g. $L=5$ and $M=20$ such that the needed ACS region is covered by $24$ ACS lines.
The $N_c$ matrix-vector multiplications in step 3 require ${\mathcal O}(N_c\, N^2\, L^2)$ operations.
Step 4 needs ${\mathcal O}(N_c \, N^2)$ operations.  Finally, step 5 is the most expensive step, since it requires to solve a large linear system. However, for regular sampling patterns  ${\mathcal P}$ usually taken in parallel MRI (e.g. each second, third, or fourth row/column is acquired), the structure of the coefficient matrix can be employed to derive fast solvers with ${\mathcal O}(N_c N^2 \log N)$ operations, see Section \ref{sec:incompletefast}.
Furthermore, in Section  \ref{sec:WJR} we  propose an iterative algorithm to solve (\ref{riesig}) efficiently.

\subsection{Efficient Solution of the Least-Squares Problem}

\noindent
Note that the coefficient matrix \linebreak  $\beta\, {\mathbf I} + \sum_{j=0}^{N_c-1} ({\mathbf B}^{(j)})^* \, {\mathbf B}^{(j)} $ (with $\beta \ge 0$) 
in (\ref{riesig}) is of size $N^2 \times N^2$.
Therefore, despite its special structure, it is very large for large $N$.
To solve (\ref{riesig})  efficiently, we propose 
an  iterative algorithm  and a direct algorithm, where the latter one can be applied in case of structured acquired data.
\medskip

\noindent
\subsubsection{Weighted Jacobi-Richardson Iteration}
\label{sec:WJR}

We propose the following simple iteration algorithm to solve (\ref{riesig}).

\begin{algorithm}\label{alg3}
[Weighted Jacobi-Richardson iteration for $(\ref{riesig})$]\newline
\noindent
\textbf{Input:}  ${\mathcal P} \,  {\yy}^{(j)}$ and $\tilde{\mathbf s}^{(j)}$, $j=0, \ldots , N_c-1$,  
from step 2 of Algorithm $\ref{alg2}$, 
$\epsilon >0$, $\beta\ge 0$. 
\begin{enumerate}
\item Set  $\tilde{\mm}_{-1}:= {\mathbf 0} \in {\mathbb C}^{N^2}$, $\yy^{(j)}_0 := {\mathcal P}  {\yy}^{(j)}$ 
and $\kappa:=0$. \newline
Compute 
$\tilde{\mm}_0 :=  \sum\limits_{j=0}^{N_c-1} (\overline{{\tilde{\mathbf s}^{(j)}} } \circ  {\mathcal F}^{-1} \yy^{(j)}_0). $ \\
\item  While $\|\tilde{\mm}_{\kappa} -\tilde{\mm}_{\kappa-1}\|_\infty > \epsilon$ do\newline
Compute 
\begin{align*}
\yy_{\kappa+1}^{(j)} &:= \textstyle  ( (1 - \frac{\beta}{N^2}){\mathbf I} - {\mathcal P}) {\mathcal F} (\tilde{\mathbf s}^{(j)} \circ  \tilde{\mm}_{\kappa}) + \yy^{(j)}_0,  \quad j=0, \ldots , N_c-1, \\
\tilde{\mm}_{\kappa+1} &:=  \sum\limits_{j=0}^{N_c-1} \overline{{\tilde{\mathbf s}^{(j)}} } \circ ( {\mathcal F}^{-1}{\yy}_{\kappa+1}^{(j)}),
\end{align*}
and set $\kappa:=\kappa+1$.
\item Set  $\tilde{\mathbf m}:=\tilde{\mathbf m}_{\kappa}$. 
\end{enumerate}
{\textbf{Output: } $\tilde{\mm}$ solving the system $(\ref{riesig})$. }
\end{algorithm}

\noindent
Each iteration step to compute an improved approximation $\tilde{\mm}_{\kappa}$  in Algorithm \ref{alg3}
requires ${\mathcal O} (N_c N^2 \log N)$ operations for the needed discrete 2D-Fourier transforms.
Thus,  this system can be solved efficiently. In our numerical experiments in Section \ref{sec:num}, we need for example 12 iteration steps for the case if outside the ACS region each second column of data is missing.
We will show in Section \ref{sec:ana2} that  Algorithm \ref{alg3} always converges, even  if   the matrix $\sum_{j=0}^{N_c-1} ({\mathbf B}^{(j)})^* {\mathbf B}^{(j)}$ is  not invertible and $\beta=0$.
Assuming that the model (\ref{discmodel}) is satisfied, the iteration of Algorithm \ref{alg3} then leads to a solution $\tilde{\mm}$ with minimal norm, see Theorem \ref{theoB}.

\begin{remark}
Algorithm $\ref{alg3}$  can be interpreted as  a weighted Jacobi (Richardson) iteration, or  as an alternating  projection  iteration, see Theorem  $\ref{theoB}$. Note that the normalization of the sensitivities is essential for a fast convergence of this iteration scheme.
In Algorithm $\ref{alg3}$, we iteratively compute the vectors $\yy_{\kappa+1}^{(j)}$,  which can be understood as approximations of the $k$-space vectors $\yy^{(j)}$. In other words, the algorithm provides  approximations of the unacquired measurements $y^{(j)}_{\nn}$ for $\nn \in \Lambda_N \setminus \Lambda_{\mathcal P}$. 
 Each further iteration step provides more detail while the image "loses" smoothness.
By contrast to other algorithms for parallel MRI  based on local interpolation of $k$-space data, as e.g.\ GRAPPA \cite{grappa},
these approximations depend on all acquired data  $y^{(j)}_{\nn}$ for $\nn \in \Lambda_{\mathcal P}$, $j=0, \ldots , N_c-1$. 
Of course, Algorithm \ref{alg3} for  $(\ref{riesig})$ can be replaced by another iterative solver, as e.g.\ a  cg-solver. In the MR literature, such iterative solvers are usually called iterative SENSE methods.
 For accelerated convergence in Algorithm \ref{alg3}, one can employ an Alternating Anderson–Richardson method, \cite{Sury}.
\end{remark}

\subsubsection{Direct Fast Algorithm for Structured Incomplete Data}
\label{sec:incompletefast}

\noindent
The computational effort  to recover  $\tilde{\mm}$ in step 3 of Algorithm \ref{alg2} can be strongly reduced for special patterns of acquired data.
Assume for example that $N$ is a multiple of $8$, and $\Lambda_{\mathcal P}$ contains, beside the ACS region $\Lambda_{L+M-1}$, acquired $k$-space data for each fourth column, i.e.,
$\Lambda_{\mathcal P} = \{-\frac{N}{2}, \ldots , \frac{N}{2}-1 \} \times \{-\frac{N}{2}, -\frac{N}{2}+4, \ldots , \frac{N}{2}-8, \frac{N}{2}-4\} \cup \Lambda_{L+M-1}$. In this case, the system (\ref{riesig}) falls apart into $\frac{N}{4}$ systems of size $4 \times 4$ to recover ${\mathbf m}$. This idea can be easily transferred to other structured projection patterns. The algorithm is summarized in Algorithm 
\ref{alg4}.

\begin{algorithm}
[Fast direct solution of (\ref{riesig})]
(if every fourth column of $\yy^{(j)}$ is acquired)\\
\label{alg4}
{\textbf{Input:}  ${\mathcal P} \,  {\yy}^{(j)}$ and $\tilde{\mathbf s}^{(j)}$, $j=0, \ldots , N_c-1$,  from step 2 of Algorithm $\ref{alg2}$, 
$\beta \ge 0$.}
\begin{enumerate}
\item Compute the right-hand side of $(\ref{riesig})$, 
${\mathbf R} = (r_{k,\ell})_{(k,\ell) \in \Lambda_N} :=  \sum\limits_{j=0}^{N_c-1} \overline{\tilde{\mathbf s}^{(j)}} \circ {\mathcal F}^{-1} ({\mathcal P} {\mathbf y}^{(j)})$\\
\quad \ using  2-D FFTs of size $N \times N$.
\item for $k=-\frac{N}{2}:1:\frac{N}{2}-1$ do \newline
\null  \qquad for $\ell=-\frac{N}{8}:1:\frac{N}{8}-1$ do \newline
\null  \qquad \quad  Set $\tilde{\mathbf s}^{(j)}_{k,\ell} := \big(\tilde{s}^{(j)}_{k,\ell- \frac{3N}{8}}, \tilde{s}^{(j)}_{k,\ell- \frac{N}{8}}, 
 \tilde{s}^{(j)}_{k,\ell+ \frac{N}{8}}, \tilde{s}^{(j)}_{k,\ell+ \frac{3N}{8}}\big)^T \in {\mathbb C}^4$ for $j=0, \ldots , N_c-1$. \\
 \null  \qquad \quad  Form the matrix ${\mathbf G}_{k,\ell}:= \frac{4 \beta}{N^2} {\mathbf I}_4+ \sum\limits_{j=0}^{N_c-1}  \overline{\tilde{\mathbf s}^{(j)}_{k,\ell}} \, (\tilde{\mathbf s}^{(j)}_{k,\ell})^T \in {\mathbb C}^{4 \times 4}$. \newline
\null  \qquad \quad  Form the  vector ${\mathbf b}_{k,\ell} := 4 \big( r_{k,\ell-\frac{3N}{8}}, r_{k,\ell-\frac{N}{8}}, r_{k,\ell+\frac{N}{8}}, 
 r_{k,\ell+\frac{3N}{8}} \big)^T \in {\mathbb C}^4$. \newline
 \null  \qquad \quad  Solve  ${\mathbf G}_{k,\ell} \, \mm = {\mathbf b}_{k,\ell}$ and set 
 $\big(\tilde{m}_{k,\ell- \frac{3N}{8}}, \tilde{m}_{k,\ell- \frac{N}{8}}, \tilde{m}_{k,\ell+ \frac{N}{8}}, \tilde{m}_{k,\ell+ \frac{3N}{8}}\big)^T := \mm$. \newline
\null \qquad  end (for)\newline
\null \quad \ end (for)
\end{enumerate}
{\textbf{Output: } $\tilde{\mm}$ solving the system $(\ref{riesig})$. }
\end{algorithm}

\begin{remark} 
1. Having given for example that every second row and every second column is acquired from the $k$-space data ${\mathbf y}^{(j)}$, i.e., $\Lambda_{\mathcal P} \subset \{-\frac{N}{2}, -\frac{N}{2}+2, \ldots , \frac{N}{2}-2 \}^2 \cup \Lambda_{M+L-1}$
(where $N$ is a multiple of $4$), then step 2 of Algorithm \ref{alg4} needs to be replaced by\\
for $k=-\frac{N}{4}:1:\frac{N}{4}-1$ do \newline
\null \quad for $\ell=-\frac{N}{4}:1:\frac{N}{4}-1$ do \newline
\null  \quad \;  Set $\tilde{\mathbf s}^{(j)}_{k,\ell} \!:= \!\big(\tilde{s}^{(j)}_{k- \frac{N}{4},\ell- \frac{N}{4}}, 
 \tilde{s}^{(j)}_{k- \frac{N}{4},\ell+ \frac{N}{4}}, 
 \tilde{s}^{(j)}_{k+ \frac{N}{4},\ell- \frac{N}{4}}, \tilde{s}^{(j)}_{k+ \frac{N}{4},\ell+ \frac{N}{4}}\big)^T\! \in \!{\mathbb C}^4$ for $j=0, \ldots , N_c-1$. \newline
\null  \quad \;  Form the matrix ${\mathbf G}_{k,\ell}:= \frac{4 \beta}{N^2} {\mathbf I}_4+\sum\limits_{j=0}^{N_c-1} {\overline{\tilde{\mathbf s}^{(j)}_{k,\ell}} \, (\tilde{\mathbf s}^{(j)}_{k,\ell})^T} \in {\mathbb C}^{4 \times 4}$. \newline
\null \quad \;  Form the  vector ${\mathbf b}_{k,\ell} := 4 \big( r_{k- \frac{N}{4},\ell- \frac{N}{4}}, r_{k- \frac{N}{4},\ell+ \frac{N}{4}}, r_{k+ \frac{N}{4},\ell- \frac{N}{4}}, 
 r_{k+ \frac{N}{4},\ell+ \frac{N}{4}} \big)^T \in {\mathbb C}^4$. \newline
 \null \quad \;  Solve  ${\mathbf G}_{k,\ell} \, \mm = {\mathbf b}_{k,\ell}$. \newline
  \null \quad \; Set 
 $\big(\tilde{m}_{k- \frac{N}{4},\ell- \frac{N}{4}}, \tilde{m}_{k- \frac{N}{4},\ell+ \frac{N}{4}}, \tilde{m}_{k+ \frac{N}{4},\ell- \frac{N}{4}}, \tilde{m}_{k+\frac{N}{4},\ell+ \frac{N}{4}}\big)^T := \mm$. \newline
\null \quad  end (for)\newline
 \ end (for)
 
\noindent
2. Fast algorithms for special structures of the acquired $k$-space data may be of 
particular interest in the 3D case employing  partially parallel acquisition strategies, see \cite{CAIPI}.

\end{remark}
 
The first step of Algorithm \ref{alg4} requires $N_c$ 2-D FFTs with an effort of ${\mathcal O}(N^2 \log N)$ operations.
The second step requires to solve $\frac{N^2}{4}$ linear systems of size $4 \times 4$, 
with ${\mathcal O}(N^2)$ operations. The theoretical justification of  Algorithm \ref{alg4} is provided in Section \ref{sec:ana3}.

\subsection{Nonlinear Local Smoothing}
\label{smoothing}

As a post-processing step, we use a  local nonlinear smoothing scheme as proposed e.g. in \cite{Plonka09}. To every value $m_{\nn}$, $\nn \in \Lambda_N$,  of the reconstructed magnetization image we apply the following procedure,
\begin{align*}
m_{\nn}^{s} := m_{\nn} + \tau_{\nn} \sum_{\rr \in (\Lambda_1 \setminus \{ {\mathbf 0} \}) \cap \Lambda_N} \frac{g(|m_{\nn}-m_{\nn-\rr}|)}{\|\nn-\rr\|^2}
(m_{\nn-\rr} -m_{\nn}),
\end{align*}
where $\Lambda_1:= \{-1,0,1\} \times \{-1,0,1\}$, $\tau_{\nn} := \big(\sum_{\rr \in (\Lambda_1 \setminus \{ {\mathbf 0} \}) \cap \Lambda_N} \frac{1}{\|\nn-\rr\|^2}\big)^{-1}$, and with
  a function $g$ which is a so-called diffusivity function. In our experiments, we have applied the Perona-Malik diffusivity $g(s)= \frac{1}{1+s^2/\lambda}$ with $\lambda$ as given in Section \ref{sec:num}.
 This procedure  is only one step of the smoothing algorithm in \cite{Plonka09}, which can be interpreted as a discretized nonlinear diffusion smoothing. 

\section{Theoretical Background of Algorithms \ref{alg2}-\ref{alg4}} 
\label{sec:ana}

In this section we show that the MOCCA algorithm almost surely provides a unique reconstruction of the sensitivities ${\mathbf s}^{(j)}$  (up to an unavoidable global ambiguity factor) if the data ${\mathbf y}^{(j)}$ satisfy models (\ref{discmodel})-(\ref{sjk}) or (\ref{models1})-(\ref{general}), i.e., if $s_{\nn}^{(j)}$ can be represented by  trigonometric polynomials of small degree.
In Subsection \ref{sec:ana2} we show that the iterative Algorithm \ref{alg3} always converges and yields a solution image $\tilde{\mm}$ with minimal Frobenius norm, if the coefficient matrix occurring in the least squares system (\ref{riesig}) (with $\beta =0$) is not invertible.
Finally, in Subsection \ref{sec:ana3} we provide a derivation of the fast Algorithm \ref{alg4} and present simple conditions ensuring that the coefficient matrix of the system (\ref{riesig}) (with $\beta=0$) is invertible.

\subsection{Analysis of the MOCCA Algorithm \ref{alg3}}
\label{sec:ana1}
The MOCCA Algorithm \ref{alg2} gives rise to the following theorem.

\begin{theorem}\label{theorec2}
 Let the incomplete $k$-space data ${\mathcal P} \yy^{(j)}$ for   $j=0, \ldots , N_{c}-1$, be given,
where the index set $\Lambda_{\mathcal P}$ of acquired  measurements contains the ACS region, i.e.,  $\Lambda_{M+L-1} \subset \Lambda_{\mathcal P}  \subset \Lambda_N$. Let $(\ref{models1})$, $(\ref{general})$ and 
the sos condition $(\ref{sos1})$ be satisfied for $\yy^{(j)}$,  $\tilde{\mathbf s}^{(j)}$, $j=0, \ldots , N_c-1$, and  $\tilde{\mm}$,  and suppose that:\\
1. The  matrix 
 ${\mathbf A}_M$ in $(\ref{AM})$
has a nullspace of dimension $1$.\\
2. The matrix $\sum_{j=0}^{N_c-1} ({\mathbf B}^{(j)})^* \, {\mathbf B}^{(j)}$ in $(\ref{riesig})$ is invertible.\\
 Then $\tilde{\mm}$  and   $\tilde{\mathbf s}^{(j)}$, $j=0, \ldots , N_c-1$, are  uniquely reconstructed from the $k$-space data ${\mathcal P} \yy^{(j)}$ by Algorithm $\ref{alg2}$. 
\end{theorem} 
\begin{proof} Assumption 1 of Theorem \ref{theorec2} ensures that we find the nullspace vector ${\mathbf c}$ 
and thus the coil sensitivities ${\mathbf s}^{(j)} = {\mathbf W} {\mathbf c}^{(j)}$  in (\ref{sj}) uniquely up to a complex constant in step 2 of Algorithm \ref{alg2}.  As shown in Section \ref{sec:sos}, condition (\ref{sos1}) implies  (\ref{gamma}) for the  sensitivities in  (\ref{models1}). Step 4 of Algorithm \ref{alg2} provides  normalized  sensitivities $\tilde{\mathbf s}^{(j)}$ satisfying $\sum_{j=0}^{N_c-1} |\tilde{s}_{\nn}^{(j)}|^2 =1$ (for $\sum_{j=0}^{N_c-1} |s_{\nn}^{(j)}|^2 \neq 0$). Assumption 2 ensures the unique solvability of the system (\ref{riesig}) to recover $\tilde{\mm}$. 
 Note that  Assumption 2 can only be satisfied
if  for $m_{\nn} \neq 0$ we have $\sum_{j=0}^{N_c-1} |\tilde{s}_{\nn}^{(j)}|^2 =1$. 
 Finally, according to (\ref{sos1}), $\tilde{\mm}$ is uniquely defined after transition to the nonnegative real image $|\tilde{\mm}| = (|m_{\nn}|)_{\nn\in \Lambda_N}$, and the 
factors $\alpha_{\nn}=\sign(m_{\nn})$ as in (\ref{gamma}) finally yield the unique sensitivities 
$\tilde{\mathbf s}^{(j)}$ in (\ref{models1}).
\end{proof}

\medskip
\noindent
The matrices ${\mathbf B}^{(j)}$ in (\ref{riesig}) depend only on the coil sensitivities  and on the projection operator ${\mathcal P}$, i.e., on the number and the distribution of the acquired measurements.
We will study the invertibility of $\sum_{j=0}^{N_c-1} ({\mathbf B}^{(j)})^* {\mathbf B}^{(j)}$  more exactly in Section \ref{sec:ana2}.

In this section we focus on investigating  the first assumption in Theorem  \ref{theorec2}  
that the matrix ${\mathbf A}_M$ in (\ref{AM}) possesses a nullspace of  dimension 1.

\noindent
We start with some notations.
For  ${\mathbf m} = (m_{\nn})_{\nn \in \Lambda_N} \in {\mathbb C}^{N^2}$ denote by
\begin{equation}\label{sm} S_N(\mm):= \{ \nn \in \Lambda_N, \, m_{\nn} \neq 0 \} , \qquad N_m := |S_N(\mm)|
\end{equation}
the index set of nonzero components of $\mm$ and its cardinality.
Let $L=2n+1$ as before. For   ${\mathbf c}^{(j)} =(c_{\rr}^{(j)})_{\rr \in \Lambda_L}$ 
 we define the corresponding bivariate Laurent polynomial
$$ \textstyle c^{(j)}(\zz):=\sum\limits_{\rr \in \Lambda_L} c_{\rr}^{(j)} \zz^{\rr} = \sum\limits_{r_1 =-n} ^n \sum\limits_{r_2 =-n} ^n c_{(r_1,r_2)}^{(j)} 
 z_1^{r_1} z_2^{r_2}, \quad   \, \zz \in {\mathbb C}^2, \, z_1\neq 0, z_2 \neq 0,$$
with support $\Lambda_L$. We say that $c^{(j)}(\zz)$ has degree $2L-2$,  if the coefficients corresponding to boundary indices of $\Lambda_{L}$ do not all vanish, i.e., if 
\begin{align}\label{bound}\textstyle \sum\limits_{r_2=-n}^n |c_{(n,r_2)}^{(j)}| >0, \quad \sum\limits_{r_2=-n}^n |c_{(-n,r_2)}^{(j)}| >0, \quad \sum\limits_{r_1=-n}^n |c_{(r_1,n)}^{(j)}| >0,
\quad \sum\limits_{r_1=-n}^n |c_{(r_1,-n)}^{(j)}| >0. 
\end{align}

As shown  in Theorem \ref{theoc}, $(\ref{discmodel})$ with  $(\ref{sjk})$ implies that the matrix ${\mathbf A}_N$ in (\ref{riesig}) always possesses a nullspace of dimension at least $1$. The next lemma provides necessary and sufficient conditions for vectors in the nullspace of ${\mathbf A}_N$.  This observation will be useful to show that the nullspace of ${\mathbf A}_N$ is of dimension $1$ almost surely.

\begin{lemma}\label{lemma1}
Let  $(\ref{discmodel})$ with $(\ref{sjk})$ be  satisfied for $j=0, \ldots , N_c-1$ and $2L < N$.
Assume that  ${\mathbf V}_{\mm} :=(\omega_N^{-\nn \cdot \elll})_{\nn \in S_N(\mm), \elll \in \Lambda_{2L-1}} \in {\mathbb C}^{N_m \times (2L-1)^2}$, with $S_N(\mm)$ in $(\ref{sm})$ has full rank $(2L-1)^2$.
Let ${\mathbf c}:=((\cc^{(0)})^T, \ldots , (\cc^{(N_c-1)})^T)^T \in {\mathbb C}^{N_c L^2}$   with ${\mathbf c}^{(j)} =(c_{\rr}^{(j)})_{\rr \in \Lambda_L} \in {\mathbb C}^{L^2}$ and 
and ${\mathbf v}= (({\mathbf v}^{(0)})^T, ({\mathbf v}^{(1)})^T,  \ldots , ({\mathbf v}^{(N_c-1)})^T)^T$ 
with
${\mathbf v}^{(j)} =(v_{\rr}^{(j)})_{\rr \in \Lambda_L} \in {\mathbb C}^{L^2}$ be  two vectors  in the nullspace of the matrix ${\mathbf A}_N$ in $(\ref{gross})$. Then
the corresponding  bivariate Laurent polynomials $c^{(j)}(\zz)$ and $v^{(j)}(\zz)$ 
satisfy for all $\zz \in {\mathbb C}^2$ with $z_1\neq 0$, $z_2 \neq 0$ the relation
\begin{equation}\label{lau} \textstyle 
c^{(j)}(\zz) \, \sum\limits_{\mycom{\ell=0}{\ell\neq j}}^{N_c-1} v^{(\ell)}(\zz) = v^{(j)}(\zz) \, \sum\limits_{\mycom{\ell=0}{\ell\neq j}}^{N_c-1}  c^{(\ell)}(\zz), \quad j=0, \ldots , N_c-1. 
\end{equation}
\end{lemma}

\begin{proof}
Recall that the matrices ${\mathbf Y}_{N,L}^{(j)} = (y_{(\nuu-\rr) \mod \Lambda_N}^{(j)})_{\nuu \in \Lambda_N, \rr \in \Lambda_L}$ can by (\ref{discmodel}), (\ref{eigc0}) and (\ref{projectR}) be rewritten as
$ {\mathbf Y}_{N,L}^{(j)} = {\mathcal F} \diag(\check{\yy}^{(j)}) \, {\mathbf W} = {\mathcal F} \diag(\mm \circ {\mathbf s}^{(j)}) \, {\mathbf W}$.
By ${\mathbf s}^{(j)} = {\mathbf W} {\mathbf c}^{(j)}$, it follows that 
$$ {\mathcal F}^{-1} {\mathbf Y}_{N,L}^{(j)} \, {\mathbf v}^{(\ell)} = \diag(\mm) \, \big(({\mathbf W} {\mathbf c}^{(j)}) \circ ({\mathbf W} {\mathbf v}^{(\ell)}) \big), \qquad j, \ell = 0, \ldots, N_c-1. $$
Thus, ${\mathbf A}_N {\mathbf v} = {\mathbf 0}$ is equivalent to 
\begin{align}\label{eig10}\diag (\mm)  \Big( -\big(\sum\limits_{\mycom{\ell=0}{\ell\neq j}}^{N_c-1} {\mathbf W} {\mathbf c}^{(\ell)} \big)  \circ ({\mathbf W} {\mathbf v}^{(j)})
+  ({\mathbf W} {\mathbf c}^{(j)}) \circ \big(\sum\limits_{\mycom{\ell=0}{\ell\neq j}}^{N_c-1} {\mathbf W} {\mathbf v}^{(\ell)} \big) \Big)= {\mathbf 0}
\end{align}
for $j=0, \ldots, N_c-1$. Using the Laurent polynomial notation in Lemma \ref{lemma1}, we have ${\mathbf W}{\mathbf c}^{(j)}= (c^{(j)}(\omega_N^{-\nn}))_{\nn \in \Lambda_N}$ and ${\mathbf W}{\mathbf v}^{(j)}= (v^{(j)}(\omega_N^{-\nn}))_{\nn \in \Lambda_N}$. We conclude from (\ref{eig10}) that 
\begin{align}\label{lau1}
m_{\nn} \Big( c^{(j)}(\omega_N^{-\nn}) \, \sum\limits_{\mycom{\ell=0}{\ell\neq j}}^{N_c-1} v^{(\ell)}(\omega_N^{-\nn}) - v^{(j)}(\omega_N^{-\nn}) \, \sum\limits_{\mycom{\ell=0}{\ell\neq j}}^{N_c-1}  c^{(\ell)}(\omega_N^{-\nn}) \Big) = 0
\end{align}
for all $j=0, \ldots , N_c-1$ and all  $\nn \in \Lambda_N$.
Since the Laurent polynomials $c^{(j)}(\zz)$ and $v^{(\ell)}(\zz)$  are  of degree at most $2(L-1)$, it follows that  the sum $c^{(j)}(\zz) \, \sum_{\mycom{\ell=0}{\ell\neq j}}^{N_c-1} v^{(\ell)}(\zz) - v^{(j)}(\zz) \, \sum_{\mycom{\ell=0}{\ell\neq j}}^{N_c-1}  c^{(\ell)}(\zz)$ has at most degree $4(L-1)$, 
 and by construction can only have nonzero coefficients  corresponding to the index set $\Lambda_{2L-1}$. 
By assumption, we have $m_{\nn} \neq 0$ for all $\nn \in  S_N(\mm)$,  and the matrix 
${\mathbf V}_{\mm} :=(\omega_N^{-\nn \cdot \rr})_{\nn \in S_N(\mm), \rr \in \Lambda_{2L-1}}$ has full rank $(2L-1)^2$. Therefore, the interpolation conditions in (\ref{lau1}) imply  the condition (\ref{lau}).
\end{proof}

\begin{theorem}\label{theoana1}
Let the assumptions of Lemma $\ref{lemma1}$ be satisfied and assume that ${\mathbf A}_N {\mathbf c} = {\mathbf 0}$  with ${\mathbf c}:=((\cc^{(0)})^T, \ldots , (\cc^{(N_c-1)})^T)^T \in {\mathbb C}^{N_c L^2}$,  where the partial vectors ${\mathbf c}^{(j)}$ satisfy $(\ref{bound})$. 
Then ${\mathbf A}_N$ in $(\ref{gross})$ has a nullspace of dimension $1$ almost surely.
\end{theorem}

\begin{proof}
By  (\ref{gross}), the dimension of the nullspace of ${\mathbf A}_N$ is at least $1$ and we have ${\mathbf A}_N {\mathbf c} = {\mathbf 0}$ with ${\mathbf c}:=((\cc^{(0)})^T, \ldots , (\cc^{(N_c-1)})^T)^T \in {\mathbb C}^{N_c L^2}$.
By Lemma \ref{lemma1}, each further vector ${\mathbf v}= (({\mathbf v}^{(0)})^T, ({\mathbf v}^{(1)})^T, \ldots , ({\mathbf v}^{(N_c-1)})^T)^T$ with ${\mathbf v}^{(j)} =(v_{\rr}^{(j)})_{\rr \in \Lambda_L} \in {\mathbb C}^{L^2}$ in the nullspace of ${\mathbf A}_N$ satisfies (\ref{lau}).
Condition (\ref{bound}) ensures that $c^{(j)}(\zz)$, $j=0, \ldots , N_c-1$, have full degree $2(L-1)$, such that the (bivariate) rational functions
$  \sum_{\mycom{\ell=0}{\ell\neq j}}^{N_c-1}  \frac{c^{(\ell)}(\zz)}{c^{(j)}(\zz)} 
=  \sum_{\mycom{\ell=0}{\ell\neq j}}^{N_c-1}  \frac{z_1^n z_2^n c^{(\ell)}(\zz)}{z_1^n z_2^n c^{(j)}(\zz)}
$ are of type $(2(L-1), 2(L-1))$.
This type is almost surely minimal, i.e., $ \sum_{\mycom{\ell=0}{\ell\neq j}}^{N_c-1}  z_1^n z_2^n c^{(\ell)}(\zz)$ and $z_1^n z_2^n c^{(j)}(\zz)$ almost surely have no common zeros.
For each $j =0, \ldots , N_c-1$ we conclude: If $v^{(j)}(\zz)$ is not the zero polynomial, then (\ref{lau}) implies that 
$$ \sum\limits_{\mycom{\ell=0}{\ell\neq j}}^{N_c-1}  \frac{c^{(\ell)}(\zz)}{c^{(j)}(\zz)} = \sum\limits_{\mycom{\ell=0}{\ell\neq j}}^{N_c-1}  \frac{v^{(\ell)}(\zz)}{v^{(j)}(\zz)} $$
almost everywhere, and since $v^{(j)}(\zz)$ is  by construction a Laurent polynomial of degree at most $2(L-1)$,
it follows that $v^{(j)}(\zz)= \mu c^{(j)}(\zz)$ with some global constant  $\mu$ which does not depend on $j$. 
Hence, ${\mathbf A}_N$ in $(\ref{gross})$ has a nullspace of dimension 1. 
\end{proof}

\begin{remark}
The conditions of Lemma $\ref{lemma1}$ 
can be easily satisfied.
For example,  ${\mathbf V}_{\mm}$ already satisfies the rank assumption, if $m_{\nn} \neq 0$ for $\nn \in \Lambda_{2L-1}$.
If $m_{\nn} =0$ for some $\nn \in \Lambda_{2L-1}$ then it is sufficient to find an index $\nn+(2L-1)\nuu \in \Lambda_{N}$ with $m_{\nn+(2L-1)\nuu} \neq 0$ to satisfy the condition. Therefore, the nullspace condition is satisfied almost surely for ${\mathbf A}_M$ with $M \ge 2L$, if (\ref{bound}) holds and if the  ACS region is chosen large enough.
 If (\ref{bound}) is not satisfied, this is an indication that $L$ has been chosen too large in the model $(\ref{sjk})$ or the ACS region is too small.\\
In practice, the considered models (\ref{discmodel}) and (\ref{sjk}) are not exactly satisfied for MRI data, and the matrix ${\mathbf A}_M$ may have several small singular values of similar size. In this case, a suitable linear combination singular vectors corresponding to these singular values seems to be appropriate such that ${\mathbf d}$ in step 4 of Algorithm \ref{alg2} is well bounded away from zero, see Section \ref{sec:num}.
\end{remark}

\subsection{Analysis of Algorithm \ref{alg3}}
\label{sec:ana2}

In this subsection, we show the convergence of the iteration method in Algorithm \ref{alg3} to a solution of (\ref{riesig}).

\begin{theorem}\label{theoB}
If the model assumptions $(\ref{discmodel})$ and $(\ref{sjk})$ are satisfied with $\sum_{j=0}^{N_c-1} |s^{(j)}_{\nn}|^2 >0$ for all $\nn \in {S}_N({\mm})$ in $(\ref{sm})$,  then the iteration  scheme in Algorithm $\ref{alg3}$ converges  for $0 \le \beta < N^2$ to a solution $\tilde{\mm}$ of $(\ref{riesig})$.  In particular, for $\beta=0$, Algorithm  $\ref{alg3}$ always converges to a solution with minimal $2$-norm.
\end{theorem}

\begin{proof} 1. We first assume that $\sum_{j=0}^{N_c-1} |s^{(j)}_{\nn}|^2 >0$ for all $\nn \in \Lambda_N$.
Then the normalized vectors $\tilde{\mathbf s}^{(j)}$ defined in (\ref{sjn}) satisfy
$ \sum_{j=0}^{N_c-1}  |\tilde{s}_{\nn}^{(j)}|^2 =1$ for $\nn \in \Lambda_N$.
Since $\frac{1}{N} {\mathcal F}$ is a unitary matrix, i.e., ${\mathcal F}^{-1} = \frac{1}{N^2} {\mathcal F}^*$, and ${\mathcal P}$ a projection matrix,  we obtain for  the largest eigenvalue of the matrix ${\mathbf Q} :=\frac{1}{N^2} ( \beta {\mathbf I}  + \sum_{j=0}^{N_c-1} ({\mathbf B}^{(j)})^* {\mathbf B}^{(j)})$  with ${\mathbf B}^{(j)} = {\mathcal P} {\mathcal F} \diag(\tilde{\mathbf s}^{(j)})$ that
\begin{align*}  
\textstyle \max\limits_{\mycom{{\mathbf w} \in {\mathbb C}^{N^2}}{\|{\mathbf w}\|_2=1}} {\mathbf w}^*  {\mathbf Q} {\mathbf w} 
&=   \textstyle  \frac{\beta}{N^2} \!+\!
\max\limits_{\mycom{{\mathbf w} \in {\mathbb C}^{N^2}}{\|{\mathbf w}\|_2=1}} \sum\limits_{j=0}^{N_c-1} \|\frac{1}{N}{\mathbf B}^{(j)} {\mathbf w}\|_2^2
=  \textstyle   \frac{\beta}{N^2} \!+\! \max\limits_{\mycom{{\mathbf w} \in {\mathbb C}^{N^2}}{\|{\mathbf w}\|_2=1}} \sum\limits_{j=0}^{N_c-1} \|{\mathcal P} \frac{1}{N}{\mathcal F}  \diag(\tilde{\mathbf s}^{(j)} ) {\mathbf w}\|_2^2
\end{align*}
and therfore
\begin{align} \nonumber
\textstyle \max\limits_{\mycom{{\mathbf w} \in {\mathbb C}^{N^2}}{\|{\mathbf w}\|_2=1}} {\mathbf w}^*  {\mathbf Q} {\mathbf w} 
&\le  \textstyle   \frac{\beta}{N^2} + \max\limits_{\mycom{{\mathbf w} \in {\mathbb C}^{N^2}}{\|{\mathbf w}\|_2=1}} \sum\limits_{j=0}^{N_c-1} \|\diag(\tilde{\mathbf s}^{(j)} ) {\mathbf w}\|_2^2 
=     \frac{\beta}{N^2} + \max\limits_{\mycom{{\mathbf w} \in {\mathbb C}^{N^2}}{\|{\mathbf w}\|_2=1}} \sum\limits_{j=0}^{N_c-1} \sum\limits_{\nn \in \Lambda_N} |\tilde{s}_{\nn}^{(j)}|^2 |w_{\nn}|^2 
\\
\label{eigab}
&=  \textstyle  \frac{\beta}{N^2} +  \max\limits_{\mycom{{\mathbf w} \in {\mathbb C}^{N^2}}{\|{\mathbf w}\|_2=1}}  \sum\limits_{\nn \in \Lambda_N} \big( \sum\limits_{j=0}^{N_c-1}  |\tilde{s}_{\nn}^{(j)}|^2 \big) |w_{\nn}|^2 
= 1  +  \frac{\beta}{N^2}.
\end{align}
The iteration procedure in step 2 of Algorithm \ref{alg3} implies 
\begin{align}\nonumber
\textstyle 
\tilde{\mm}_{\kappa+1} \!&=  \textstyle  \! \Big( \sum\limits_{j=0}^{N_c-1} \!\!\! \diag (\tilde{\mathbf s}^{(j)})^* {\mathcal F}^{-1} (  (1- \frac{\beta}{N^2}) {\mathbf I} - {\mathcal P})
{\mathcal F}  \diag (\tilde{\mathbf s}^{(j)}) \! \Big)  \tilde{\mm}_{\kappa} \!+ \!\!  \sum\limits_{j=0}^{N_c-1} \!\! \diag (\tilde{\mathbf s}^{(j)})^*  {\mathcal F}^{-1} {\mathbf y}_0^{(j)}\\
&= \textstyle   \Big((1- \frac{\beta}{N^2}) {\mathbf I} - \frac{1}{N^2} \big( \sum\limits_{j=0}^{N_c-1}  ({\mathbf B}^{(j)})^* {\mathbf B}^{(j)} \big)\Big) \tilde{\mm}_{\kappa} +  \frac{1}{N^2}   \sum\limits_{j=0}^{N_c-1} ({\mathbf B}^{(j)})^*  {\mathbf y}_0^{(j)}. 
\label{iteration}
\end{align}
Convergence of this iteration scheme is ensured if   the spectral radius of the iteration matrix 
\begin{align} \textstyle   (1- \frac{\beta}{N^2}) {\mathbf I} - \frac{1}{N^2} \Big( \sum\limits_{j=0}^{N_c-1}  ({\mathbf B}^{(j)})^* {\mathbf B}^{(j)} \Big) =  {\mathbf I} - {\mathbf Q}
\end{align}
 is smaller than $1$.
The matrix ${\mathbf Q}$ is by construction  positive semidefinite, and by (\ref{eigab}), all its eigenvalues are in the interval  $[\frac{\beta}{N^2}, 1+\frac{\beta}{N^2}]$. Therefore, ${\mathbf I} - {\mathbf Q}$ has all its eigenvalues in $[-\frac{\beta}{N^2}, 1-\frac{\beta}{N^2}]$.

\noindent
2. Let ${\mathbf Q}$ be  invertible  and $0 \le \beta < N^2$. Then $\|{\mathbf I} - {\mathbf Q}\|_2 < 1$ and  (\ref{iteration}) converges. Its  fixed point satisfies 
$$  \textstyle \tilde{\mm}=  \Big( (1- \frac{\beta}{N^2}){\mathbf I} - \frac{1}{N^2} \Big( \sum\limits_{j=0}^{N_c-1}  ({\mathbf B}^{(j)})^* {\mathbf B}^{(j)} \Big)  \Big) \tilde{\mm} +  \frac{1}{N^2}  \sum\limits_{j=0}^{N_c-1} ({\mathbf B}^{(j)})^*  {\mathbf y}_0^{(j)}, $$
i.e., $\big(\beta {\mathbf I} + \sum_{j=0}^{N_c-1}  ({\mathbf B}^{(j)})^* {\mathbf B}^{(j)}\big) \tilde{\mm} = \sum_{j=0}^{N_c-1}  ({\mathbf B}^{(j)})^* {\mathbf y}^{(j)}$.

\noindent
3. Assume now that ${\mathbf Q}$ is not invertible, i.e., $\beta=0$ and $\sum_{j=0}^{N_c-1}  ({\mathbf B}^{(j)})^* {\mathbf B}^{(j)}$ is singular. Then
from the definitions of 
${\mathbf y}^{(j)}_0= {\mathcal P} {\mathbf y}^{(j)} = {\mathcal P} {\mathcal F} \diag(\tilde{\mathbf s}^{(j)}) \tilde{\mm} = {\mathbf B}^{(j)} \tilde{\mm}$  and of $\tilde{\mm}_0$ in Algorithm \ref{alg3} it follows with ${\mathcal P}= {\mathcal P}^* = {\mathcal P}^* {\mathcal P}$
that
\begin{align*}  \textstyle \tilde{\mm}_0 =  
\sum\limits_{j=0}^{N_c-1} \overline{{\tilde{\mathbf s}^{(j)}} } \circ  ({\mathcal F}^{-1} \yy^{(j)}_0)
= \frac{1}{N^2}  \sum\limits_{j=0}^{N_c-1}  ({\mathbf B}^{(j)})^* \yy^{(j)}_0
= \frac{1}{N^2} \sum\limits_{j=0}^{N_c-1}  ({\mathbf B}^{(j)})^* {\mathbf B}^{(j)} \tilde{\mm} = {\mathbf Q} \tilde{\mm}
\end{align*} 
is in the image space of ${\mathbf Q}= {\mathbf Q}^*$ and therefore orthogonal to the
 nullspace of ${\mathbf Q}$. In particular,  we conclude that the system (\ref{riesig}) has always solutions.
 The iteration scheme (\ref{iteration}) implies that also each further vector $\mm_{\kappa}$ is in the image space of ${\mathbf Q}$. Restricted to the image space of ${\mathbf Q}$, the iteration matrix in (\ref{iteration}) has again a spectral radius smaller than $1$.
 Hence, the  fixed point of (\ref{iteration}) exists also in this case and is a solution of (\ref{riesig}). Moreover, this solution is of minimal 2-norm, since it is orthogonal to the nullspace of $ {\mathbf Q} $.
 
 \noindent
4. Finally, let $\Lambda_0 \neq \emptyset$ be the index set with  $\sum_{j=0}^{N_c-1} | s_{\nn}^{(j)}|^2 =0 $ for $\nn \in \Lambda_0$. By assumption, $\Lambda_0  \cap   {\mathcal S}_N(\mm) =\emptyset$. Then the vectors $\tilde{\mathbf s}^{(j)}$, $\yy_0^{(j)}$ and $\tilde{\mm}_{\kappa}$  all have zero components for $\nn \in \Lambda_0$, the matrices ${\mathbf B}^{(j)}$ have zero columns, and the matrix ${\mathbf Q}$ has zero columns and zero rows for all indices $\nn \in \Lambda_0$. In this case the proof applies as before for the reduced vectors and matrices, where all components of the vectors and all columns or columns/rows of the matrices corresponding to the index set $\Lambda_0$ are removed.
\end{proof}

\subsection{Derivation of Algorithm \ref{alg4} and Unique Solvability of (\ref{riesig})}
\label{sec:ana3}

In this section we derive Algorithm \ref{alg4}. Furthermore, we study the invertibility of $\sum_{j=0}^{N_c-1}  ({\mathbf B}^{(j)})^* {\mathbf B}^{(j)}$, which depends on the index set $\Lambda_{\mathcal P} \subset \Lambda_N$ of acquired data.

Assume that we have computed $\tilde{\mathbf s}^{(j)}$, $j=0, \ldots , N_c-1$, as in step 5 of Algorithm \ref{alg2}. 
To derive a direct efficient algorithm to solve  (\ref{riesig}), we require some notations. 
To avoid confusion,  in this section, we explicitly distinguish between the $N \times N$ matrices ${\mathbf y}^{(j)}$, ${\mathbf s}^{(j)}$, $\tilde{\mathbf s}^{(j)}$, $\mm$, ${\mathcal P}$ and their columnwise vectorized counterparts $\vec({\mathbf y}^{(j)})$, $\vec({\mathbf s}^{(j)})$, $\vec(\tilde{\mathbf s}^{(j)})$, $\vec(\mm)$, $\vec({\mathcal P})$ in ${\mathbb C}^{N^2}$.
The Kronecker product of  two $N \times N$ matrices ${\mathbf A}=(a_{k,\ell})_{k,\ell=-\frac{N}{2}}^{\frac{N}{2}-1} \in {\mathbb C}^{N \times N}$  and ${\mathbf B} \in {\mathbb C}^{N \times N}$ is given as 
$$ {\mathbf A} \otimes {\mathbf B} := (a_{k,\ell} {\mathbf B} )_{k,\ell=-\frac{N}{2}}^{\frac{N}{2}-1} \in {\mathbb C}^{N^2 \times N^2}. $$
For properties of the Kronecker product of matrices and the discrete Fourier transform, we refer to Section 3.4 in \cite{PPST23}.
In matrix form, the Fourier operator ${\mathcal F}$  denotes the centered 2D DFT, and in vector form it can be written as ${\mathcal F} = {\mathbf F}_N \otimes {\mathbf F}_N$, where ${\mathbf F}_N = (\omega_N^{k \ell})_{k,\ell=-\frac{N}{2}}^{\frac{N}{2}-1}$ is the  (centered) Fourier matrix of size $N \times N$, which satisfies ${\mathbf F}_{N}^{*} = N {\mathbf F}_{N}^{-1}$.
With  these notations, (\ref{discmodel})  with incomplete $k$-space data  has the matrix form 
$${\mathcal P} \circ {\mathbf y}^{(j)} = {\mathcal P} \circ( {\mathbf F}_N ( \mm \circ {\mathbf s}^{(j)}) {\mathbf F}_N)$$
and the vector form 
$\vec({\mathcal P}) \circ \vec({\mathbf y}^{(j)}) = \vec({\mathcal P}) \circ ({\mathbf F}_N \otimes {\mathbf F}_N) ( \vec(\mm) \circ \vec({\mathbf s}^{(j)}) ).$ 

We present exemplarily the case, where $N$ is a multiple of $8$, and $\Lambda_{\mathcal P}$ contains 
beside the ACS region $\Lambda_{L+M-1}$ the subset  
$\Lambda_{{\mathcal P}_1} := \{k: \, -\frac{N}{2} \le k \le  \frac{N}{2}-1\} \times \{4k: \, -\frac{N}{8} \le k \le  \frac{N}{8}-1\} $, with $\frac{N^2}{4}$ components,
where,  besides the ACS region, only every fourth column of the $N \times N$ image representation ${\mathbf y}^{(j)}$ is acquired. Let ${\mathcal P}_{1}$ denote the corresponding projection matrix, i.e., 
${\mathcal P}_{1} = (p_{k,\ell})_{(k,\ell) \in \Lambda_N}$ with $p_{k,\ell} = 1$ if $\ell \equiv 0 \, \mod \, 4$ and $p_{k,\ell} = 0$ otherwise.
Then $\diag(\vec({\mathcal P}_{1}))$ is of the form
$$ \diag(\vec({\mathcal P}_{1})) = {\mathbf P}_1 \otimes {\mathbf I}_N \quad \text{with} {\mathbf P}_1
= \diag((1,0,0,0,1,0,0,0, \ldots , 1,0,0,0)) \in {\mathbb C}^{N \times N}, $$
where ${\mathbf  I}_N$ is the $N \times N$ unitary matrix. For the $N^2 \times N^2$ coefficient matrix of (\ref{riesig}), $\sum_{j=0}^{N_c-1} ({\mathbf B}^{(j)})^* {\mathbf B}^{(j)}$, we obtain by $\diag(\vec({\mathcal P}_{1}))^*\diag(\vec({\mathcal P}_{1}))= \diag(\vec({\mathcal P}_{1}))={\mathbf P}_1 \otimes {\mathbf I}_N$
\begin{align*} 
\sum_{j=0}^{N_c-1} ({\mathbf B}^{(j)})^* {\mathbf B}^{(j)} &= \textstyle  \sum\limits_{j=0}^{N_c-1}
\diag(\vec(\tilde{\mathbf s}^{(j)}))^* ({\mathbf F}_N \otimes {\mathbf F}_N)^* ({\mathbf P}_1 \otimes {\mathbf I}_N) ({\mathbf F}_N \otimes {\mathbf F}_N) \diag(\vec(\tilde{\mathbf s}^{(j)})) \\
&= \textstyle  \sum\limits_{j=0}^{N_c-1}
\diag(\vec(\tilde{\mathbf s}^{(j)}))^* (({\mathbf F}_N^*  {\mathbf P}_1 {\mathbf F}_N) \otimes ({\mathbf F}_N)^* {\mathbf F}_N )) \diag(\vec(\tilde{\mathbf s}^{(j)})),
\end{align*}
where $({\mathbf F}_N)^* {\mathbf F}_N  = N \, {\mathbf I}_N$ and $({\mathbf F}_N)^{*}  {\mathbf P}_1 {\mathbf F}_N
= \frac{N}{4} (\delta_{k,\ell  \, \mod \frac{N}{4}})_{k,\ell=-\frac{N}{2}}^{\frac{N}{2}-1} =  \frac{N}{4} {\mathbf E}_4 \otimes {\mathbf I}_{\frac{N}{4}}$, where
$ {\mathbf E}_4$ is a $4 \times 4$ matrix containing only ones.
Therefore, using that ${\mathbf I}_{\frac{N}{4}} \otimes {\mathbf I}_N = {\mathbf I}_{\frac{N^2}{4}}$, 
\begin{align*} 
\textstyle \sum\limits _{j=0}^{N_c-1} ({\mathbf B}^{(j)})^* {\mathbf B}^{(j)} &= \textstyle  \frac{N^2}{4} \sum\limits_{j=0}^{N_c-1}
\diag(\vec(\overline{\tilde{\mathbf s}^{(j)}})) ( {\mathbf E}_4 \otimes {\mathbf I}_{\frac{N^2}{4}}) \diag(\vec(\tilde{\mathbf s}^{(j)}))
\end{align*}
is a block diagonal matrix with only $4$ entries per row and per column. The right-hand-side of (\ref{riesig}) is given by
\begin{align*} &  \textstyle \sum\limits_{j=0}^{N_c-1} ({\mathbf B}^{(j)})^* \diag( \vec({\mathcal P}_1)) \vec({\yy}^{(j)})
=  \textstyle \sum\limits_{j=0}^{N_c-1} \diag(\vec(\overline{\tilde{\mathbf s}^{(j)}})) ({\mathbf F}_N^* \otimes {\mathbf F}_N^*) ({\mathbf P}_1 \otimes {\mathbf I}_N) \vec({\yy}^{(j)}) \\
&=   \textstyle N^2 \sum\limits_{j=0}^{N_c-1} \vec(\overline{\tilde{\mathbf s}^{(j)}}) \circ \vec( {\mathbf F}_N^{-1} ({\yy}^{(j)} {\mathbf P}_1)  {\mathbf F}_N^{-1}) = N^2 \sum\limits_{j=0}^{N_c-1} \vec(\overline{\tilde{\mathbf s}^{(j)}})  \circ  \vec(\check{\yy}_{{\mathbf P}_1}^{(j)}),
\end{align*}
where $\check{\yy}_{{\mathbf P}_1}^{(j)} = (\check{y}_{{\mathbf P}_1,(k,\ell)}^{(j)})_{(k,\ell) \in \Lambda_N}$ denotes the inverse 2D FFT of the incomplete $k$-space data ${\mathcal P}_1 \circ {\mathbf y}^{(j)}$.
The structure of $\sum_{j=0}^{N_c-1} ({\mathbf B}^{(j)})^* {\mathbf B}^{(j)}$ implies that the large system (\ref{riesig}) falls apart into $\frac{N^2}{4}$ small systems of size $4 \times 4$. For 
$k=-\frac{N}{2}, \ldots , \frac{N}{2}-1$ and $\ell= -\frac{N}{8}, \ldots , \frac{N}{8}-1$, we have to solve 
\begin{align*}
\textstyle \frac{1}{4}\left( \frac{4}{N^{2}}\, \beta \, {\mathbf I}_4 + \sum\limits_{j=0}^{N_c-1}  \overline{\mathbf s}^{(j)}_{k,\ell}  \, ({\mathbf s}^{(j)}_{k,\ell})^T \right) \, \left(\begin{array}{c}{m}_{k,\ell-\frac{3N}{8}} \\
 {m}_{k,\ell-\frac{N}{8}}\\
 {m}_{k,\ell+\frac{N}{8}}\\
 {m}_{k,\ell+\frac{3N}{8}} \end{array} \right)
 =  \left( \sum\limits_{j=0}^{N_c-1}     \overline{\mathbf s}^{(j)}_{k,\ell} \circ   \left( \begin{array}{c} \check{y}_{{\mathbf P}_1,k,\ell-\frac{3N}{8}}^{(j)} \\ \check{y}_{{\mathbf P}_1,k,\ell-\frac{N}{8}}^{(j)} \\ \check{y}_{{\mathbf P}_1,k,\ell+\frac{N}{8}}^{(j)} \\\check{y}_{{\mathbf P}_1,k,\ell+\frac{3N}{8}}^{(j)} \end{array} \right)
 \,  \right),
\end{align*}
with ${\mathbf s}^{(j)}_{k,\ell}  =(s_{k,\ell-\frac{3N}{8}}^{(j)}, s_{k,\ell-\frac{N}{8}}^{(j)} , s_{k,\ell+\frac{N}{8}}^{(j)}, s_{k,\ell+\frac{3N}{8}}^{(j)})^T \in {\mathbb C}^4$ to recover $\mm$. These equation systems are implemented in Algorithm \ref{alg4}.
For other patterns one can similarly find a decomposition of $\sum_{j=0}^{N_c-1} ({\mathbf B}^{(j)})^* {\mathbf B}^{(j)}$.

\noindent
Finally, we employ the obtained insights on the structure of the 
matrix $\sum\limits_{j=0}^{N_c-1} ({\mathbf  B}^{(j)} )^* \, {\mathbf  B}^{(j)} $ in (\ref{riesig}) (with $\beta=0$) to find simple conditions for invertibility of this matrix, such that the regularization with $\beta>0$ is not needed. First, we give an invertibility condition in terms of  the vectors ${\mathbf c}^{(j)}$ obtained in the steps 1-3 of the MOCCA Algorithm \ref{alg2}.

\begin{theorem}\label{invert1}
Let  ${\mathbf c}^{(j)} = (c_{\rr}^{(j)})_{\rr \in \Lambda_{L}} \in {\mathbb C}^{L^2}$, $j=0, \ldots , N_c-1$, be the vectors determining the (vectorized) coil sensitivities $ \vec({\mathbf s}^{(j)})= {\mathbf W} {\mathbf c}^{(j)}$ in $(\ref{sj})$ and let $({c}_{\nn}^{(j)})_{\nn \in \Lambda_{N}} \in {\mathbb C}^{N^2}$ be the extensions of ${\mathbf c}^{(j)}$ to $\Lambda_{N}$ with  ${c}_{\nn}^{(j)} = 0$ for $\nn \in \Lambda_{N} \setminus\Lambda_{L}$. Further,  let $\sum_{j=0}^{N_c-1} |s^{(j)}_{\nuu}|^2 >0$ for all $\nuu \in \Lambda_N$.
Then $\sum_{j=0}^{N_c-1} ({\mathbf  B}^{(j)} )^* \, {\mathbf  B}^{(j)} $ is invertible if and only if the sparse matrix
\begin{align}\label{rankcon} 
\left( \begin{array}{c}  (\hat{s}^{(0)}_{(\nuu - \nn)\mod \Lambda_N})_{\nuu \in \Lambda_{\mathcal P}, \nn \in \Lambda_{N}} \\
(\hat{s}^{(1)}_{(\nuu - \nn)\mod \Lambda_N})_{\nuu \in \Lambda_{\mathcal P}, \nn \in \Lambda_{N}} \\ 
\vdots \\
(\hat{s}^{(N_c-1)}_{(\nuu - \nn)\mod \Lambda_N})_{\nuu \in \Lambda_{\mathcal P}, \nn \in \Lambda_{N}} \end{array} \right)  := 
N^2 \left( \begin{array}{c}  ({c}^{(0)}_{(\nuu - \nn)\mod \Lambda_N})_{\nuu \in \Lambda_{\mathcal P}, \nn \in \Lambda_{N}} \\
({c}^{(1)}_{(\nuu - \nn)\mod \Lambda_N})_{\nuu \in \Lambda_{\mathcal P}, \nn \in \Lambda_{N}} \\ 
\vdots \\
({c}^{(N_c-1)}_{(\nuu - \nn)\mod \Lambda_N})_{\nuu \in \Lambda_{\mathcal P}, \nn \in \Lambda_{N}} \end{array} \right) \in {\mathbb C}^{N_c |\Lambda_{\mathcal P}| \times N^2}
\end{align}
has rank $N^{2}$. In particular, if $\sum_{j=0}^{N_c-1} ({\mathbf  B}^{(j)} )^* \, {\mathbf  B}^{(j)} $  is invertible, then $N_c |\Lambda_{\mathcal P}| \ge N^2$ and  $\Lambda_{N} \subset (\nuu + \Lambda_{L})_{\nuu \in \Lambda_{\mathcal P}}$.
\end{theorem}

\begin{proof}
Let ${\mathbf g} \in {\mathbb C}^{N^2}$ be a vector with ${\mathbf g}^* \big(\sum_{j=0}^{N_c-1} ({\mathbf  B}^{(j)} )^* \, {\mathbf  B}^{(j)} \big) {\mathbf g} =0$.
Since $({\mathbf  B}^{(j)} )^* \, {\mathbf  B}^{(j)}$ is positive semidefinite for all $j=0, \ldots , N_c-1$, it follows that 
$$ {\mathbf g}^* ({\mathbf  B}^{(j)} )^* \, {\mathbf  B}^{(j)} {\mathbf g} = \|{\mathbf  B}^{(j)} {\mathbf g}\|_2^2 =0, \, \qquad j=0, \ldots , N_c-1, $$
i.e., ${\mathbf  B}^{(j)} {\mathbf g} ={\mathbf 0}$ for $j=0, \ldots , N_c-1$.
By definition, we have   ${\mathbf  B}^{(j)}= {\mathcal P} {\mathcal F} \diag(\vec(\tilde{\mathbf s}^{(j)}))
= {\mathcal P} {\mathcal F}  \diag(\vec(({\mathbf d}^+)^{\frac{1}{2}} \circ {\mathbf s}^{(j)}))$ and defining $\tilde{\mathbf g} := \vec(({\mathbf d}^+)^{\frac{1}{2}} )\circ {\mathbf g} = \diag(\vec(({\mathbf d}^+)^{\frac{1}{2}})) {\mathbf g}$ with ${\mathbf d}^+$ as in step 4 of Algorithm  \ref{alg2} we find
$$  \textstyle {\mathbf B}^{(j)} {\mathbf g} =  {\mathcal P}{\mathcal F} \big(\diag(\vec({\mathbf s}^{(j)}  \circ ({\mathbf d}^+)^{\frac{1}{2}} )) {\mathbf g} \big) = {\mathcal P}{\mathcal F} \, \vec({\mathbf s}^{(j)}) \circ \tilde{\mathbf g}
=  \frac{1}{N^2} {\mathcal P} \big(({\mathcal F} \vec({\mathbf s}^{(j)})) * ({\mathcal  F} \tilde{\mathbf g}) \big) 
={\mathbf 0},$$
where $*$ denotes the 2D discrete $N$-periodic convolution.
From  $\vec({\mathbf s}^{(j)})= {\mathbf W} \cc^{(j)}$ we obtain 
 $\vec(\hat{\mathbf s}^{(j)}) := {\mathcal F} \vec({\mathbf s}^{(j)}) = N^2({c}^{(j)}_{\nuu})_{\nuu \in \Lambda_N}$ is the zero-extension of $\cc^{(j)}$ from $\Lambda_L$ to $\Lambda_N$. 
Hence, 
\begin{align*}
{\mathcal P} ({\mathcal F} \vec({\mathbf s}^{(j)})* {\mathcal  F} \tilde{\mathbf g} ) 
&= \big(\hat{s}_{(\nuu-\rr) \mod \Lambda_N}^{(j)}\big)_{\nuu \in \Lambda_{\mathcal P}, \rr \in \Lambda_N} \, ({\mathcal  F} \tilde{\mathbf g})\\
& = 
N^2 \big({c}_{(\nuu-\rr) \mod \Lambda_N}^{(j)}\big)_{\nuu \in \Lambda_{\mathcal P}, \rr \in \Lambda_N} \, ({\mathcal  F} \tilde{\mathbf g}) = 
{\mathbf 0}
\end{align*}
for $j=0, \ldots, N_c-1$.
Thus,  $\tilde{\mathbf g} ={\mathbf 0}$, if and only if the matrix in (\ref{rankcon})
has full rank $N^2$. Therefore, also  ${\mathbf g} ={\mathbf 0}$, since ${\mathbf d}^+$ has only positive components. In particular, the invertibility of $\sum_{j=0}^{N_c-1} ({\mathbf  B}^{(j)} )^* \, {\mathbf  B}^{(j)} $ implies $N_c |\Lambda_{\mathcal P}| \ge N^2$ and
$\Lambda_{N} \subset (\nuu + \Lambda_{L})_{\nuu \in \Lambda_{\mathcal P}}$. 
\end{proof}

Similarly, we obtain a  condition for other regular patterns of acquired data, as e.g. for the special example that each second row and each second column of the $k$-space data is acquired. 

\begin{corollary}
Assume that the index set of acquired measurements $\Lambda_{\mathcal P}$ contains beside the ACS region index set $\Lambda_{M+L-1}$ the set $2\Lambda_{N/2}$, i.e., every second index in $x$-direction and $y$-direction. Then, the matrix 
$\sum_{j=0}^{N_c-1} ({\mathbf  B}^{(j)} )^* \, {\mathbf  B}^{(j)}$ is invertible, if 
the $(N_c \times 4)$-matrices
\begin{align}\label{rank4} {\mathbf G}_{\nn} := \left( \begin{array}{cccc} s_{\nn + (\frac{N}{4}, \frac{N}{4})}^{(0)} &  s_{\nn + (\frac{N}{4}, -\frac{N}{4})}^{(0)} &  s_{\nn + (-\frac{N}{4}, \frac{N}{4})}^{(0)} &  s_{\nn + (-\frac{N}{4}, -\frac{N}{4})}^{(0)}  \\
s_{\nn + (\frac{N}{4}, \frac{N}{4})}^{(1)} &  s_{\nn + (\frac{N}{4}, -\frac{N}{4})}^{(1)} &  s_{\nn + (-\frac{N}{4}, \frac{N}{4})}^{(1)} &  s_{\nn + (-\frac{N}{4}, -\frac{N}{4})}^{(1)}  \\
\vdots & \vdots & \vdots & \vdots \\
s_{\nn + (\frac{N}{4}, \frac{N}{4})}^{(N_c-1)} &  s_{\nn + (\frac{N}{4}, -\frac{N}{4})}^{(N_c-1)} &  s_{\nn + (-\frac{N}{4}, \frac{N}{4})}^{(N_c-1)} &  s_{\nn + (-\frac{N}{4}, -\frac{N}{4})}^{(N_c-1)}  \end{array} \right)
\end{align}
have full rank $4$ for each $\nn \in \Lambda_{\frac{N}{2}}$.
\end{corollary}

\begin{proof}
The proof follows directly from the observations  in this section and in Section  \ref{sec:incompletefast}, Remark 3.3,  where the matrices ${\mathbf G}_{\nn} = {\mathbf G}_{(k,\ell)}$ with $k,\ell=-\frac{N}{4}, \ldots , \frac{N}{4}-1$ need to be invertible to recover the components of $\mm$.
\end{proof}

\section{Numerical Results}
\label{sec:num}

 In this section we test the Algorithms \ref{alg2}-\ref{alg4} with parallel MRI data from two different brain images. 
The data sets brain$\_8$ch and brain$\_32$ch (using only 8 coils) are in the ESPIRiT toolbox and have been also used to in \cite{ESPIRIT}.
The  magnitudes of the coil images  $|\mathcal{F}^{-1}\yy^{(j)}|$ for these data sets
are illustrated in Figure \ref{coils}.

\begin{figure}[h!]
\begin{center}\hspace*{-5mm}
	\includegraphics[width=19mm,height=19mm]{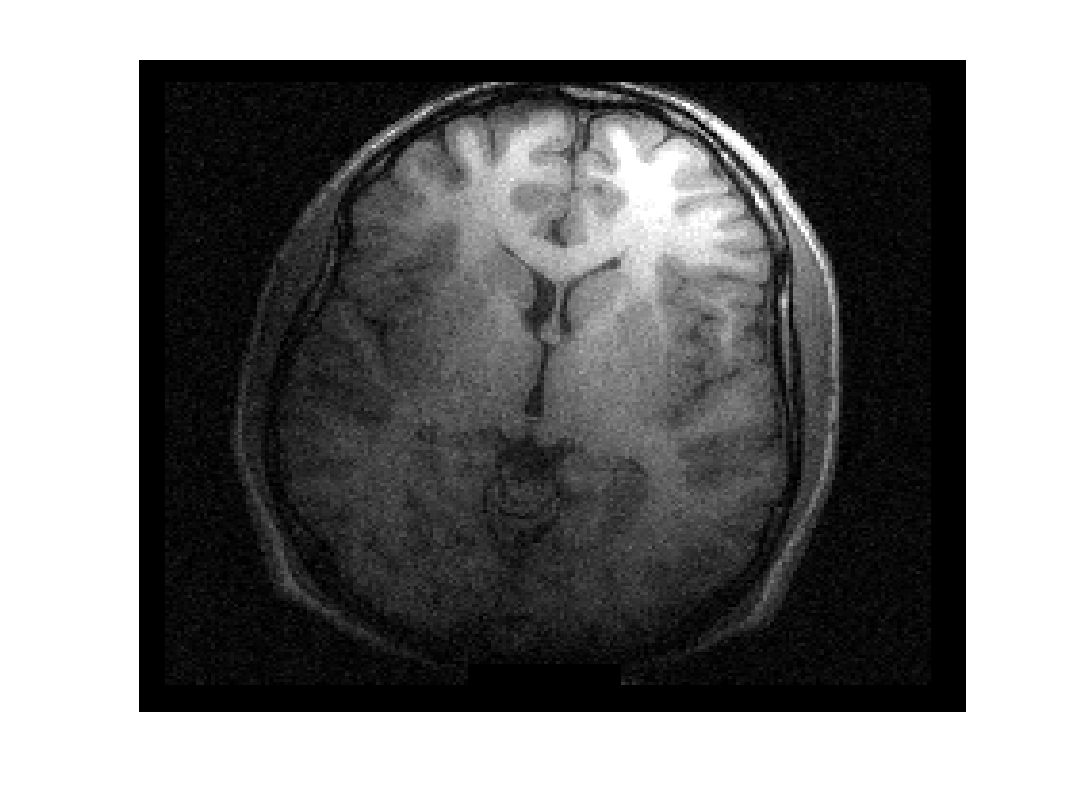} \hspace*{-5mm}
	\includegraphics[width=19mm,height=19mm]{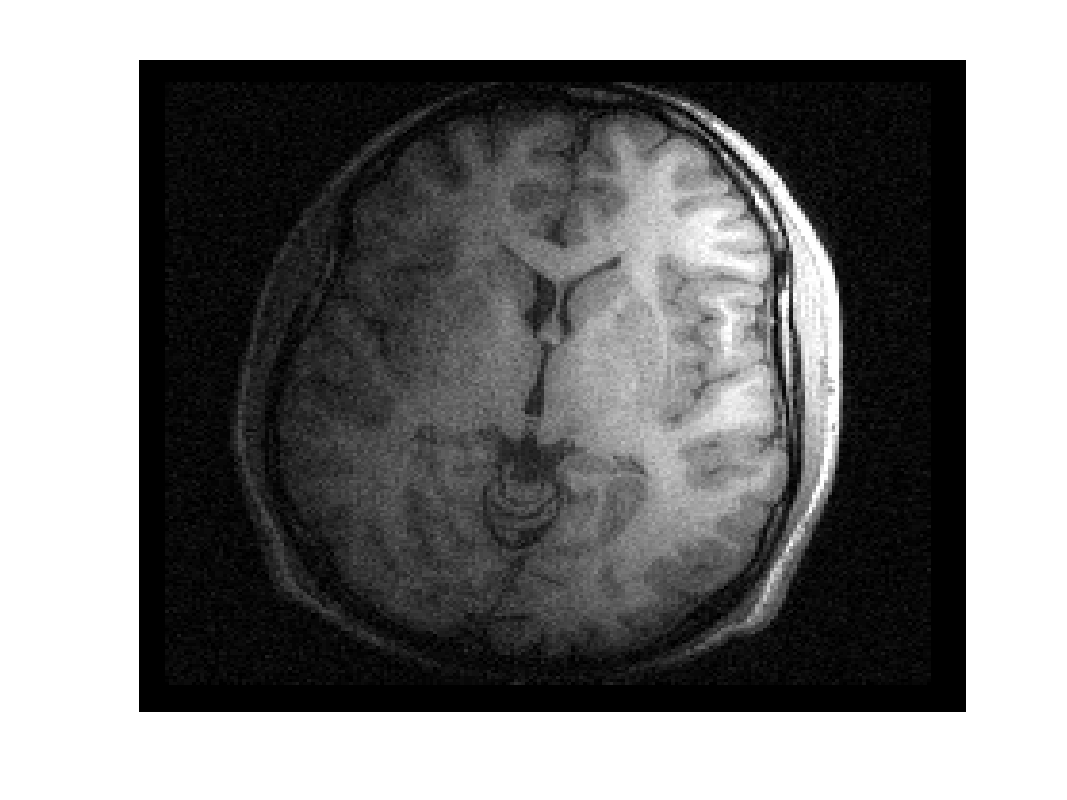} \hspace*{-5mm}
	\includegraphics[width=19mm,height=19mm]{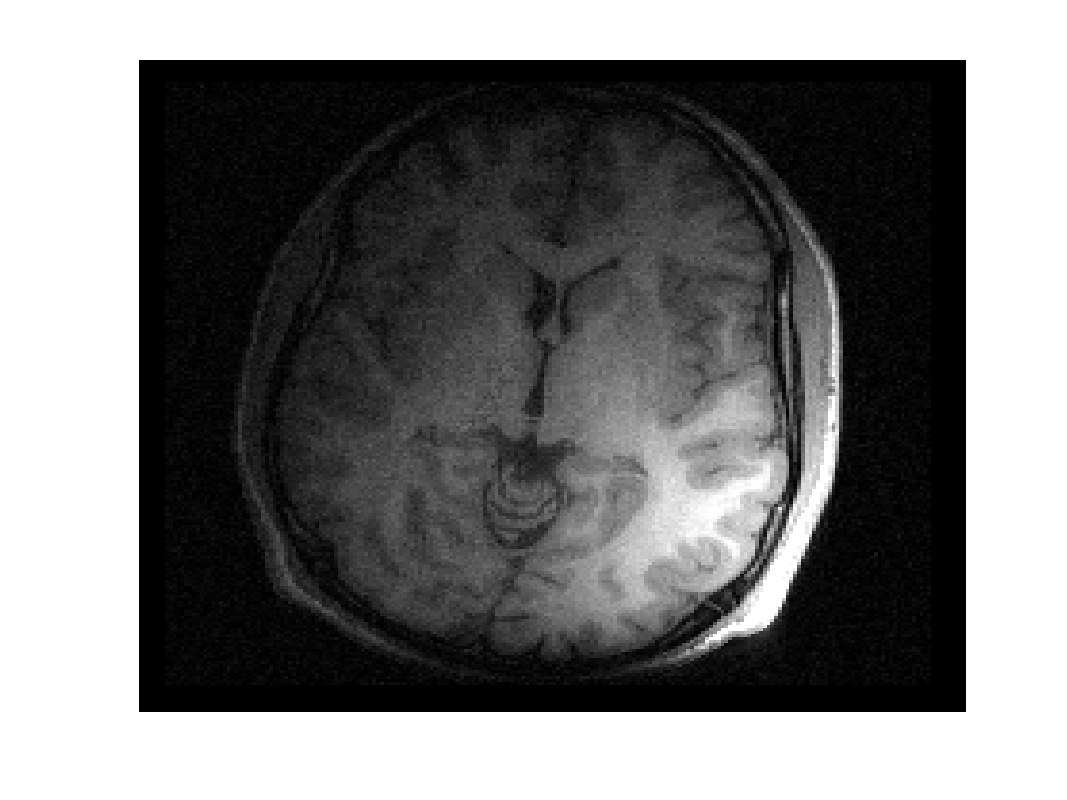} \hspace*{-5mm}
	\includegraphics[width=19mm,height=19mm]{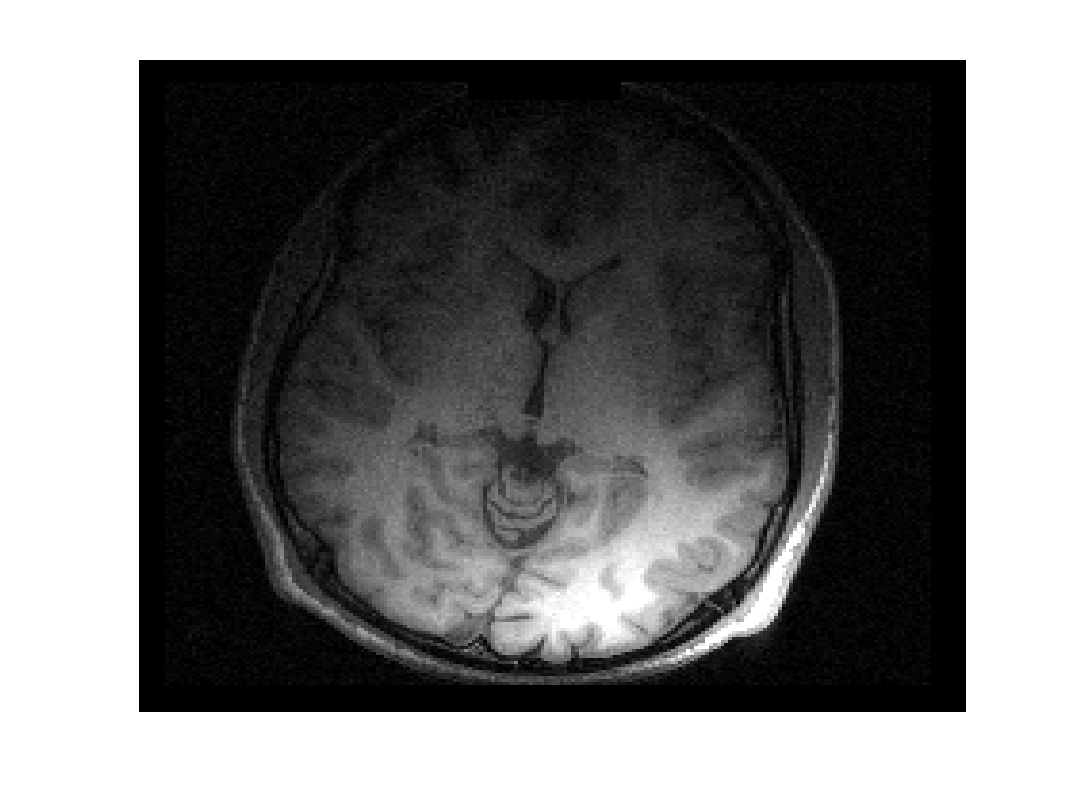} \hspace*{-5mm}
	\includegraphics[width=19mm,height=19mm]{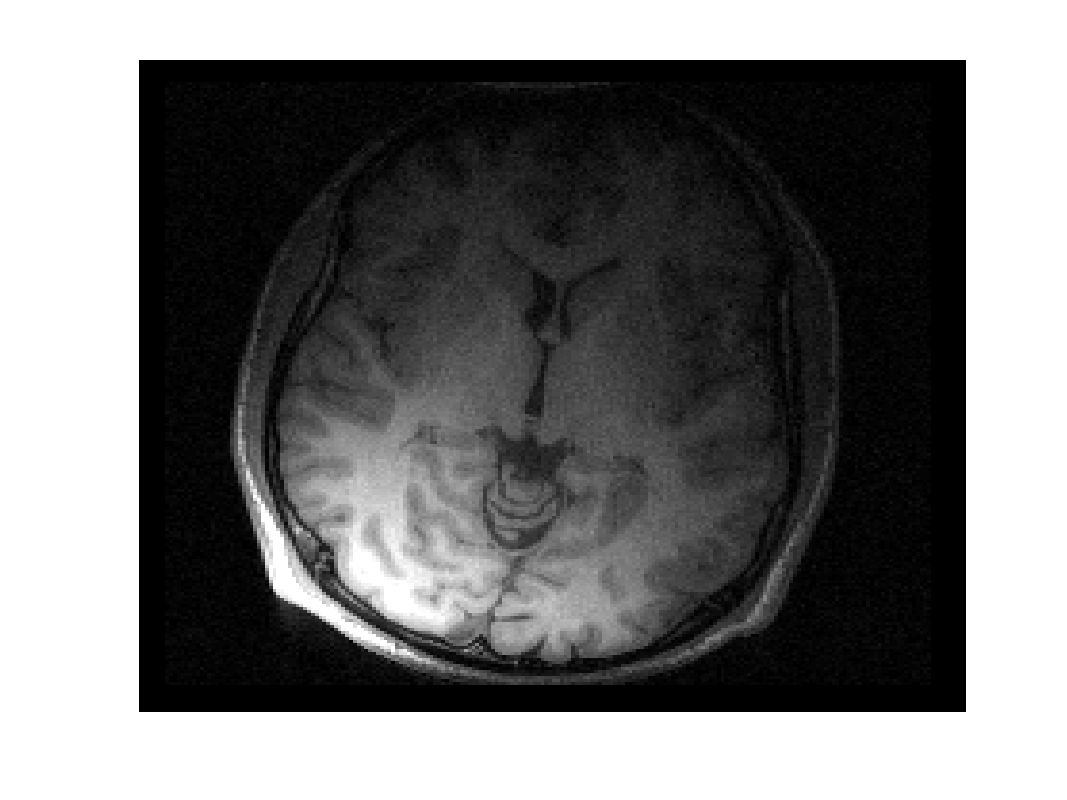} \hspace*{-5mm}
	\includegraphics[width=19mm,height=19mm]{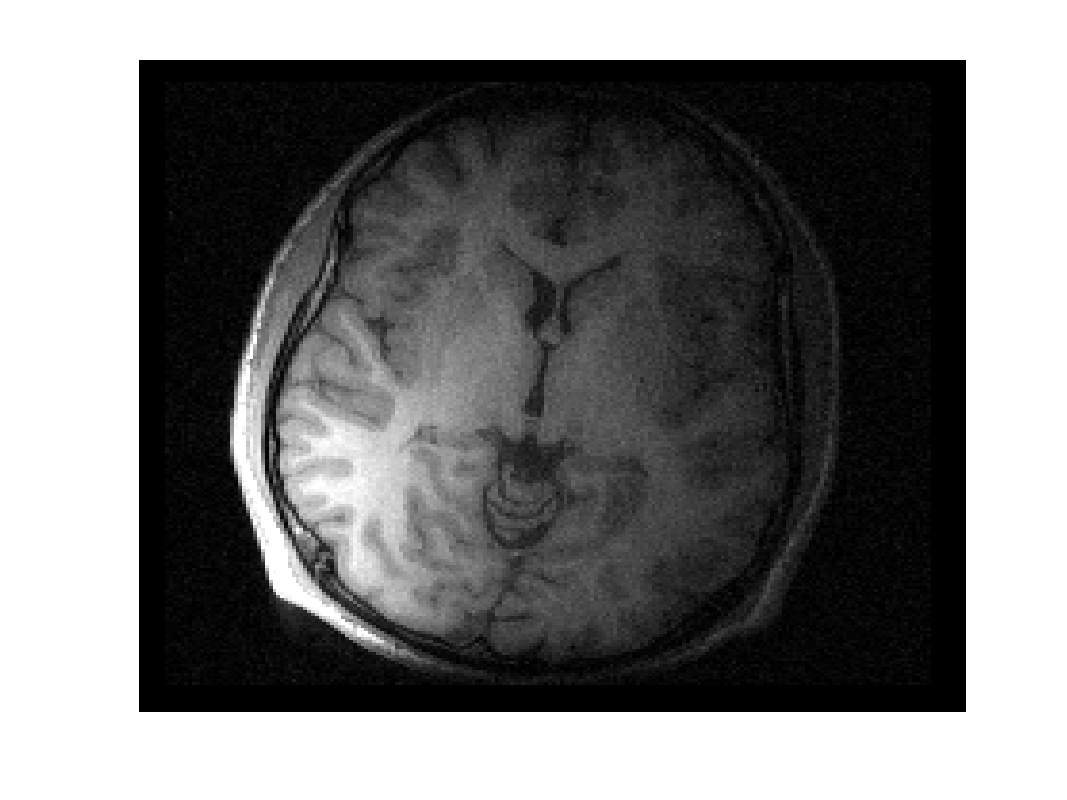} \hspace*{-5mm}
	\includegraphics[width=19mm,height=19mm]{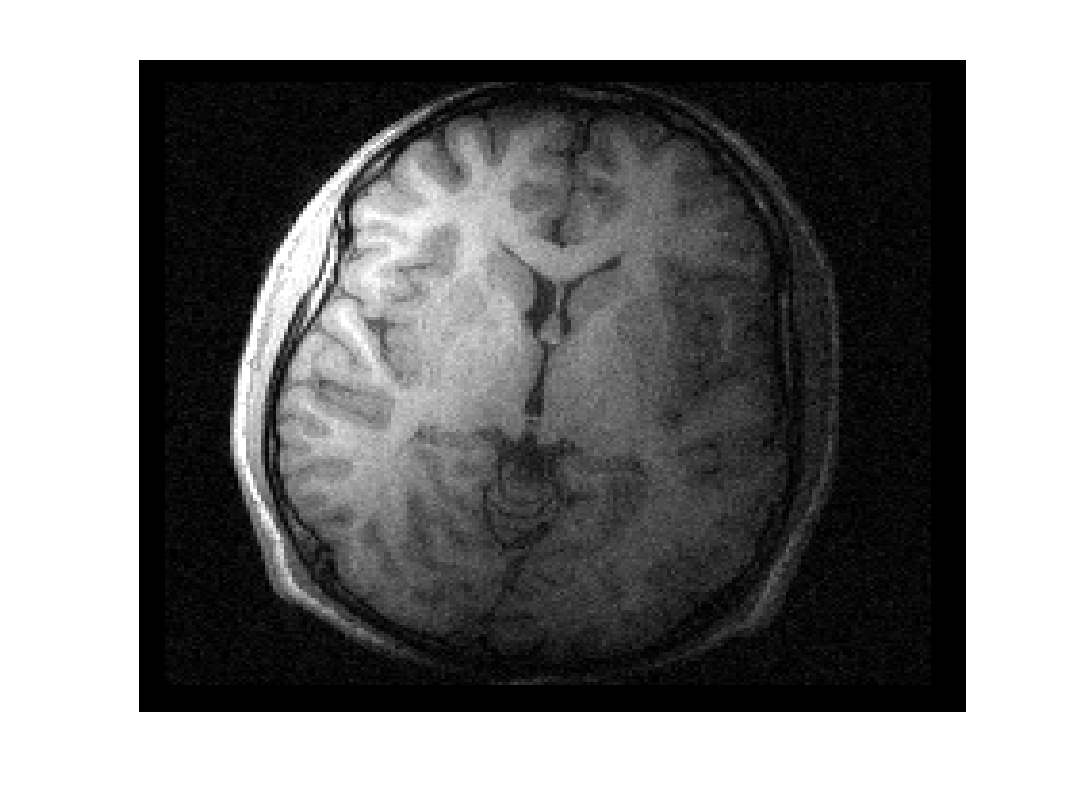} \hspace*{-5mm}
	\includegraphics[width=19mm,height=19mm]{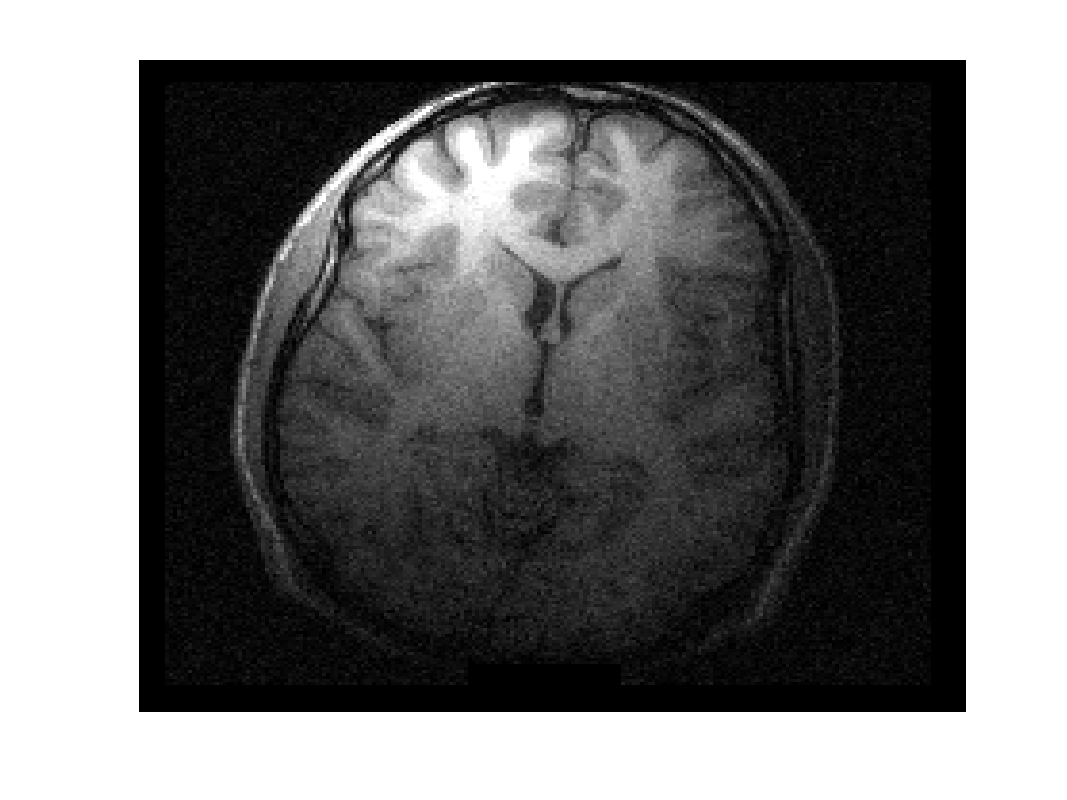} \hspace*{-5mm}
	\includegraphics[width=19mm,height=19mm]{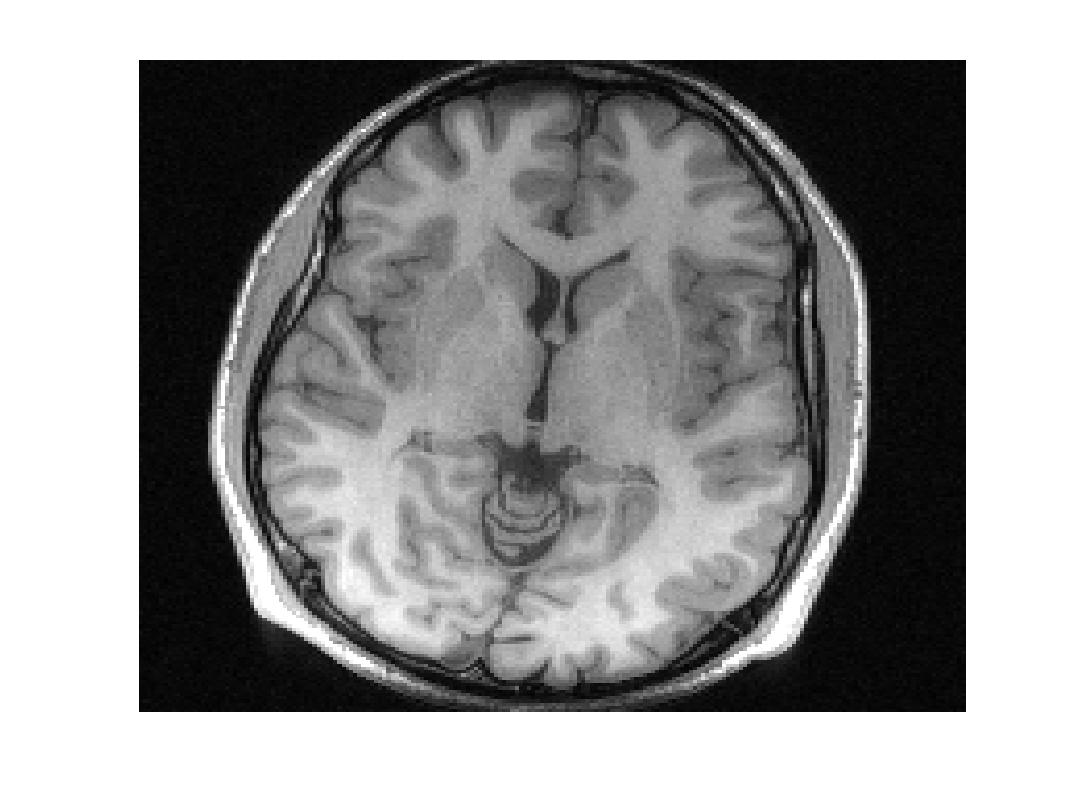} \hspace*{-5mm}
\end{center}
\begin{center}\hspace*{-5mm}
	\includegraphics[width=19mm,height=19mm]{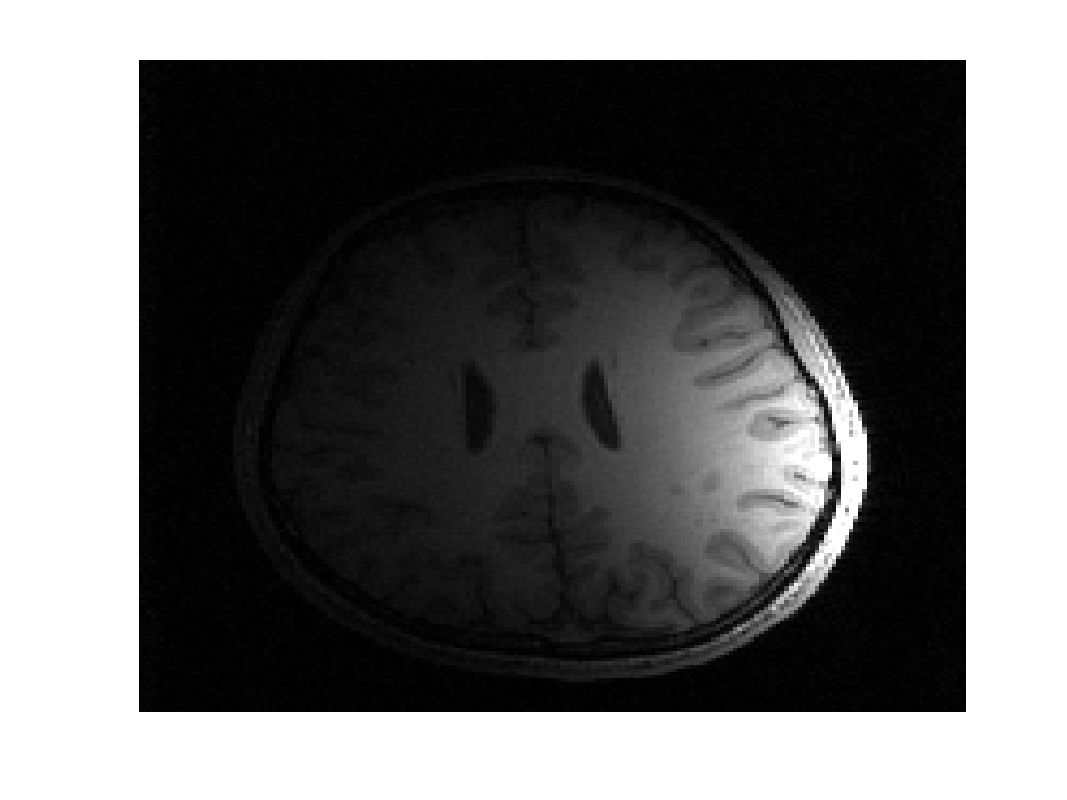} \hspace*{-5mm}
	\includegraphics[width=19mm,height=19mm]{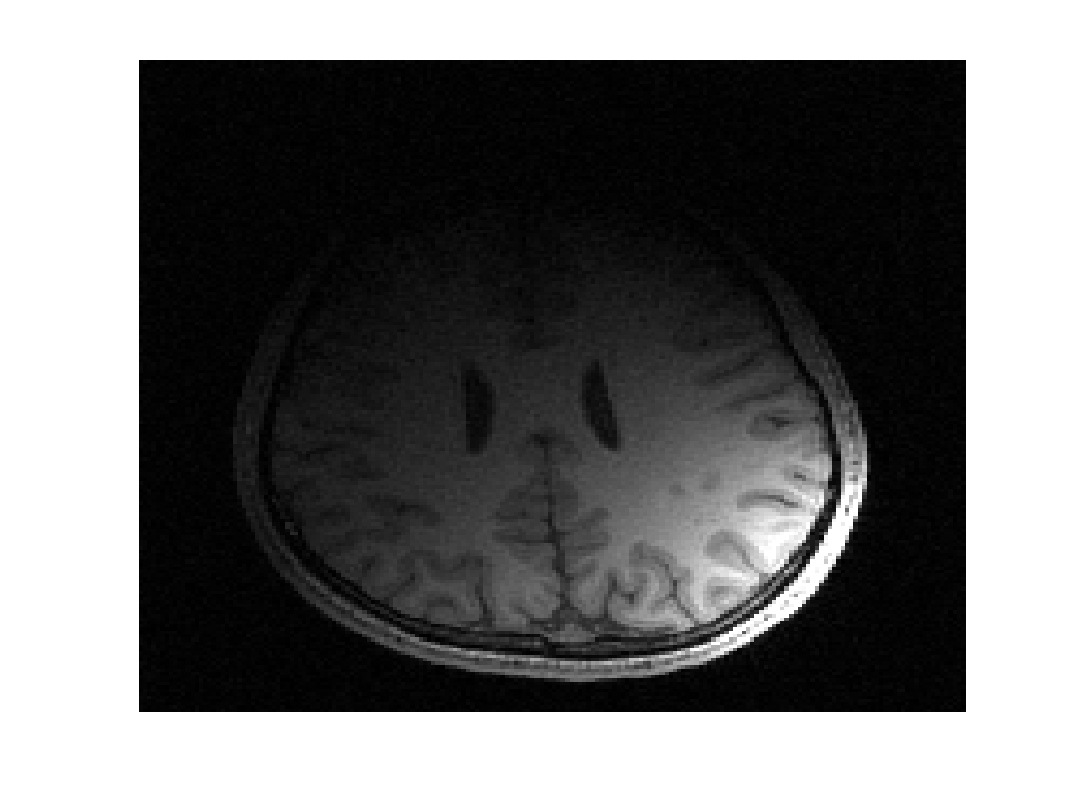} \hspace*{-5mm}
	\includegraphics[width=19mm,height=19mm]{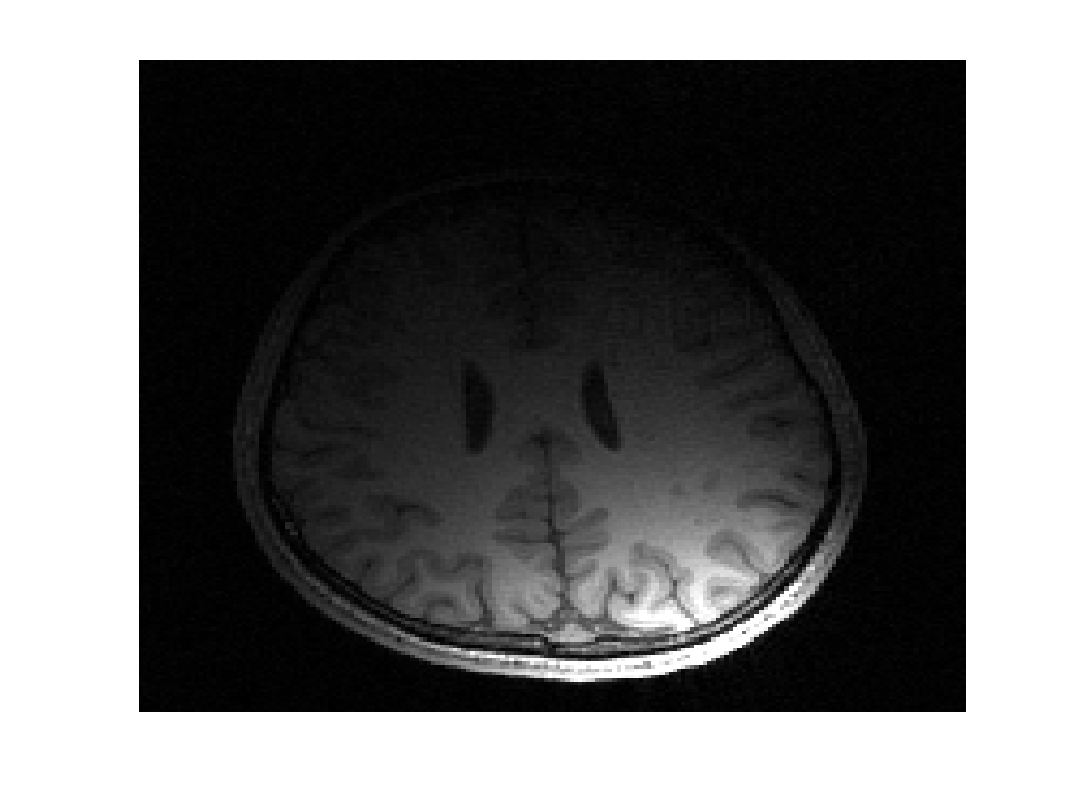} \hspace*{-5mm}
	\includegraphics[width=19mm,height=19mm]{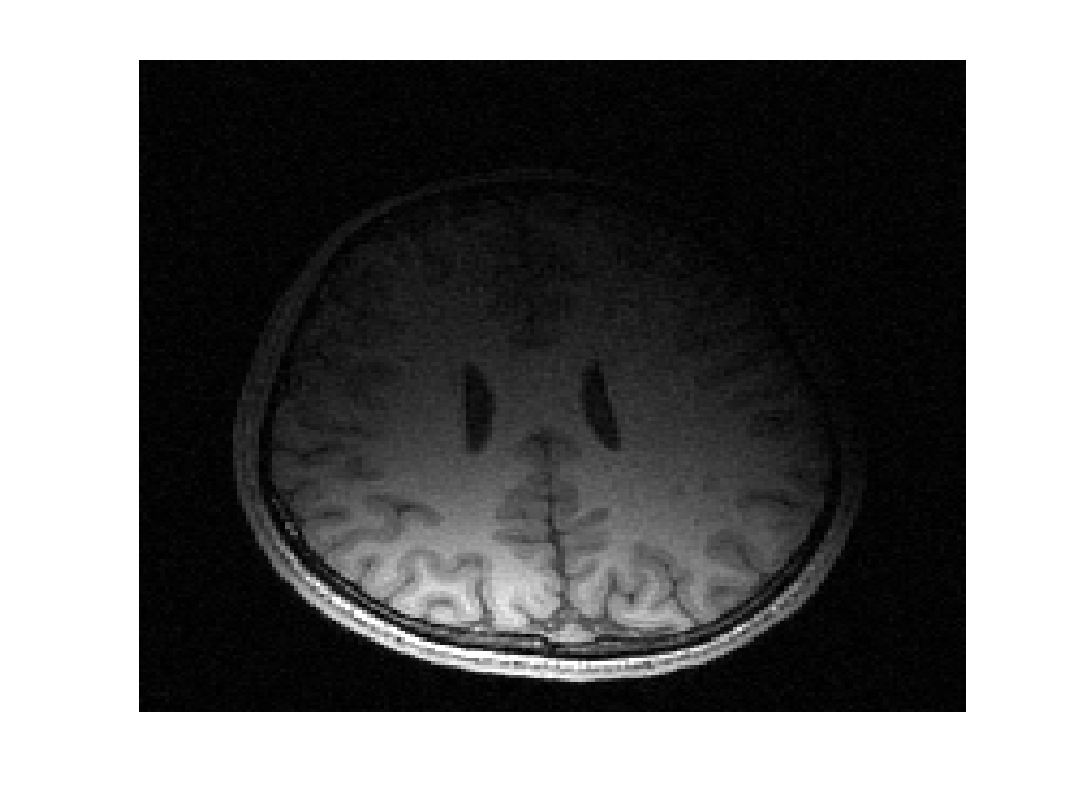} \hspace*{-5mm}
	\includegraphics[width=19mm,height=19mm]{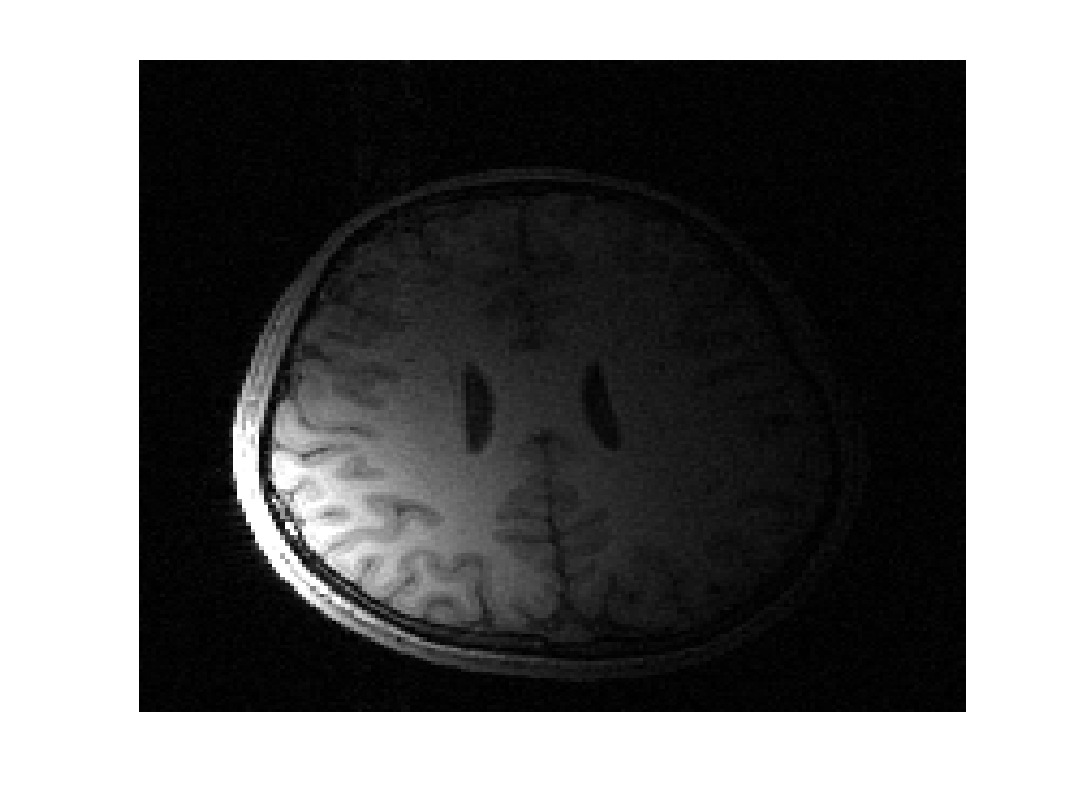} \hspace*{-5mm}
	\includegraphics[width=19mm,height=19mm]{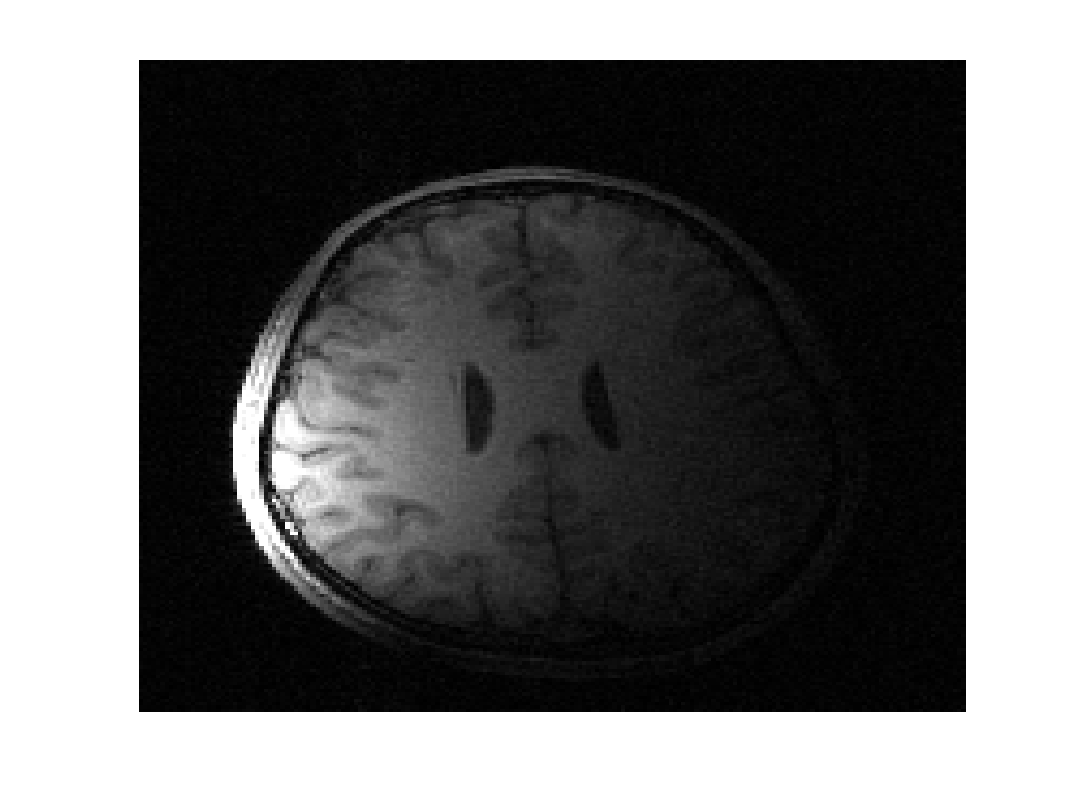} \hspace*{-5mm}
	\includegraphics[width=19mm,height=19mm]{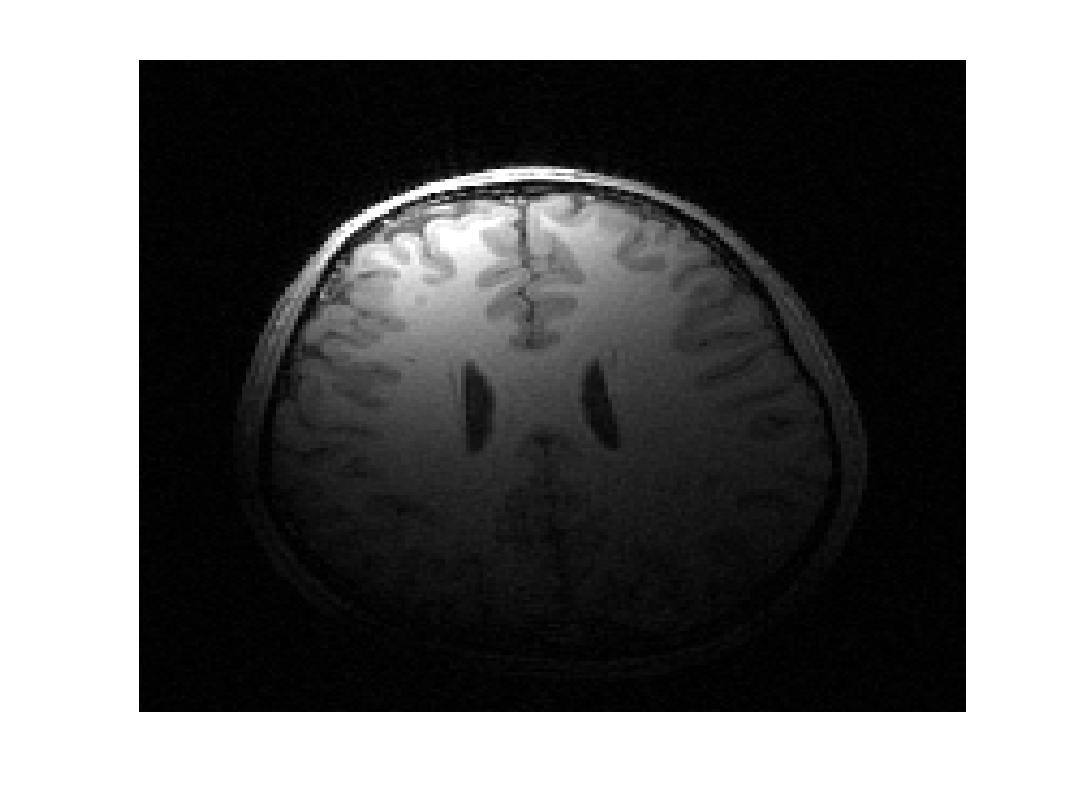} \hspace*{-5mm}
	\includegraphics[width=19mm,height=19mm]{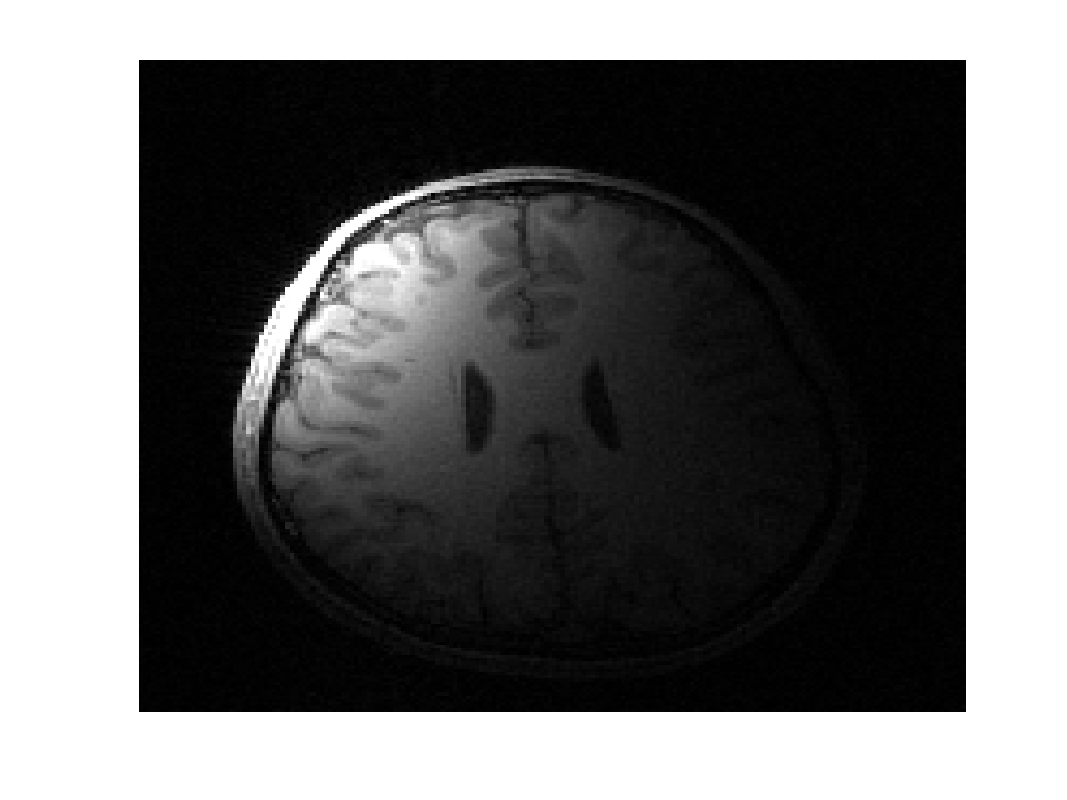} \hspace*{-5mm}
	\includegraphics[width=19mm,height=19mm]{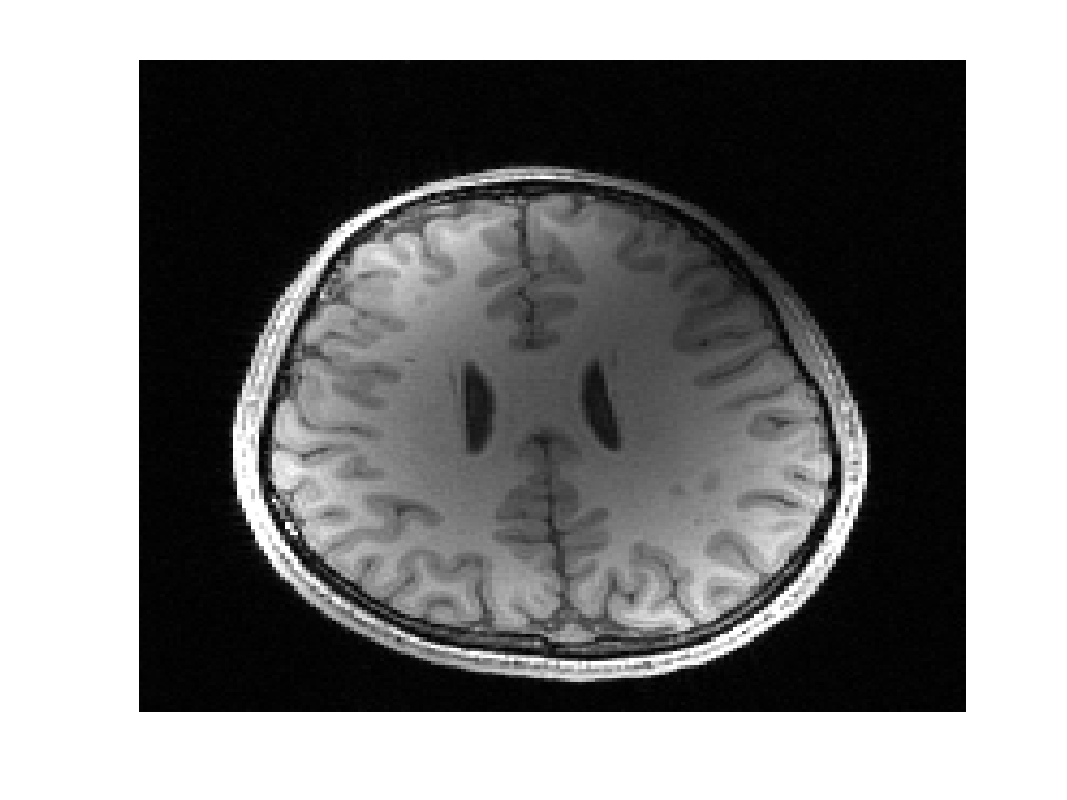} \hspace*{-5mm}
\end{center}
\caption{Magnitudes of the coil images for the two used test data sets together with the full sos reconstruction at the end.}
\label{coils}
\end{figure}

The $8$ coil images in the first data set are of size $200 \times 200$, i.e., $N=200$. From  the 
second data set we have taken $8$ coils of size $192 \times 192$.
We compare the reconstruction results of our MOCCA algorithm with the results of GRAPPA, SPIRiT, ESPIRiT, L1-ESPIRiT and JSENSE using the peak-signal-to-noise ratio (PSNR) and the structural similarity index measure (SSIM) (using the internal functions in Matlab), where the sos reconstruction from complete $k$-space data is taken as the ground truth.

The performance of the MOCCA  algorithm depends on the chosen support $\Lambda_L$ of the trigonometric polynomials determining the coil sensitivities.  If we take the iterative Algorithm \ref{alg3} to solve the least squares problem, we can profit from using a suitable number of  iteration steps. 
Further, the results of the MOCCA Algorithm \ref{alg2} can be improved by employing the smoothing procedure  proposed in Section \ref{smoothing}.
The size of the calibration area is always taken similarly as for the other algorithms. 

\noindent
 We start with reconstructions from $k$-space data of the first data set, where outside of the ACS region the following scheme for acquired incomplete data with reduction rate $R$ is taken:\\
every second column, $R=2, \,  (1, 1/2)$,\\
every third column, $R=3, \,  (1, 1/3)$,\\  
every fourth column, $R=4, \, (1,1/4)$, \\
every second row and every second column, $R=4,\,  (1/2, 1/2)$, \\
every  second row and every third column, $R=6,\,  (1/2, 1/3)$. \\
The results are summarized in Table 1.

For  MOCCA we have used $\beta=0.001$ and  a fixed number of iterations in Algorithm \ref{alg3}. 
Further, the support $\Lambda_L$ of the coil sensitivities  in $k$-space is given by 
$\Lambda_5= \{-2, -1,0,1,2\} \times \{-2, -1,0,1 , 2\}$, i.e., $L^2=25$ parameters are used per coil sensitivity. 
In Figure \ref{figsing1}(left), we  display the singular values of the MOCCA-matrix ${\mathbf A}_M$ in (\ref{AM}) obtained for the first data set with $L=5$. The decay of the singular values of ${\mathbf A}_M$ shows that, corresponding to the theoretical observations in Sections \ref{sec:incomplete} and \ref{sec:ana}, the smallest singular value of ${\mathbf A}_M$ is close to $0$. 
For this first data set and $L=5$, the  singular vector corresponding to the smallest singular value 
can be directly taken to reconstruct  the coil sensitivities. 
In our results given in Table \ref{tab:1}, we have used a linear combination of two singular vectors ${\mathbf c} = {\mathbf c}_1 +\alpha {\mathbf c}_2$ in order to improve the reconstruction result,
where  ${\mathbf c}_1$ corresponds to the smallest singular value $0.1537$ and ${\mathbf c}_2$  to  the second smallest value $0.1803$. For $R=2,3,4$ we have employed $\alpha=0.2$ and for $R=6$, we used $\alpha={\mathrm i}$.

In Table \ref{tab:1} we present the  results for MOCCA and  MOCCA-S (MOCCA with a smoothing step as in Section \ref{smoothing})  obtained by  Algorithm \ref{alg4} and by the iterative Algorithm \ref{alg3}, where the employed number of iterations depends on $R$ and
is indicated in Table \ref{tab:1} with the numbers in brackets. For example, for $R=2$, we used 12 iterations. 
For the smoothing step, we used $\lambda=0.00045$ for $R=2$, $\lambda=0.0018$ for $R=3$ and $R=4\,  (1,\frac{1}{4})$, $\lambda=0.0015$ for $R=4 \, (\frac{1}{2},\frac{1}{2})$, and $\lambda= 0.0035$ for $R=6$.
Interestingly, the recovery results for the direct fast MOCCA Algorithm  \ref{alg4} outperform the iterative  MOCCA Algorithm \ref{alg3}  for the cases $R=2$, $R=3$, and works equally well for $R=4 (\frac{1}{2}, \frac{1}{2})$, while being much worse for the other cases. We argue that, since the direct algorithm solves small independent linear systems, and does not directly exploit the full data knowledge in the ACS region, local model inconsistencies in the data directly lead to larger local solution errors for higher reduction factors. Note that the performance of MOCCA further improves for larger $k$-space support of sensitivities, i.e.,  $L>7$.

For GRAPPA and SPIRiT we applied the  implementation {\tt demo\_l1\_spirit\_pocs.m} in the ESPIRiT Matlab toolbox 
with parameters given there for this data set, i.e., SPIRiT kernel size $kSize = [5,5]$, 
$nIter = 20$ iterations,
Tykhonov regularization parameter $CalibTyk = 0.01$  in the calibration, and 
$wavWeight = 0.0015$ for wavelet soft-thresholding regularization in the reconstruction.
For ESPIRiT, we have taken {\tt demo\_ESPIRiT\_recon.m} from the ESPIRiT toolbox with 
$ncalib = 24$ ($24 \times 24$ calibration area), kernel window size $ksize = [6,6]$,  $eigThresh\_k = 0.02$, 
$eigThresh\_im = 0.9$, and two sets of eigenmaps.
For L1-ESPIRiT, we used {\tt demo\_ESPIRiT\_L1\_recon.m} from the ESPIRiT toolbox, which had been created for  this data set with the ESPIRiT parameters as above and the setting for the $L1$-reconstruction with $nIterCG = 5$  CG iterations for the PI part,
$nIterSplit = 15$ splitting iterations for CS part,
$splitWeight = 0.4$, and  
$lambda = 0.0025$ for $R=2$ and optimized parameters $splitWeight = 0.2$, $lambda = 0.7$ for $R>2$.
For JSENSE we used the software  by the authors of \cite{JSENSE} and tried to adapt the number of iterations for alternating minimization as well as the  polynomial degree to approximate the sensitivities for optimal reconstruction results  in Table \ref{tab:1}. The original implementation only works, if the  image dimension is  a multiple of the considered reduction rate, and if the data in only one direction are non-acquired. For $R=3$ we have therefore adapted the number of columns on the coil data ${\mathbf y}^{(j)}$ from $200$ to $198$. Unfortunately, due to the required alternating minimization steps, JSENSE is more time consuming than the other algorithms.

\begin{table}[htbp]
\scriptsize
\caption{Comparison of the reconstruction performance for the incomplete data from the first data set.}\label{tab:1}
\begin{center}
\begin{tabular}{rrrrrrr}
\hline
method & measure & $R=2, (1,\frac{1}{2})$ & $R=3, (1,\frac{1}{3})$  & $R=4, (1,\frac{1}{4})$ & $R=4, (\frac{1}{2},\frac{1}{2})$ & $R=6, (\frac{1}{2},\frac{1}{3})$ \\
\hline
GRAPPA & PSNR &  {41.4367} & 36.2093 & 30.3212 &34.8741 &  28.9493 \\
 & SSIM & {0.9638} & 0.9039 & 0.7561 & 0.8821 &  0.7236 \\
SPIRiT & PSNR & 26.7296 & 28.3496 & 26.7414 & 30.8745 &  29.6801 \\
  & SSIM & 0.9441 & 0.8794 & 0.7192 & 0.8871 &  0.7307 \\
ESPIRiT & PSNR & 37.2218  & 35.1977 & 32.0629 & 35.4497 & 32.6245 \\
 & SSIM & 0.8138  & 0.7808 & 0.7022 & 0.7848 &  0.7164 \\
L1-ESPIRiT & PSNR & 37.3810  
& 35.9555 &  \textbf{34.1941} & 35.9088 & 33.5887 \\
 & SSIM &  0.8279 
& 0.7689 & 0.7612 & 0.7714 & 0.7589 \\
JSENSE & PSNR & 33.5558 & 34.0555 & 30.8670 & & \\
 & SSIM & 0.8723 & 0.8665 & 0.7818 & & \\
 \hline
MOCCA & PSNR & (12)38.7136& (40)35.1875  & (75)32.0755 & (50)35.6563 & (90)32.0203 \\
(L=5) & SSIM & 0.9119 & 0.8795 &  0.8111 & 0.8896 & 0.7995\\
MOCCA-S & PSNR & (12)39.4055 & (40){36.7334} & (75){33.2165} & (50)\textbf{37.2826} & (90)\textbf {33.6516} \\
(L=5) & SSIM          & 0.9241           & {0.9301}      & \textbf{0.8845} & \textbf{0.9354}       & \textbf{0.8897}\\
\hline
MOCCA & PSNR     & 41.3746.        & 35.1676               & 29.1279 & 35.0675         & 27.8404 \\
direct (L=5) & SSIM & 0.9643          & 0.8859                  &  0.6491 & 0.8815            & 0.6078\\
MOCCA-S & PSNR & \textbf{42.1886}        & \textbf{37.4517}    & 32.2398& 36.7925           & 29.3610 \\
direct (L=5) & SSIM & \textbf{0.9717} & \textbf{0.9369}              & 0.7876   & {0.9294} & 0.7119\\
 \hline
\end{tabular}
\end{center}
\end{table}

In Figure \ref{figure_Bild1}, we exemplarily represent the reconstruction results for the reduction factor $R=3$ (every third column acquired) and the corresponding error maps. Further, in  Figure \ref{fig:coil1}, we illustrate the magnitude and the phase of the 8 normalized coil sensitivities $\tilde{\mathbf s}^{(j)}$ obtained for the first data set.  Before pointwise multiplication with $\sign(m_{\nn})$, the sensitivities have smooth magnitude and phase, since they are samples of pointwisely normalized trigonometric polynomials of small degree, i.e., samples of trigonometric rational functions.

\begin{figure}[h!]
\begin{center}
        \includegraphics[width=28mm,height=28mm]{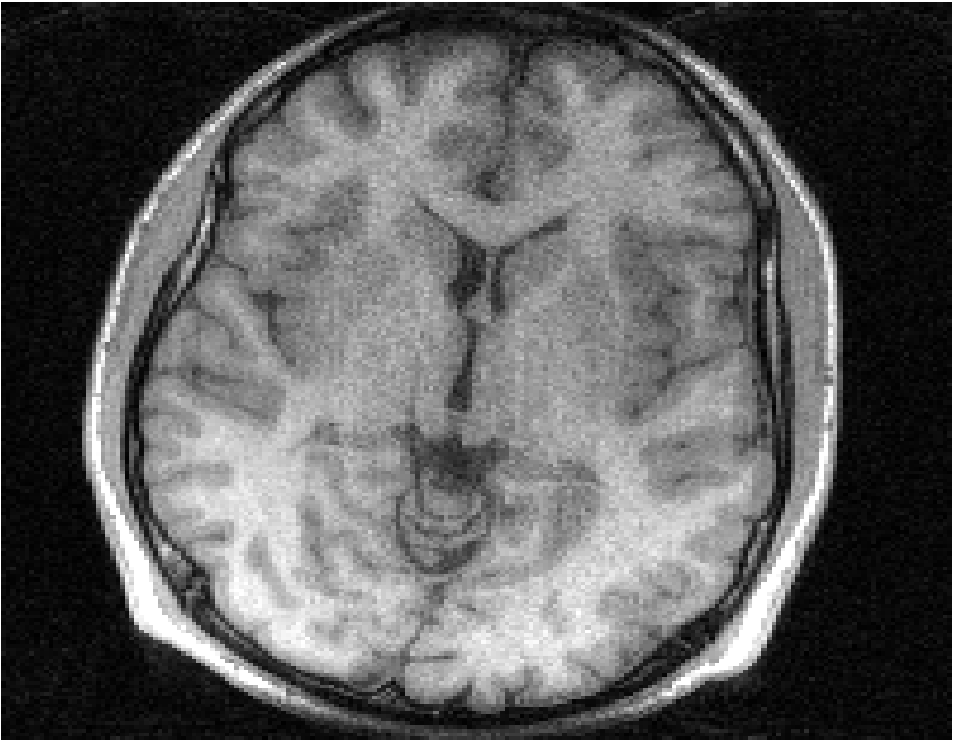}
	\includegraphics[width=28mm,height=28mm]{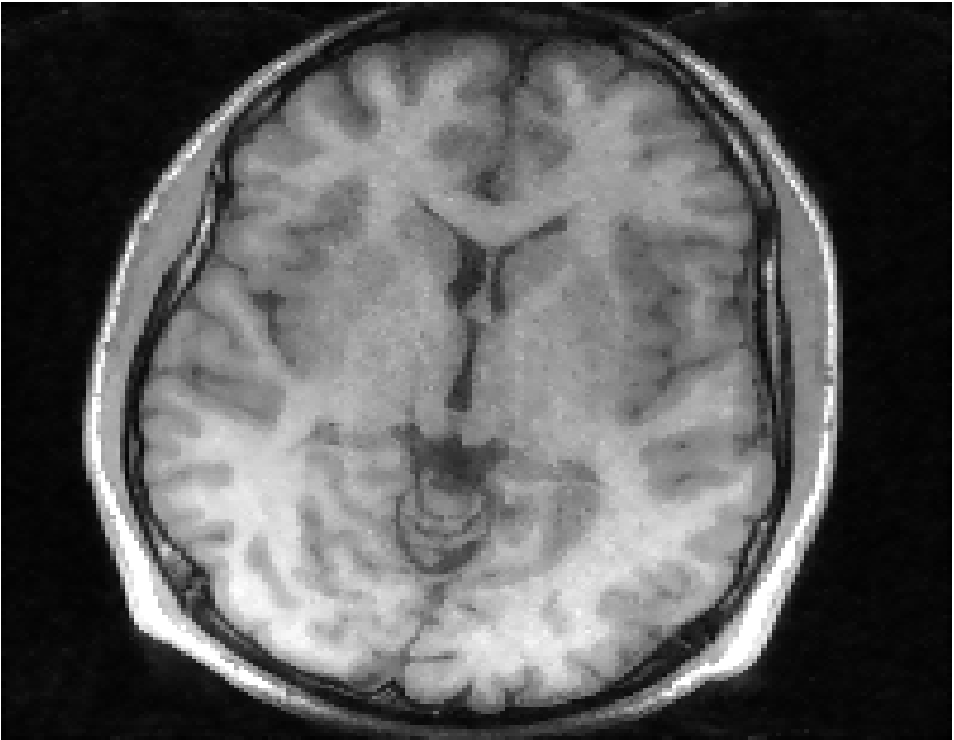}
	\includegraphics[width=28mm,height=28mm]{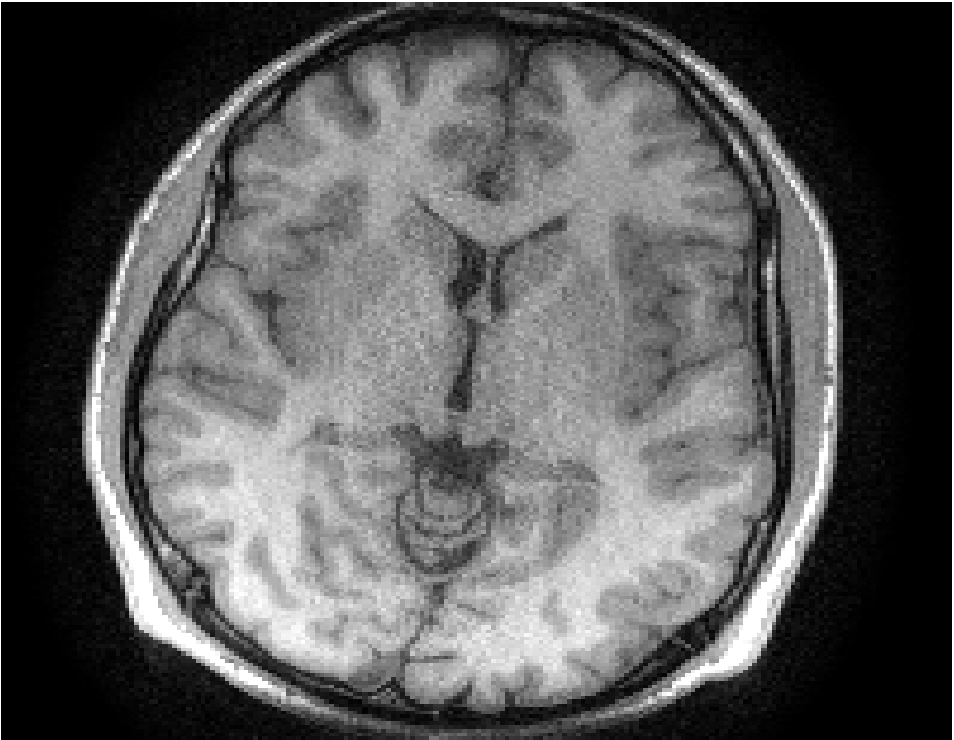}
	\includegraphics[width=28mm,height=28mm]{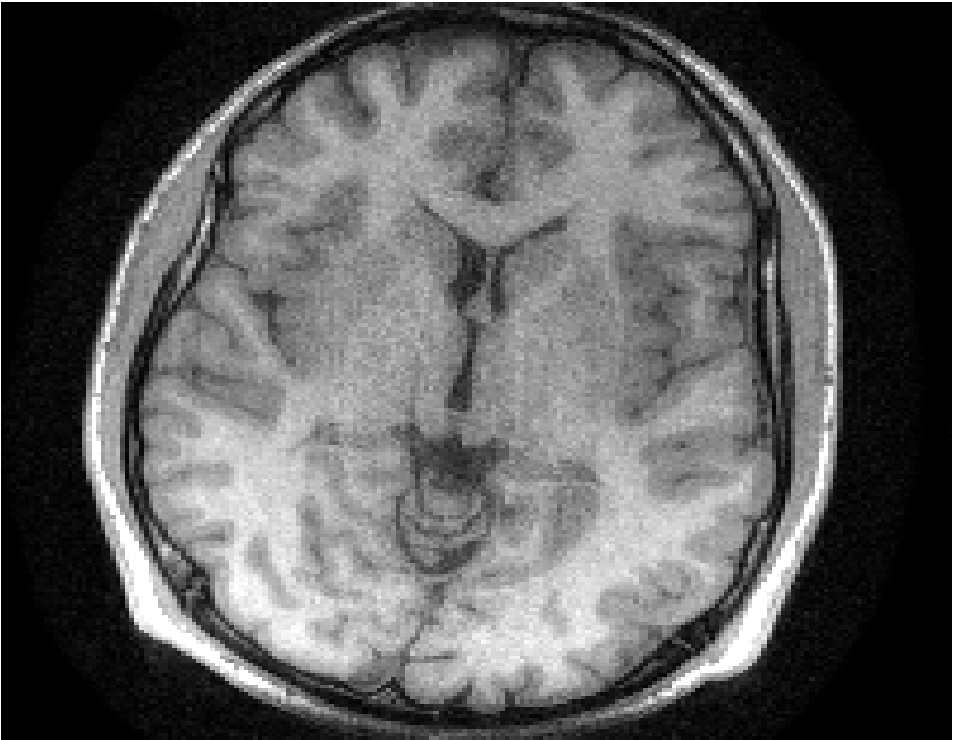}
	\includegraphics[width=28mm,height=28mm]{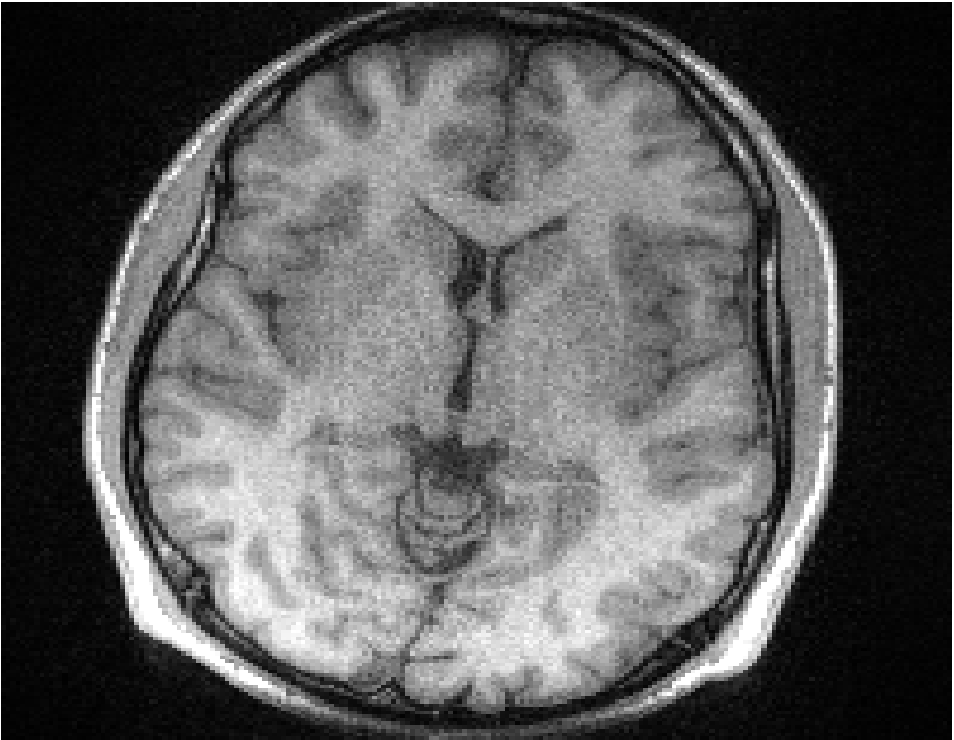}\\
        \includegraphics[width=28mm,height=28mm]{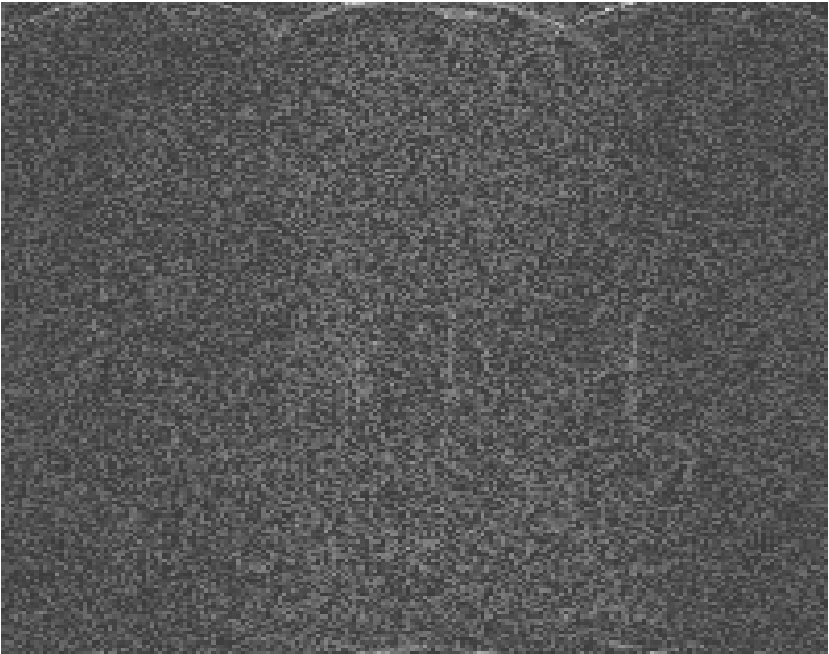} 
	\includegraphics[width=28mm,height=28mm]{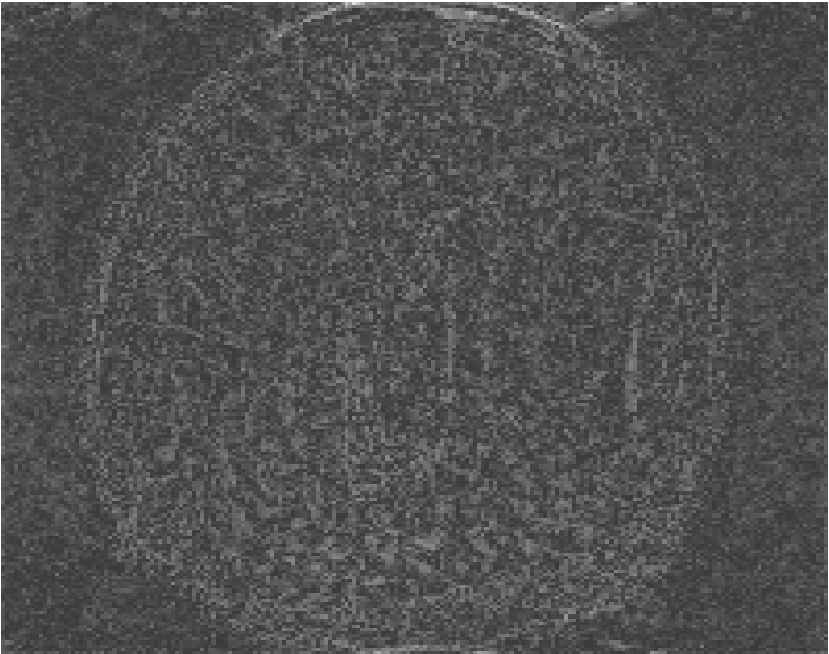}
	\includegraphics[width=28mm,height=28mm]{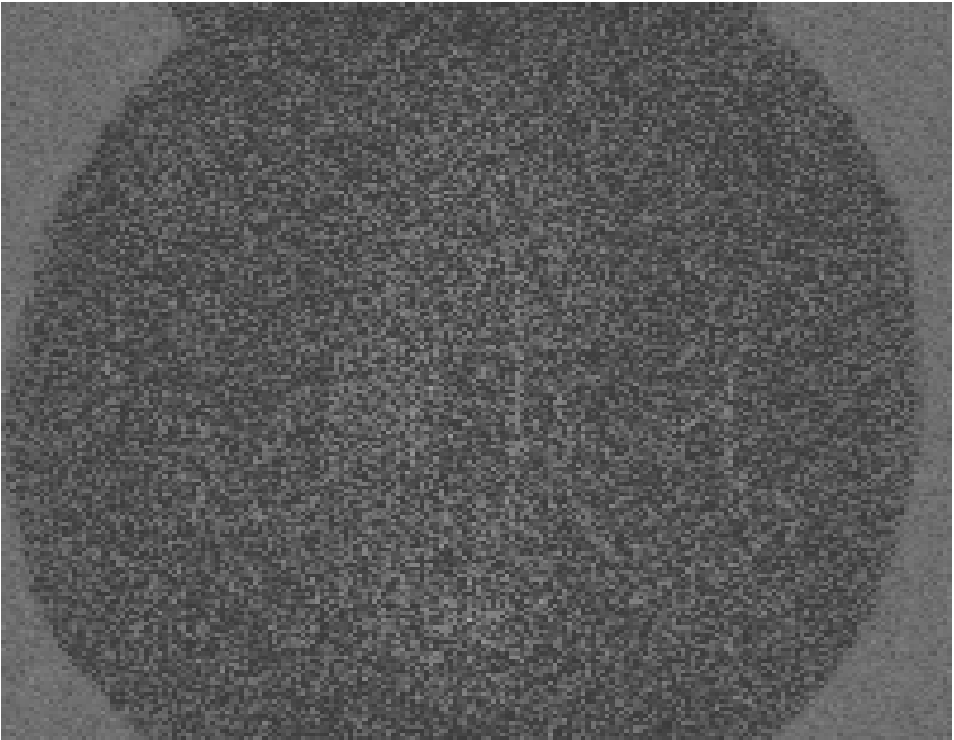}
	\includegraphics[width=28mm,height=28mm]{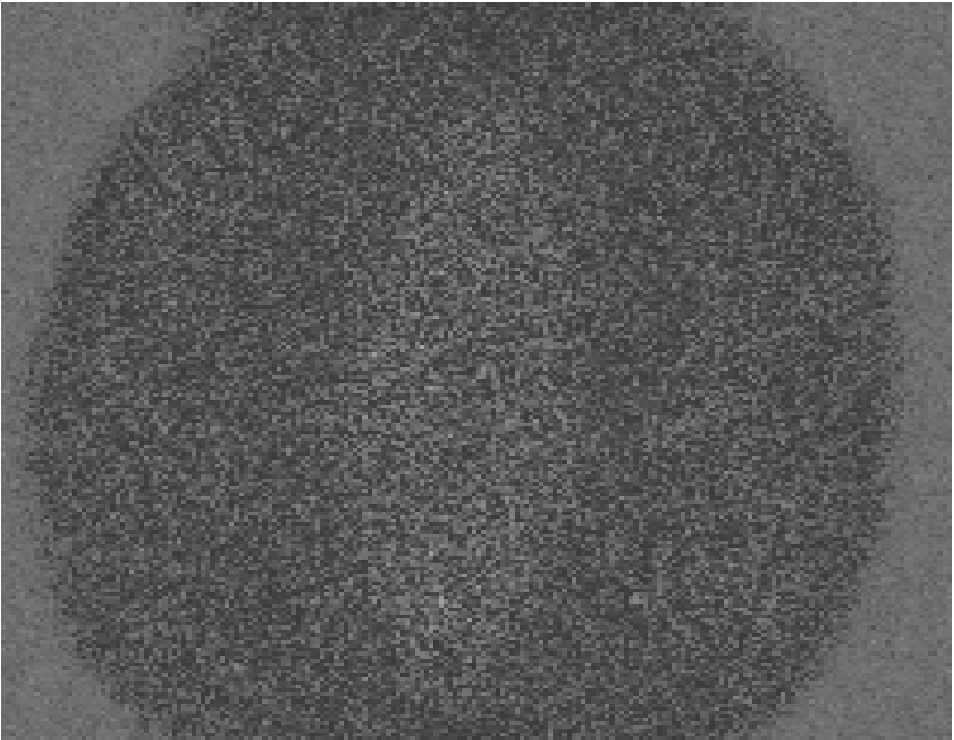}
	\includegraphics[width=28mm,height=28mm]{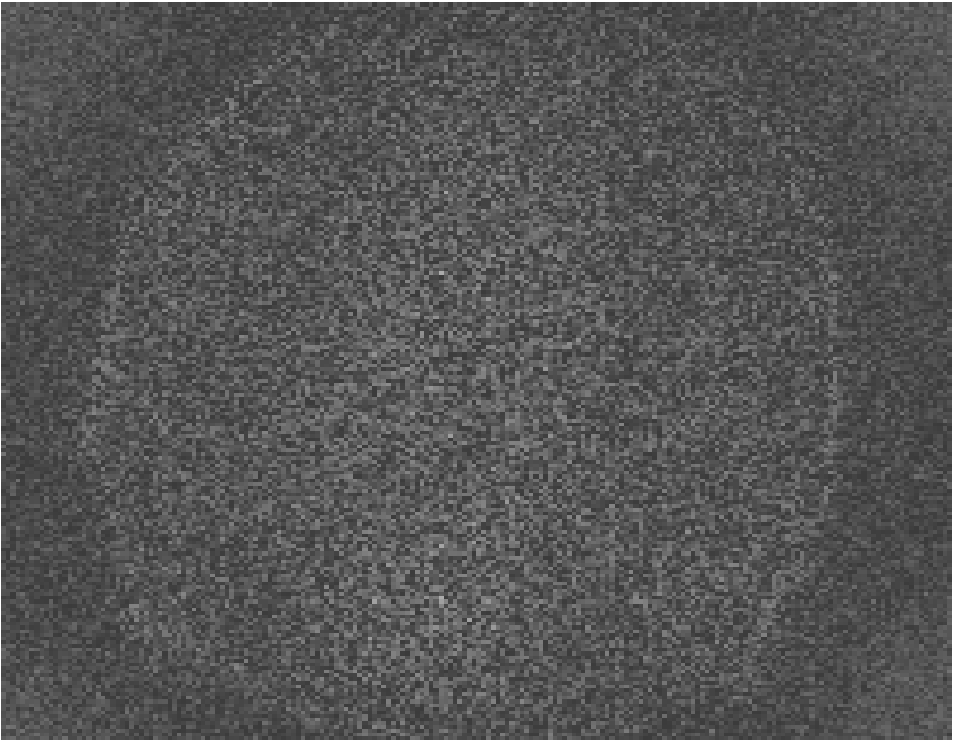}
	\end{center}
\caption{Reconstruction results for the first data set obtained  from a third of the $k$-space data of 8 coils (every third column acquired) for different methods. From left to right: (1)  MOCCA  using Algorithm \ref{alg2} (taking Algorithm \ref{alg3} with 40 iterations),  (2) MOCCA-S (with 40 iterations and smoothing), (3) L1-ESPIRiT, (4) ESPIRiT and (5) GRAPPA. Corresponding  error maps are given  below. All error images use the same scale with relative error in $[0,0.12]$, where $0$ corresponds to black and $0.12$ to white.}
\label{figure_Bild1}
\end{figure}

\begin{figure}[h!]
\begin{center}	
        \includegraphics[width=65mm,height=26mm]{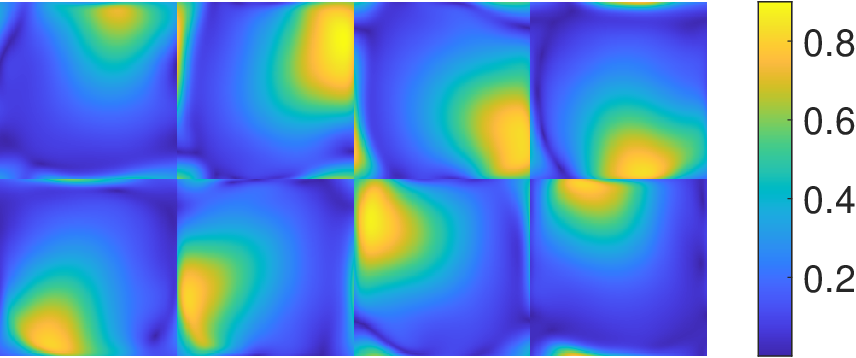} \hspace{5mm}
	\includegraphics[width=65mm,height=26mm]{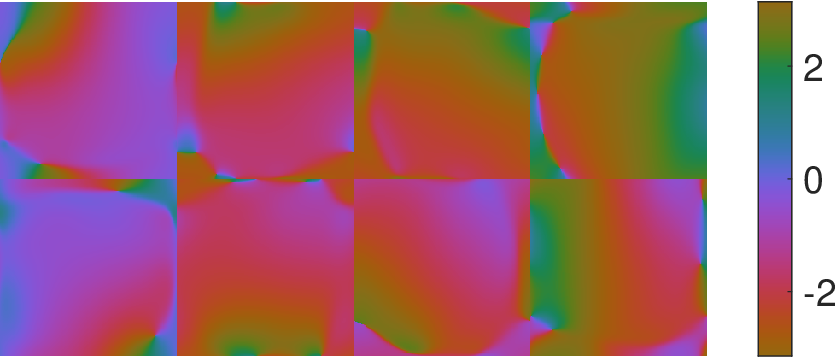}
	 \includegraphics[width=65mm,height=26mm]{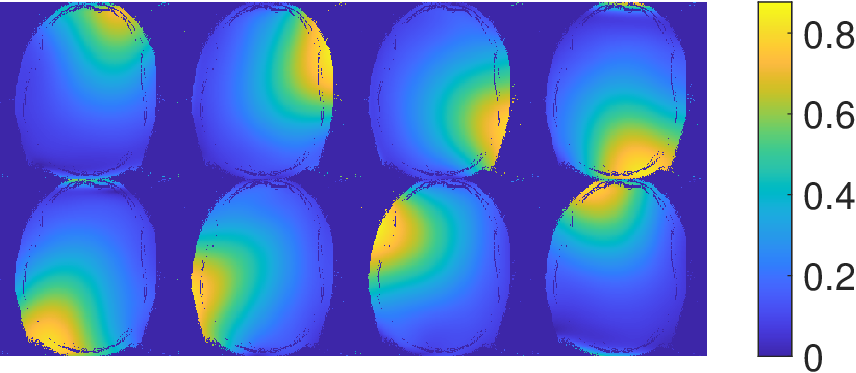} \hspace{5mm}
	\includegraphics[width=65mm,height=26mm]{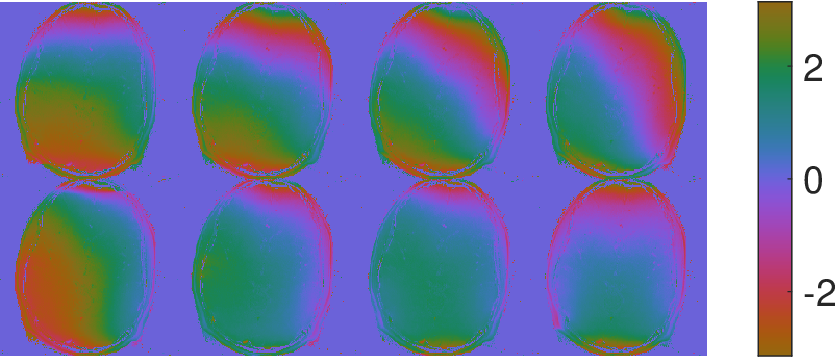}
\end{center}
\caption{Magnitude (left) and phase (in $[-\pi,\pi]$) (right) of the $8$ coil sensitivities obtained for MRI reconstruction for $L=5$ for the first data set before and after multiplication with $\sign(\mm)$, see step 7 of Algorithm \ref{alg2}. The normalized sensitivities  are samples of bivariate trigonometric polynomials with $L^2=25$ nonzero coefficients, pointwisely multiplied with the normalization factors $({\mathbf d}^{+})^{\frac{1}{2}}$ to ensure the sos condition. 
}
\label{fig:coil1}
\end{figure}

For the second data set, we want to illustrate  that the chosen support size of the trigonometric polynomials to evaluate the coil sensitivities plays an important role for obtaining very good reconstruction results. We compare with GRAPPA, ESPIRiT 
with parameters fixed as for the first data set, for L1-ESPIRiT with optimized parameters $splitWeight = 0.1$, $lambda = 0.0002$, 
and also show the results of MOCCA for different support sizes in Table \ref{tab:2}. For MOCCA, the numbers in brackets give the numbers of  iterations in Algorithm \ref{alg3}. 
 As before, MOCCA-S employs the smoothing procedure in Section \ref{smoothing} as a post-processing step, where the parameter $\lambda$ is taken in dependence of $R$. We used $\lambda=0.00015$ for $R=2$, $\lambda=0.0005$ for $R=3$ and $R=4 (\frac{1}{2},  \frac{1}{2})$, 
$\lambda=0.0013$ for  $R=4 (1,  \frac{1}{4})$ and $R=6$. The results MOCCA direct  and MOCCA-S direct are obtained by Algorithm \ref{alg4}. 
As before, Algorithm \ref{alg4} outperforms Algorithm \ref{alg3} for  $R=2, 3$, and $R=4 (\frac{1}{2}, \frac{1}{2})$ and is worse for the other cases. 

Since the parallel MRI data do not exactly fit our model with the chosen support sizes, the MOCCA matrix ${\mathbf A}_M$ obtained from the second data set possesses several very small singular values, see Figure \ref{figsing1}.
While for the first data set with $L=5$ and $M=20$ (i.e., ACS area of size $24 \times 24$), the matrix
${\mathbf A}_{M}$ contains only two singular values being smaller than $\frac{\sigma_1}{100}$ (where $\sigma_1$ denotes the largest singular value), we observe for the second data set (with $M=20$ and $L=5$) that already $11$ values are below this bound. For support size $L=7$ we have  $45$ values smaller than $\frac{\sigma_1}{100}$, and for $L=9$, even $139$ values. 

\begin{figure}[h!]
\begin{center}
        \includegraphics[width=35mm,height=30mm]{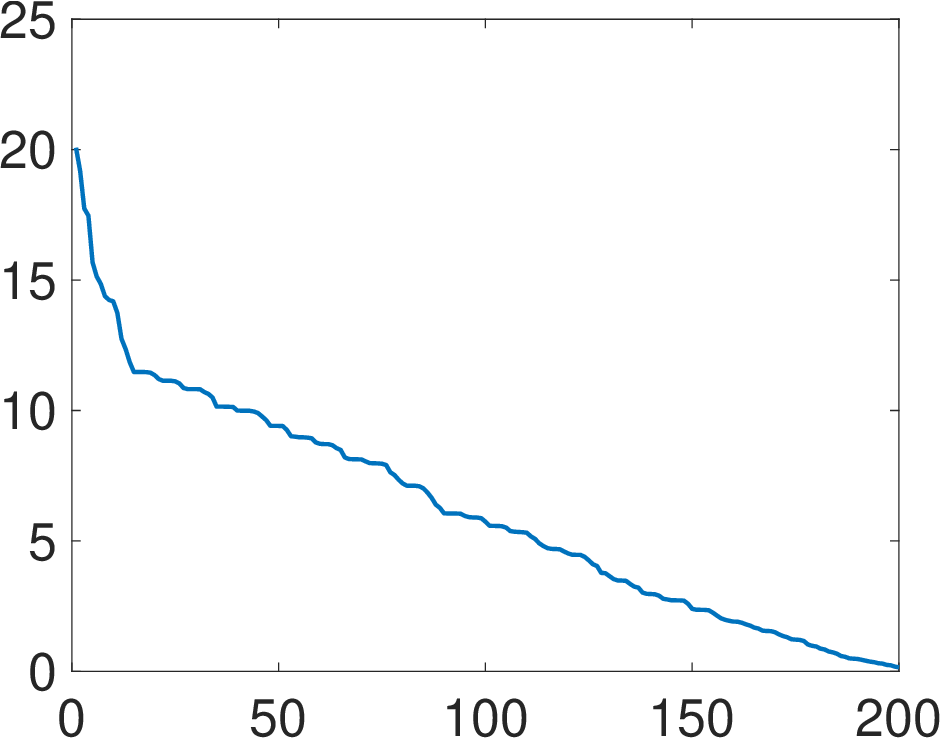}
        \includegraphics[width=35mm,height=30mm]{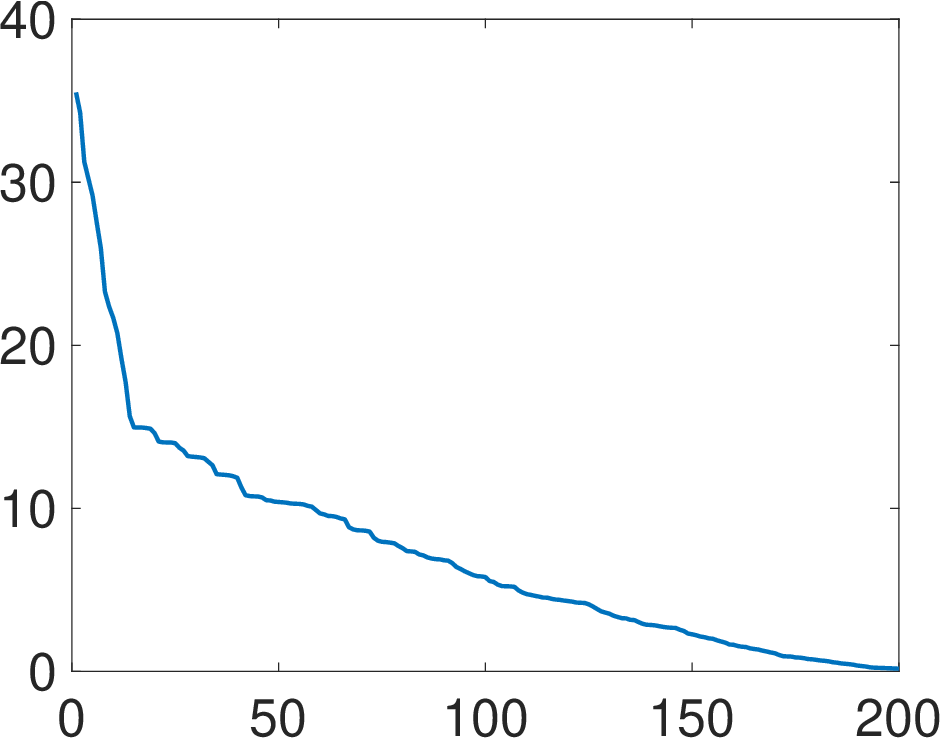}
         \includegraphics[width=35mm,height=30mm]{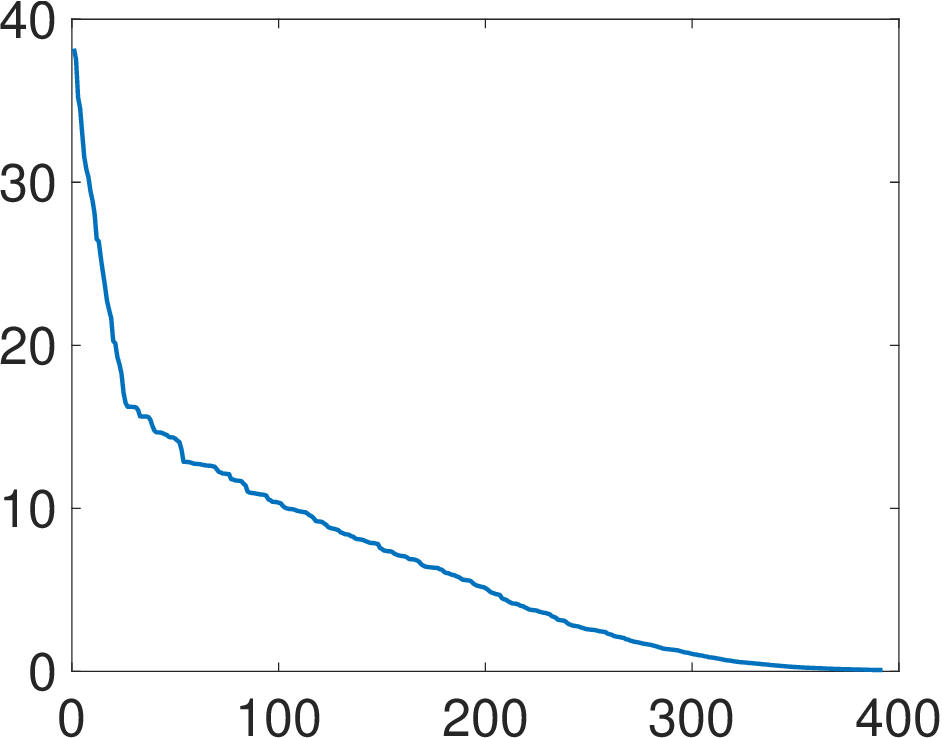}
	\includegraphics[width=35mm,height=30mm]{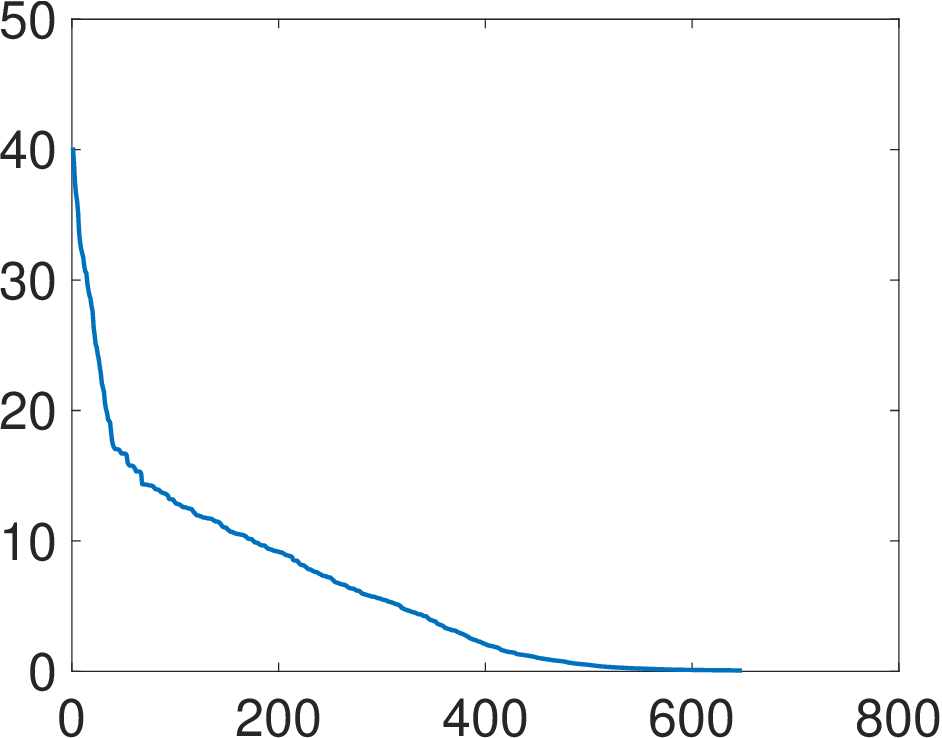}
\end{center}	
\caption{Illustration of the singular values of the matrix ${\mathbf A}_M$ for  the first and the second data set.
From left to right: (1) Singular values of ${\mathbf A}_M$ with $L=5$ and $25\cdot 8=200$ columns for the first data set. 
(2) Singular values of ${\mathbf A}_M$  with  $L=5$ and $25\cdot 8=200$ columns, (3) with $L=7$ and $49\cdot8= 392$ columns, and (4) with 
$L=9$ and $81\cdot 8=648$ columns 
for the second data set.
 }
\label{figsing1}
\end{figure}

Therefore, to compute the coefficient vector ${\mathbf c}$ (see step 2 of Algorithm \ref{alg2}), we propose to take a suitable linear combination of several singular vectors corresponding to the smallest singular values of ${\mathbf A}_M$. 
More exactly, for reconstruction from the second data set, we employ a linear combination of the form 
\begin{align}\label{cc} \textstyle {\mathbf c} = \sum\limits_{\nu=1}^{N_s} \alpha_{\nu} {\mathbf c}_{\nu}
\end{align}
with $N_s \ge 1$, $\alpha_{\nu} \in {\mathbb C}$ and with ${\mathbf c}_{\nu}$ being a singular vector of ${\mathbf A}_M$ corresponding to its $\nu$-th smallest singular value. 
Our numerical experiments have shown that this linear combination should be taken such that the final vector 
${\mathbf c} = ({\mathbf c}^{(j)})_{j=0}^{N_c-1}$ leads  via step 3 of Algorithm \ref{alg2} to sensitivities ${\mathbf s}^{(j)}$, 
where the matrix ${\mathbf d}= (d_{\nn})_{\nn \in \Lambda_N} = \sum_{j=0}^{N_c-1} \overline{\mathbf s}^{(j)} \circ  {\mathbf s}^{(j)}$ 
does not have very small entries. In other words, one should exploit the freedom to choose ${\mathbf c}= ({\mathbf c}^{(j)})_{j=0}^{N_c-1}$ in (\ref{cc}) to achieve sensitivities ${\mathbf s}^{(j)}$, which are already close to satisfying  the 
sos condition, such that all components $d_{\nn}$, $\nn \in \Lambda_N$ are of similar size.

 We have implemented the following method to compute ${\boldsymbol \alpha} = (\alpha_{\nu})_{\nu=1}^{N_s}$ determining ${\mathbf c}$ in (\ref{cc}). Let ${\mathbf A}_M= {\mathbf U} {\mathbf \Sigma} {\mathbf V}$ denote the SVD of the MOCCA-matrix ${\mathbf A}_M$ in (\ref{AM}), which is computed in step 2 of Algorithm \ref{alg2}.
 Further, let ${\mathbf V}_{L^2N_c,N_s}$ be the partial matrix of ${\mathbf V}$  containing the last 
 $N_s$ eigenvectors of ${\mathbf A}_{M}$ and  ${\mathbf w} := ({\mathbf e}_0^T,
 {\mathbf e}_0^T, \ldots ,{\mathbf e}_0^T)^T \in {\mathbb C}^{L^2N_c}$, with ${\mathbf e}_0 =(\delta_{0,k})_{k=-(L^2-1)/2}^{(L^2-1)/2}$ and the Kronecker symbol $\delta_{0,k} := 0$ for $k \neq 0$ and $\delta_{0,0} := 1$.
Then we take ${\boldsymbol \alpha}= {\mathbf V}_{L^2N_c,N_s}^* {\mathbf w}$.
The question of how to take $N_s$ suitably  and the problem of finding  better methods to choose $\alpha_{\nu}$ is still open und requires further investigations.

In Table \ref{tab:2} we have used this approach for $L=5$ with $N_s=15$ for $R=4$, $(\frac{1}{2},\frac{1}{2})$, $R=6$, and $N_s=30$ otherwise, for $L=7$ with $N_s=47$,
and for $L=9$ with  $N_s=45$ 
to achieve the presented results. 
We observe very high SSIM values for the MOCCA algorithm, which are particularly much higher than for ESPIRiT. This may be due to the fact that the errors occurring for the MOCCA reconstructions are very small also outside the boundary of the brain.

\begin{table}[htbp]
\scriptsize
\caption{Comparison of the reconstruction performance for the incomplete data from the second data set.}\label{tab:2}
\begin{center}
\begin{tabular}{rrrrrrr}
\hline
method & measure & $R=2, (1,\frac{1}{2})$ & $R=3, (1,\frac{1}{3})$  & $R=4, (1,\frac{1}{4})$ & $R=4, (\frac{1}{2},\frac{1}{2})$ & $R=6, (\frac{1}{2},\frac{1}{3})$ \\
\hline
GRAPPA & PSNR &  {47.2028} & 42.9704 & 37.6036 & 41.5961 &38.5314\\
 & SSIM &  {0.9796} & {0.9573}  & 0.9167  & 0.9498 & 0.9244\\
 \hline
ESPIRiT & PSNR & 40.3668 &39.7082 & 37.1418 & 39.6082 & 38.4716\\
                & SSIM & 0.7490 & 0.7462 & 0.7243 & 0.7461 & 0.7354 \\
                \hline
L1-ESPIRiT & PSNR 
& 40.2750 & 39.7584 & 38.1926 & 39.6515 & 38.8431 \\
 & SSIM  
  & 0.7465 & 0.7520  &0.7455 & 0.7521 & 0.7503 \\
 \hline
MOCCA & PSNR              & (10)42.6611   & (50)40.3835 & (70)34.8179 & (50)40.6865 & (90)35.9914 \\
 
 (L=5)& SSIM             & 0.9258          & 0.9474       & 0.8769          & 0.9472               & 0.8752 \\
     
 MOCCA-S & PSNR           & (10)43.2025  & (50)41.7574 & (70)35.8326 & (50)42.2273 & (90)37.5114 \\
 
 (L=5)  & SSIM                   & 0.9314        & 0.9675            & 0.9192         & 0.9679        & 0.9239   \\
 \hline
 MOCCA & PSNR      & (10)43.5172 & (50)41.9112           & (70)37.3703 & (50)41.7795 &(90)37.4211\\
 (L=7)     &  SSIM.      & 0.9251        &  0.9461                    & 0.9143          & 0.9534          & 0.9011\\
  MOCCA-S & PSNR & (10)44.0399 & (50){43.5151} & (70){38.9200} & (50){43.5679} &(90){39.3194}\\
 (L=7)& SSIM             & 0.9299         &{0.9630}        &  0.9537         & {0.9729}          &{0.9463}\\
 \hline
 MOCCA & PSNR      & (10)43.2813 & (50)41.7241 & (70)38.0292 & (50)41.2087& (90)37.9957 \\
 (L=9)& SSIM             & 0.9199        & 0.9437          & 0.9214         & 0.9489         & 0.9127\\
 MOCCA-S  & PSNR & (10)43.7645 &(50){43.1347}  & (70)\textbf{39.7013} & (50){42.6955} & (90)\textbf{40.0862}\\
 (L=9)& SSIM            &  0.9245          & 0.9597           & \textbf{0.9579}        & {0.9675}          & \textbf{0.9571}\\
 \hline
  MOCCA & PSNR      & 47.4233               & 42.2531                   & 35.8053       & 42.1402           &34.4736\\
 direct (L=7)     &  SSIM.      & 0.9844        &  0.9615                    &0.8838          & 0.9616              & 0.8293\\
  MOCCA-S & PSNR & \textbf{47.9400} & \textbf{43.8822}         & 38.0142       & \textbf{43.9457} &36.4445\\
direct (L=7)& SSIM    & \textbf{0.9866}         &\textbf{0.9764}        &  0.9374       & \textbf{0.9792}          &{0.8954}\\
 \hline
\end{tabular}
\end{center}
\end{table}

In Figure \ref{figure_Bild2}, we exemplarily represent the reconstruction results for $R=4$ (every fourth column acquired) and the corresponding error maps for the second data set. 
The  MOCCA-S  reconstructions for $L=9$ and $L=11$ in Figure \ref{figure_Bild2} do  not contain obviously visible aliasing artifacts. The GRAPPA and ESPIRiT reconstructions in Figure \ref{figure_Bild2} contain slightly visible alising artifacts.
Note that a high PSNR value not always implies a very good visual reconstruction result  since pointwise large errors are not strongly punished by this measure.

\begin{figure}[h!]
\begin{center}
 \includegraphics[width=28mm,height=28mm]{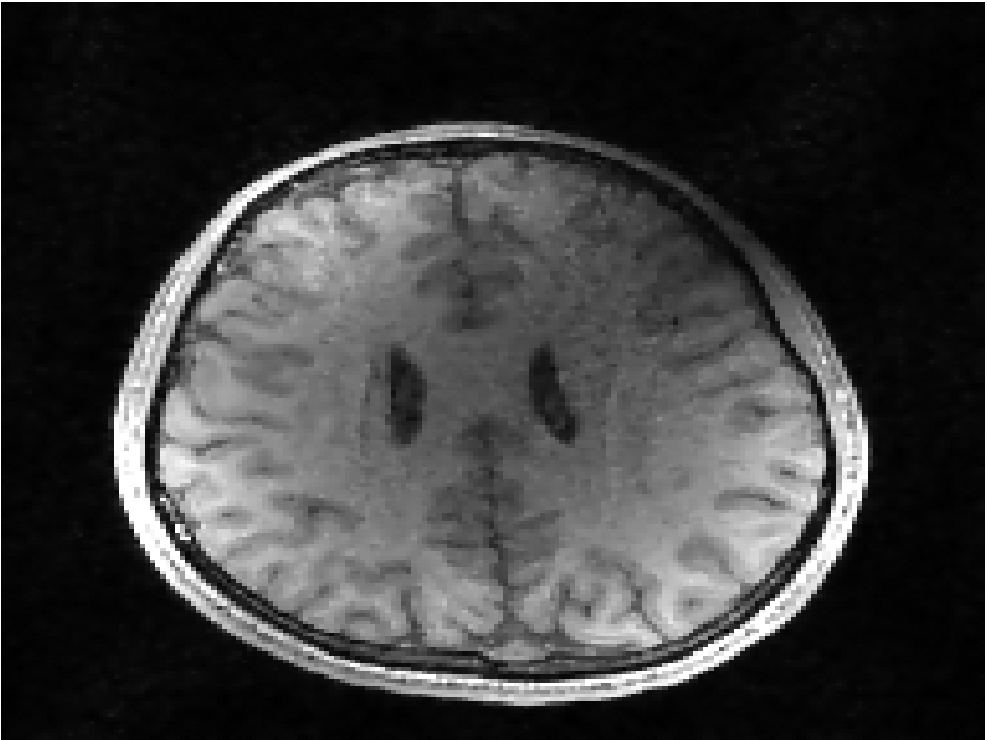}
  \includegraphics[width=28mm,height=28mm]{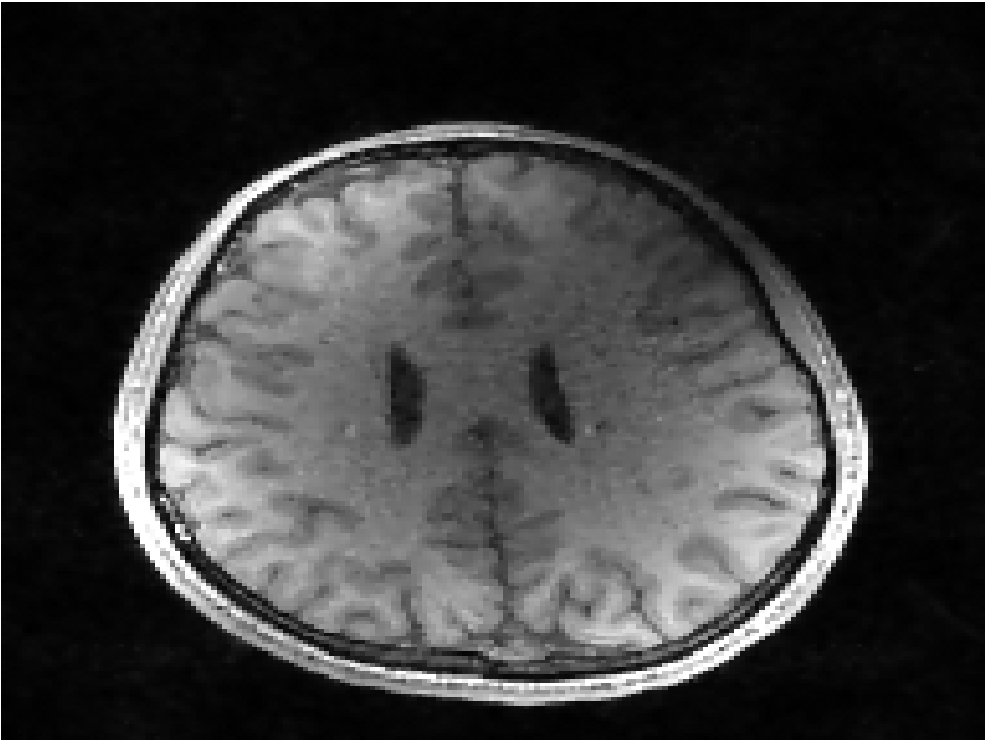}
  \includegraphics[width=28mm,height=28mm]{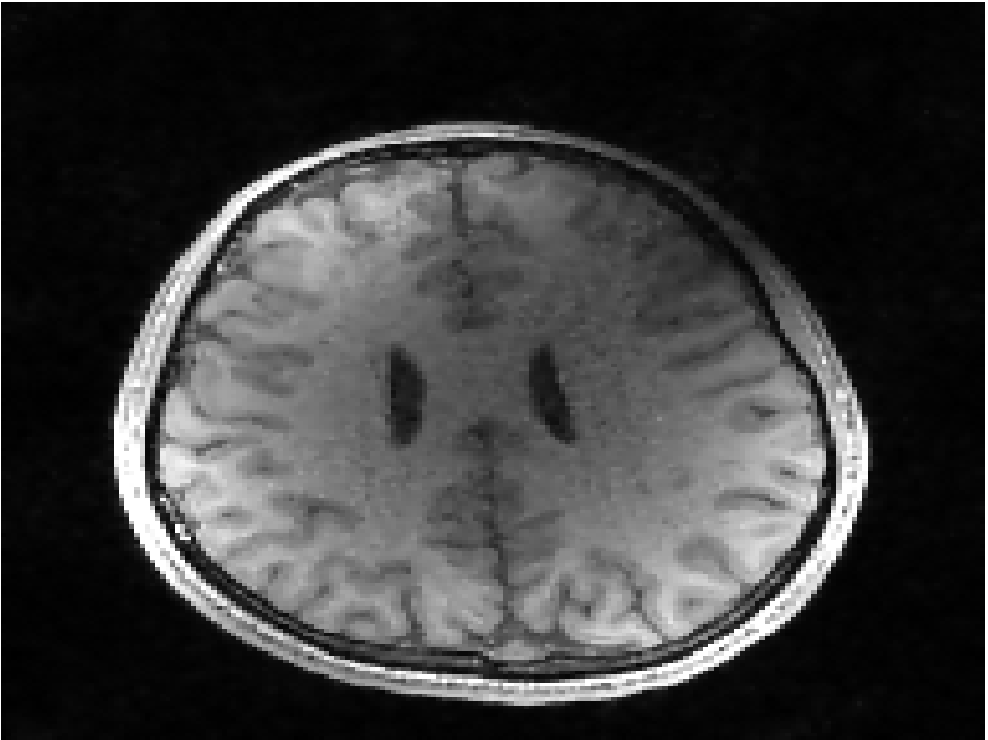}
 \includegraphics[width=28mm,height=28mm]{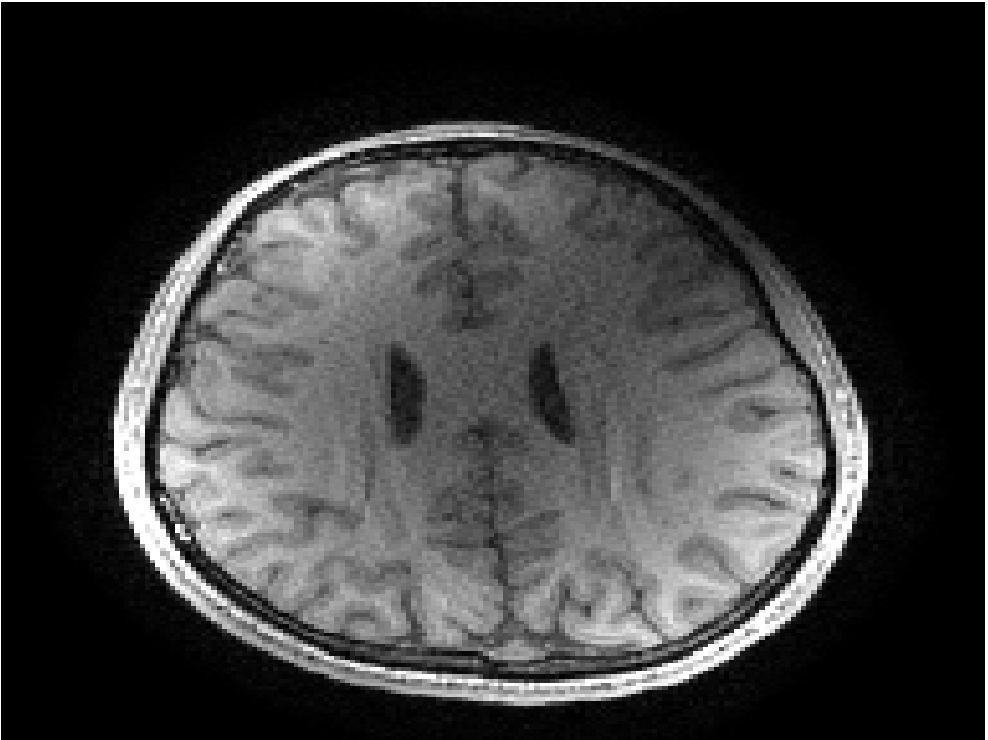}
 \includegraphics[width=28mm,height=28mm]{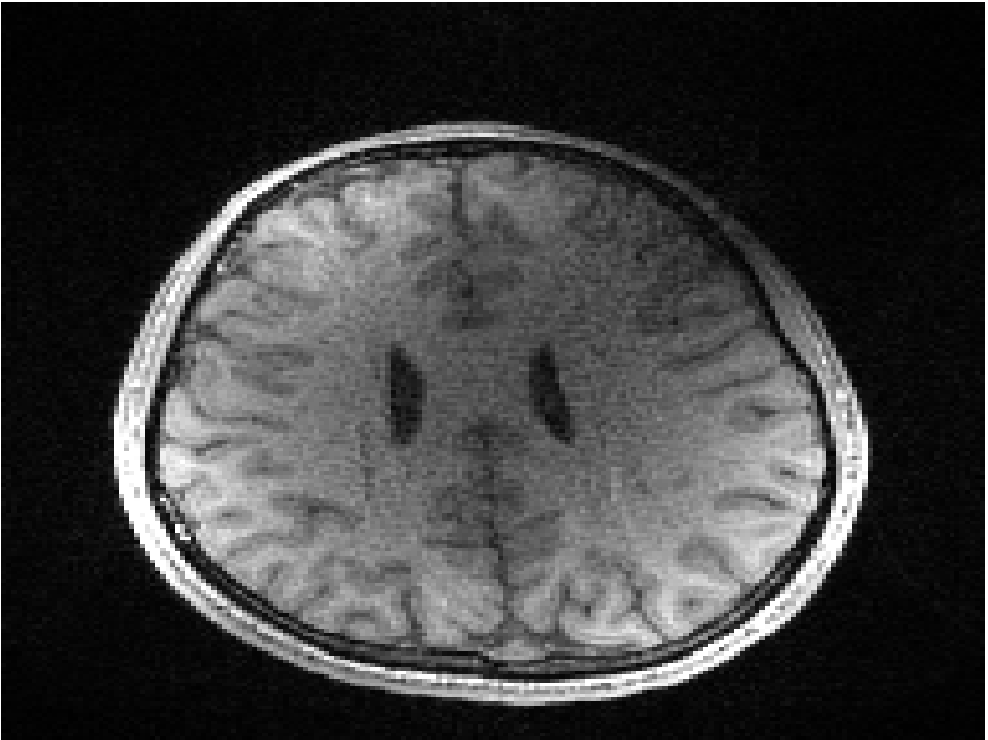}
 
 \includegraphics[width=28mm,height=28mm]{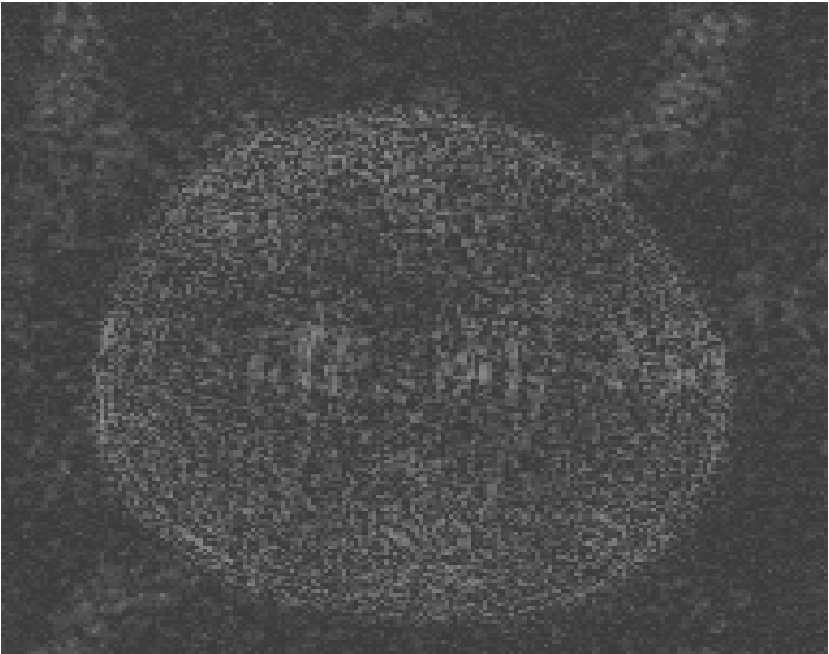}
  \includegraphics[width=28mm,height=28mm]{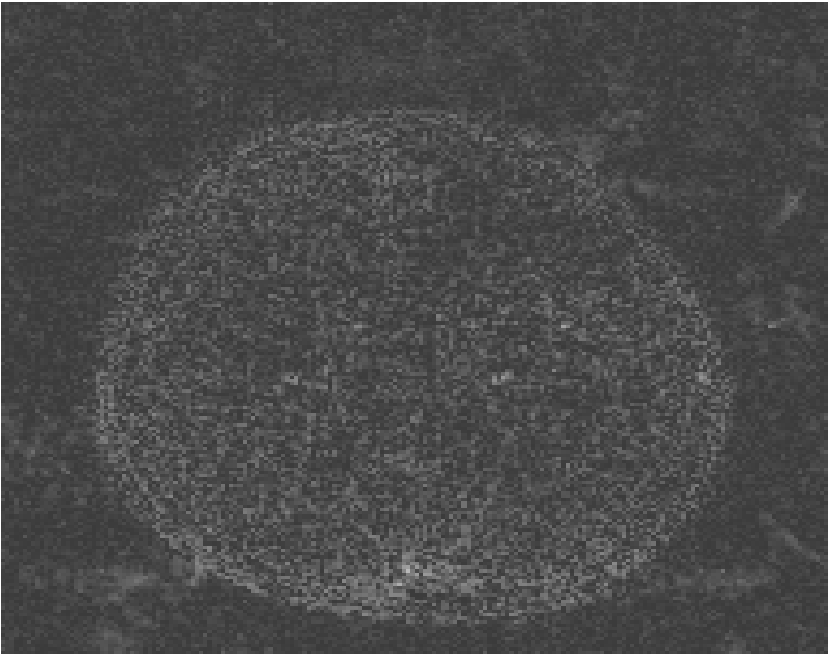}
  \includegraphics[width=28mm,height=28mm]{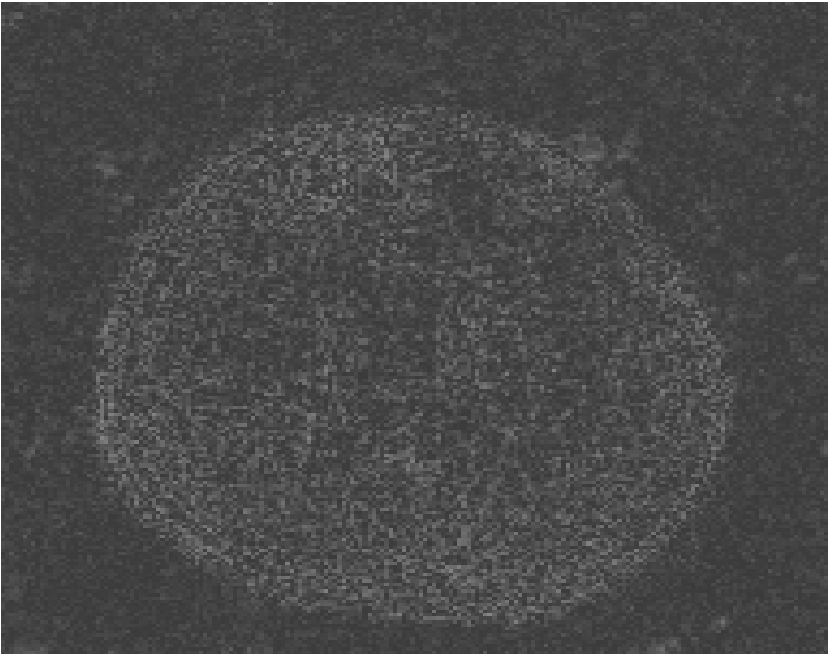}
	\includegraphics[width=28mm,height=28mm]{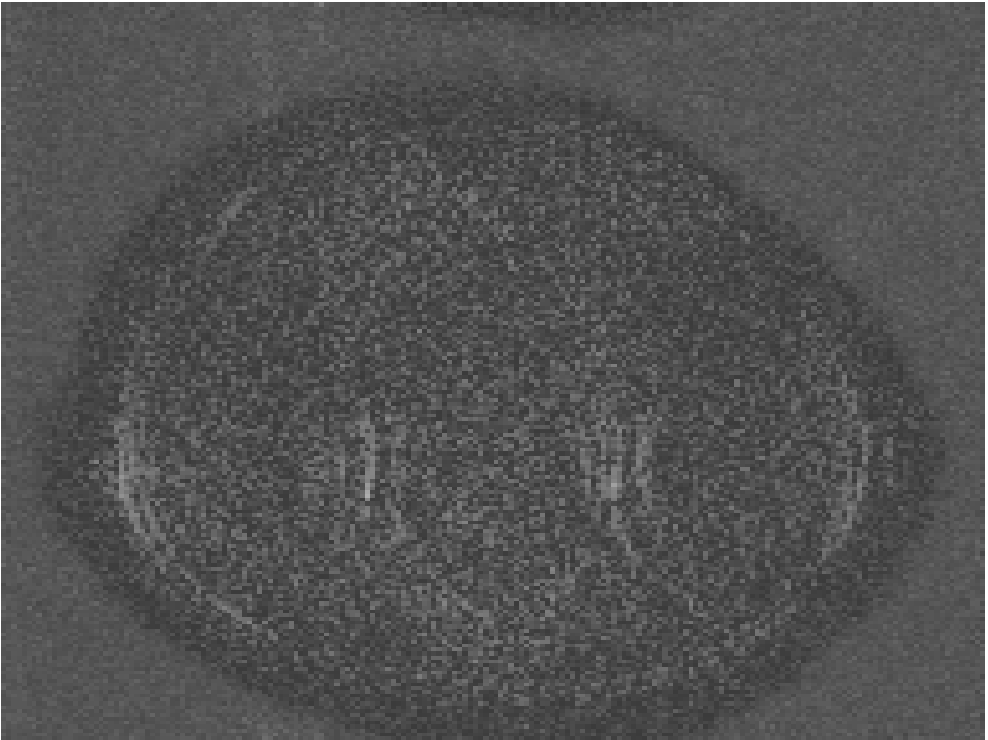}
	\includegraphics[width=28mm,height=28mm]{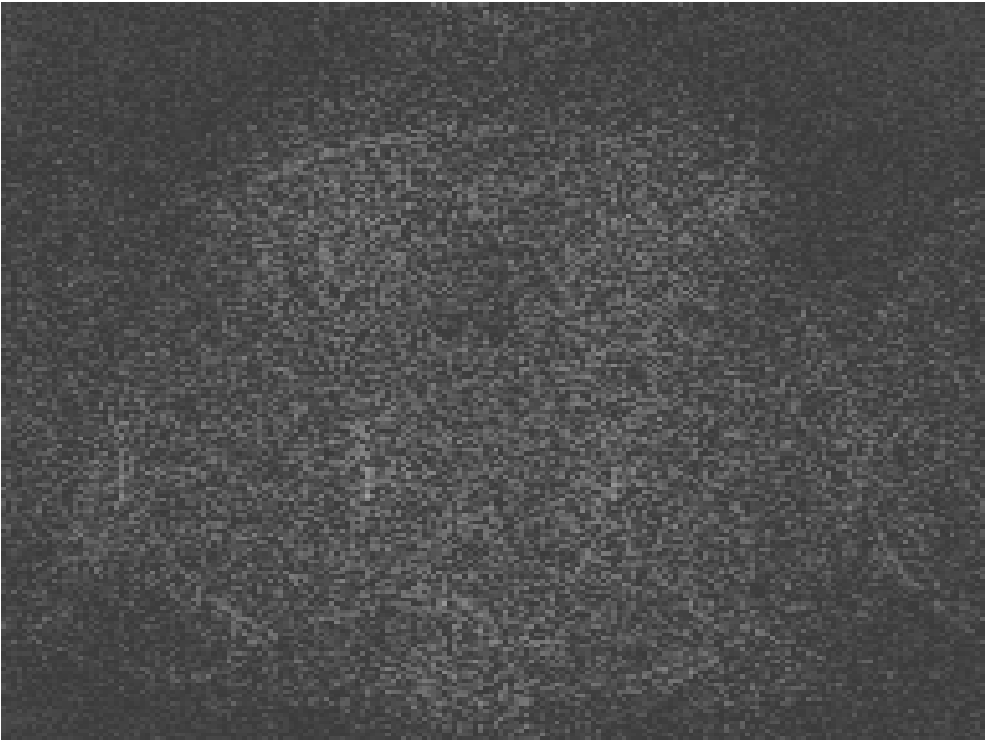}

	\end{center}
\caption{Reconstruction results obtained from a fourth of the $k$-space data of the second data set with $8$ coils (every fourth column acquired) for MOCCA-S, ESPIRiT, and GRAPPA. From left to right: 
(1)  MOCCA-S with $L=7$,   (2) MOCCA-S with $L=9$,  and (3) MOCCA-S with $L=11$,  all using Algorithm \ref{alg2} with Algorithm \ref{alg3} and smoothing with $\lambda=0.002$, (4) ESPIRiT, (5) GRAPPA.
Corresponding  error maps are given  below, where darker means smaller error. All error images use the same scale with relative error in $[0,0.14]$, where $0$ corresponds to black and $0.14$ to white.}
\label{figure_Bild2}
\end{figure}

Finally, in  Figure \ref{fig:coil2}, we illustrate the magnitude and the phase of the 8 normalized coil sensitivities $\tilde{\mathbf s}^{(j)}$ obtained for the second data set with $L=9$.  Since these  sensitivities  are constructed from bivariate trigonometric polynomials of higher degree (with $81$ nonzero coefficients) they possess a more oscillatory behavior in magnitude and phase than the sensitivities for the first data set for $L=5$. 

\begin{figure}[h!]
\begin{center}
	  \includegraphics[width=65mm,height=26mm]{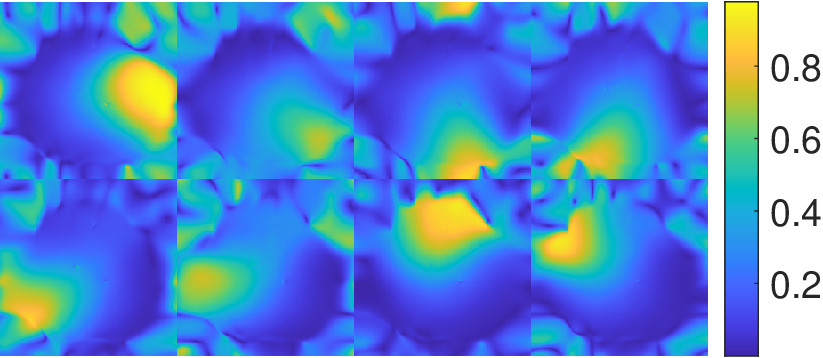} \hspace{3mm}
	\includegraphics[width=65mm,height=26mm]{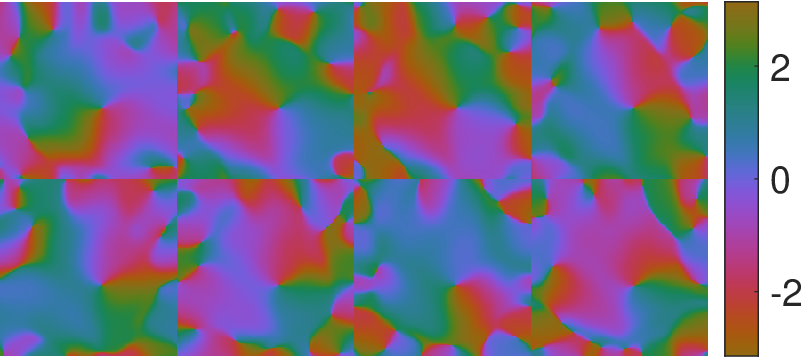}
		  \includegraphics[width=65mm,height=26mm]{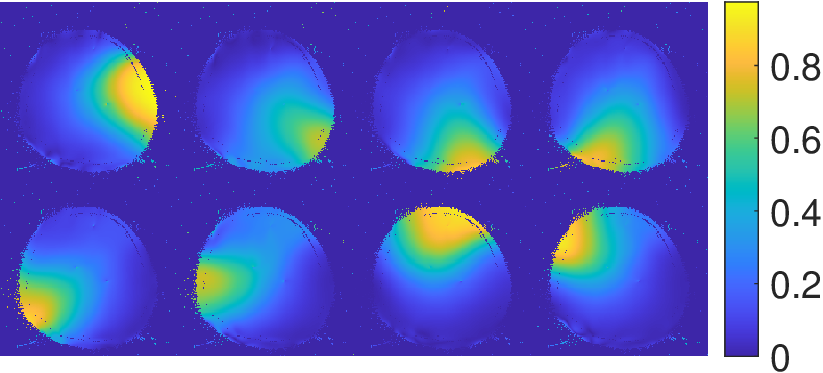} \hspace{3mm}
	\includegraphics[width=65mm,height=26mm]{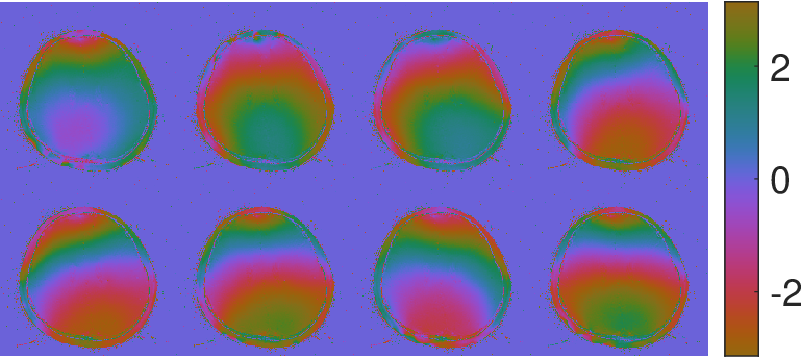}
\end{center}
\caption{Magnitude (left) and phase $($in $[-\pi,\pi])$ (right) of the $8$ coil sensitivities obtained for MRI reconstruction for $L=9$ for the second data set before and after multiplication with $\sign(\mm)$, see step 7 of Algorithm \ref{alg2}. The  normalized sensitivities are samples of bivariate trigonometric polynomials with $L^2=81$ nonzero coefficients, pointwisely multiplied with the normalization factors $({\mathbf d}^{+})^{\frac{1}{2}}$ to ensure the sos condition. }
\label{fig:coil2}
\end{figure}

The  implementation of the MOCCA algorithms have been performed  in Matlab. The presentation of 
coils  in Figures \ref{fig:coil1} and \ref{fig:coil2} uses  {\tt imshow3.m} by M. Lustig and  {\tt phasemap.m} by 
C. Greene. 
The code is available at  \url{https://na.math.uni-goettingen.de}.

\section{Conclusions}
Previous sub-space methods in parallel MRI  \cite{SAKE,ESPIRIT,PLORAKS,ALOHA} are usually based on the assumption that the block Hankel matrix $({\mathbf Y}_{N,L}^{(0)}, {\mathbf Y}_{N,L}^{(1)}, \ldots , {\mathbf Y}_{N,L}^{(N_c-1)}) \in {\mathbb C}^{N^2 \times L^2 N_c}$, where ${\mathbf Y}_{N,L}^{(j)} = (y_{(\nuu-\rr) \mod \Lambda_N}^{(j)})_{\nuu \in \Lambda_N, \rr \in \Lambda_L}$ with $L$ being a suitable small window size, has low rank. More exactly, it is assumed that the rank of this matrix is essentially smaller than $L^2 N_c$. The application of these methods based on (structured) low-rank matrix approximations and the study the reconstructed sensitivities ${\mathbf s}^{(j)}$ leads to the observation that the ${\mathbf s}^{(j)}$  usually have (approximately) a small support in $k$-space, see e.g. \cite{SAKE,ESPIRIT}. Our new MOCCA approach  directly computes  the sensitivities and the magnetization image based on the a priori fixed model (\ref{sjk}) (resp.\ (\ref{models1})) and can therefore be seen as a counterpart of the known algorithms. Our different view provides several advantages: We can provide simple and fast reconstruction algorithms for  incomplete k-space data in parallel MRI achieving similarly good reconstruction results as the best sub-space methods. 

One question, which still stays to be open regarding the MOCCA approach is to determine a suitable support index set $\Lambda_L$ for the coil sensitivities in practice. A larger support index set may lead to many singular values of the MOCCA matrix ${\mathbf A}_M$ in (\ref{AM}) being close to zero. In this case, the choice of a suitable  linear combination of singular vectors is crucial to achieve satisfying sensitivities and image reconstructions. 
In this paper, we have exemplarily shown that the MOCCA algorithm can outperform several other methods, while the problem of finding an optimal linear combination of singular vectors in case of overestimated support set is still under investigation. 

There is a close connection between the MOCCA approach and the subspace methods ESPIRiT and SAKE which we study in a forthcoming paper \cite{KP24}. 
A better understanding of the relation between subspace methods and sensitivity modelling will help us to answer the question of optimal recovering  of sensitivity profiles with small $k$-space support and  to find other sensitivity models being appropriate for parallel MRI reconstructions thereby still allowing fast reconstruction procedures.

\section*{Acknowledgements}
G. Plonka and Y. Riebe acknowledge support  by the Deutsche Forschungsgemeinschaft in the CRC 1456, project B03 and by the 
 European Union’s Horizon 2020 research and innovation programme under the Marie Skłodowska-Curie grant agreement No 101008231. Yannick Riebe acknowledges support by the Deutsche Forschungsgemeinschaft in the RTG 2088.
 The authors thank Benjamin Kocurov and Martin Uecker for helpful discussions.  

 \bibliographystyle{siam}
\bibliography{references}
\end{document}